\documentclass[10pt]{amsart}
\usepackage{amssymb,amsmath,mathtools}
\usepackage{enumerate}
\usepackage{xcolor}
\usepackage{mathrsfs}
\usepackage{enumitem}
\usepackage{graphicx}
\usepackage{hyperref}
\usepackage{float}

\usepackage{setspace}

\usepackage[left=1.25in,
            right=1.25in,
            top=1.52in,
            bottom=1.52in,centering]{geometry}
\newtheorem{theorem}{Theorem}[section]
\newtheorem{proposition}[theorem]{Proposition}
\newtheorem{lemma}[theorem]{Lemma}
\newtheorem{corollary}[theorem]{Corollary}
\theoremstyle{definition}
\newtheorem{definition}[theorem]{Definition}
\newtheorem{example}[theorem]{Example}
\newtheorem{remark}[theorem]{Remark}

\newcommand{\F}{\mathcal F}
\newcommand{\AW}{\mathcal A\mathcal W}
\newcommand{\Law}{\mathscr L}

\newcommand{\CFP}{\mathrm{CFP}}

\newcommand{\ip}{\mathrm{ip}}
\newcommand{\id}{\mathrm{id}}
\newcommand{\pp}{\mathrm{pp}}

\newcommand{\fp}[1]{{\mathbb #1}}
\newcommand{\cfp}[1]{{\overline{\mathbb #1}}}

\newcommand{\W}{\mathcal W}
\newcommand{\N}{\mathbb N}
\newcommand{\R}{\mathbb R}
\newcommand{\X}{\mathcal X}
\newcommand{\Y}{\mathcal Y}
\newcommand{\Z}{\mathcal Z}
\newcommand{\A}{\mathcal A}
\newcommand{\B}{\mathcal B}
\newcommand{\Pc}{\mathcal P}
\newcommand{\proj}{\mathrm{pj}}

\newcommand{\cpl}{\mathrm{Cpl}}
\newcommand{\cpla}{\cpl_{\mathrm{c}}}
\newcommand{\cplba}{\cpl_{\mathrm{bc}}}
\newcommand{\cplbc}{\mathrm{Cpl}_{\mathrm{bc}}}

\newcommand{\unfold}{\mathrm{uf}}
\newcommand{\AF}{\mathrm{AF}}
\newcommand{\depth}{\mathrm{depth}}

\def\P{{\mathbb P}}
\def\E{{\mathbb E}}

\def\PP{{\mathcal P}}
\def\Q{{\mathbb Q}}
\def\R{{\mathbb R}}

\newcommand{\FP}{\mathcal{FP}}

\newcommand{\FFP}{\mathrm{FP}}
\newcommand{\WA}{\mathcal W}

\newcommand{\AWA}{\mathcal {AW}}

\usepackage{relsize}
\def\fcmp{\mathbin{\raise 0.6ex\hbox{\oalign{\hfil$\scriptscriptstyle \mathrm{o}$\hfil\cr\hfil$\scriptscriptstyle\mathrm{9}$\hfil}}}}

\numberwithin{equation}{section}

\author{Daniel Bartl \and
Mathias Beiglb\"ock \and 
Gudmund Pammer}
\thanks{The authors are grateful to Julio Backhoff-Veraguas, Manu Eder, Daniel Lacker, Marcel Nutz, and Stefan Schrott for many helpful comments. We are also indebted to Stephan Eckstein for generous support through numerical implementations and experiments.}

\date{\today}

\begin{document}

\title{The Wasserstein space of stochastic processes}

\begin{abstract}

Wasserstein distance induces a natural Riemannian structure for the probabilities on the Euclidean space. This insight of classical transport theory is fundamental for tremendous applications in  various fields of pure and applied mathematics.

We believe that 
an appropriate probabilistic variant, the \emph{adapted Wasserstein distance $\AW$},   can play a similar role for the class $\FFP$ of \emph{filtered processes}, i.e.\ stochastic processes together with a filtration. 
In contrast to other topologies for stochastic processes, probabilistic operations such as the Doob-decomposition, optimal stopping and stochastic control are continuous w.r.t.\ $\AW$.  We also show that  $(\FFP, \AW)$  is a geodesic space, isometric to a classical Wasserstein space, and that  martingales form a closed geodesically convex subspace.
\end{abstract}

\maketitle

\section{Overview}

It is often useful to change the view from considering objects in isolation to  studying the space of these objects and specifically their mutual relationship w.r.t.\ the \emph{ambient space}.
A classic instance  is to
 switch from studying functions to considering the  Lebesgue- or Sobolev-spaces
of functions in functional analysis; a more contemporary example is the 
passing from 
measures 
to the manifold of measures based on  optimal transport theory.
The aim of this article is to investigate what should be the appropriate  ambient space for the class of stochastic processes.

\subsection{The space of laws on $\R^N$ and its limitations}

A natural starting point for this is to represent stochastic processes on the  canonical probability space. Specifically, the class of real valued processes in finite discrete time $\{1, \ldots, N\}$  is naturally represented as the set $\PP(\R^N)$ of probability measures on $\R^N$. 
Importantly,  the usual weak topology on $\PP(\R^N)$  fails  to capture the temporal structure of stochastic processes and is not strong enough to guarantee continuity of stochastic optimization problems  or basic operations like the Doob-decomposition.

As a remedy, a number of researchers from different scientific communities have introduced topological structures on the set of stochastic processes
with the common goal to adequately capture the temporal structure.
We list Aldous extended weak topology \cite{Al81} (stochastic analysis), Hellwig's information topology \cite{He96, HeSc02} (economics), Bion-Nadal and Talay's version of the Wasserstein distance \cite{BiTa19} (stochastic analysis, optimal control), Pflug and Pichler's nested distance \cite{PfPi12,PfPi14} (stochastic optimization), R\"uschendorf's Markov constructions \cite{Ru85} (optimal transport), Lassalle's causal transport problem \cite{La18} (optimal transport), and Nielsen and Sun's chain rule transport \cite{NiSu20} (machine learning). Remarkably,  in finite discrete time, these seemingly independent approaches define the same topology on $\PP (\R^N)$, the \emph{weak adapted topology}, see \cite{BaBaBeEd19b}. 

A natural compatible metric\footnote{Precisely, convergence in $\AWA_p$ is equivalent to convergence in the weak adapted topology plus convergence of the $p$-th moment, see also \cite{BaBaBeEd19b}.} for the weak adapted topology is the \emph{adapted Wasserstein distance} 
$$\AWA_p^p(\mu, \nu)=\inf_{\pi\in \cplbc(\mu, \nu)} \E_\pi[\|X-Y\|_p^p], \quad p\in [1,\infty).$$ 
The difference to the classical Wasserstein distance comes from the fact that one considers only \emph{bicausal} couplings $\cplbc(\mu, \nu)$, see Definition \ref{def:causal} below. 
These couplings are non-anticipative and can be viewed as a Kantorovich analogue of non-anticipative transport maps, see \cite{BePaSc21c}. 

While the weak adapted topology / $\AWA_p$ appear canonical  and have recently seen a burst of applications (see \cite{Do14, PfPi15, PfPi16, GlPfPi17, BaBaBeEd19a, AcBePa20, Wi20, PiSc20, KiPfPi20, Mo20, PiWe21, PiSh21, Wi21,  BaBaBeWi20, JoWi21} among others) we also highlight two limitations: 
\begin{enumerate}
\item The metric space $(\PP_p(\R^N), \AWA_p)$ is not complete.
In fact, this shortcoming also arises for other natural  distances that respect the information structure of stochastic processes.  
\item Following the classical theory of stochastic analysis one would like to consider processes together with a general filtration, not just the filtration generated by the process itself. 

\end{enumerate}

\subsection{Filtered processes as the completion of $\mathcal{P}_p(\mathbb{R}^N)$}
Rather conveniently, these supposed shortcomings already represent their mutual resolution: 
a possible interpretation of the incompleteness of 
$(\PP_p(\R^N), \AWA_p)$ is that the space $\PP_p(\R^N)$ is not `large'  enough to represent all processes one would like to consider. In our first main result we show that  the extra information that can be  stored in an ambient filtration is precisely what is needed to arrive at the completion of $(\PP_p(\R^N), \AWA_p)$.
To make this precise we need the following definition:

\begin{definition}\label{FPDef}
A five-tuple
\begin{align}
\label{eq:def.fp}
 \fp{X}:= \big(\Omega, \F, \P, (\F_t)_{t=1}^N, (X_t)_{t=1}^N\big),
 \end{align}
where $(X_t)_{t=1}^N$ is adapted to $(\F_t)_{t=1}^N$, is called a \emph{filtered (stochastic) process}.
We write $\FP$ for the class of all filtered processes and $\FP_p$ for the subclass of processes with $\E[\|X\|_p^p]<\infty$.
\end{definition}

Clearly,  $\mathcal{P}_p(\R^N)$ is embedded in $\FP_p$: for  $\mu\in\mathcal{P}(\R^N)$ set $\fp X:= ( \R^N, \mathcal{B}(\R^N), \mu, (\F_t)_{t=1}^N,(X_t)_{t=1}^N)$, where $(X_t)_{t=1}^N$ is the canonical process, i.e.\ $X_t(\omega)=\omega(t)$, and $\F_t=\sigma(X_s:s\leq t)$ for $t\leq N$.

It is relatively straightforward  to extend  the concept of bicausal couplings to processes with filtrations (see Definition \ref{def:causal} for details) and accordingly the notion of adapted Wasserstein distance extends to filtered processes via
\[\AWA_p^p(\fp{X},\fp{Y}):=\inf_{\pi\in \cplbc(\fp{X}, \fp{Y})} \E_\pi[ \|X-Y\|_p^p].\]
As in the case of  ${\mathscr{L}}^p$ / $ L^p$ spaces (or similar situations), we identify two filtered processes $\fp{X},\fp{Y}$ if $\AWA_p(\fp{X},\fp{Y})=0$ and denote the corresponding set of equivalence classes by $\FFP_p$.
Our first main result is:

\begin{theorem}\label{MainTheorem}
$\AWA_p$ is a metric on $\FFP_p$ and  $(\FFP_p, \AWA_p)$ is the completion of $(\mathcal P_p(\R^N), \AWA_p)$. 
\end{theorem}

We also show in Theorem \ref{thm:dense} below that certain simpler classes of processes are dense in $\FFP_p$, e.g.\ filtered processes that can be represented on a finite state space $\Omega$ or finite state Markov chains.   
This seems important in view of numerical applications. 

\subsection{$\FFP_p$ as geodesic space}

In fact, our proof of Theorem \ref{MainTheorem}  reveals more, namely the following result on the metric structure of $(\FFP_p, \AWA_p)$:

\begin{theorem}\label{WREP} 
There exists a Polish space $(\mathcal{V}, d)$ such that $(\FFP_p, \AWA_p)$ is isometric to the  classical Wasserstein space $(\mathcal P_p (\mathcal{V}), \WA_p^{(d)})$.
\end{theorem}
Explicitly, $\mathcal{V}$ is constructed by considering Wasserstein spaces of Wasserstein spaces in an iterated fashion as already considered by Vershik \cite{Ve70}. 
Theorem \ref{WREP} allows to transfer concepts from optimal transport to the theory of stochastic processes, e.g.\ it allows to consider displacement interpolation and Wasserstein barycenters of filtered processes and to view $(\FFP_p, \AWA_p)$ as a (formal) Riemannian manifold.
Specifically, using  the work of Lisini \cite{Li07} we obtain:

\begin{theorem} \label{geodesicspace}
	Assume $p>1$. Then $(\FFP_p, \AWA_p)$ is a geodesic space and the set of martingales forms a closed, geodesically convex subspace.
\end{theorem}

Famously, McCann introduced the concept of displacement interpolation in his thesis \cite{Mc94}, giving a new meaning to the transformation of one probability into another. In analogy, Theorem \ref{geodesicspace} suggests an interpolation between stochastic processes.

 \vspace{-40mm}
 \begin{figure}[H]
    \centering
    \includegraphics[page=1,width=0.7\textwidth]
        {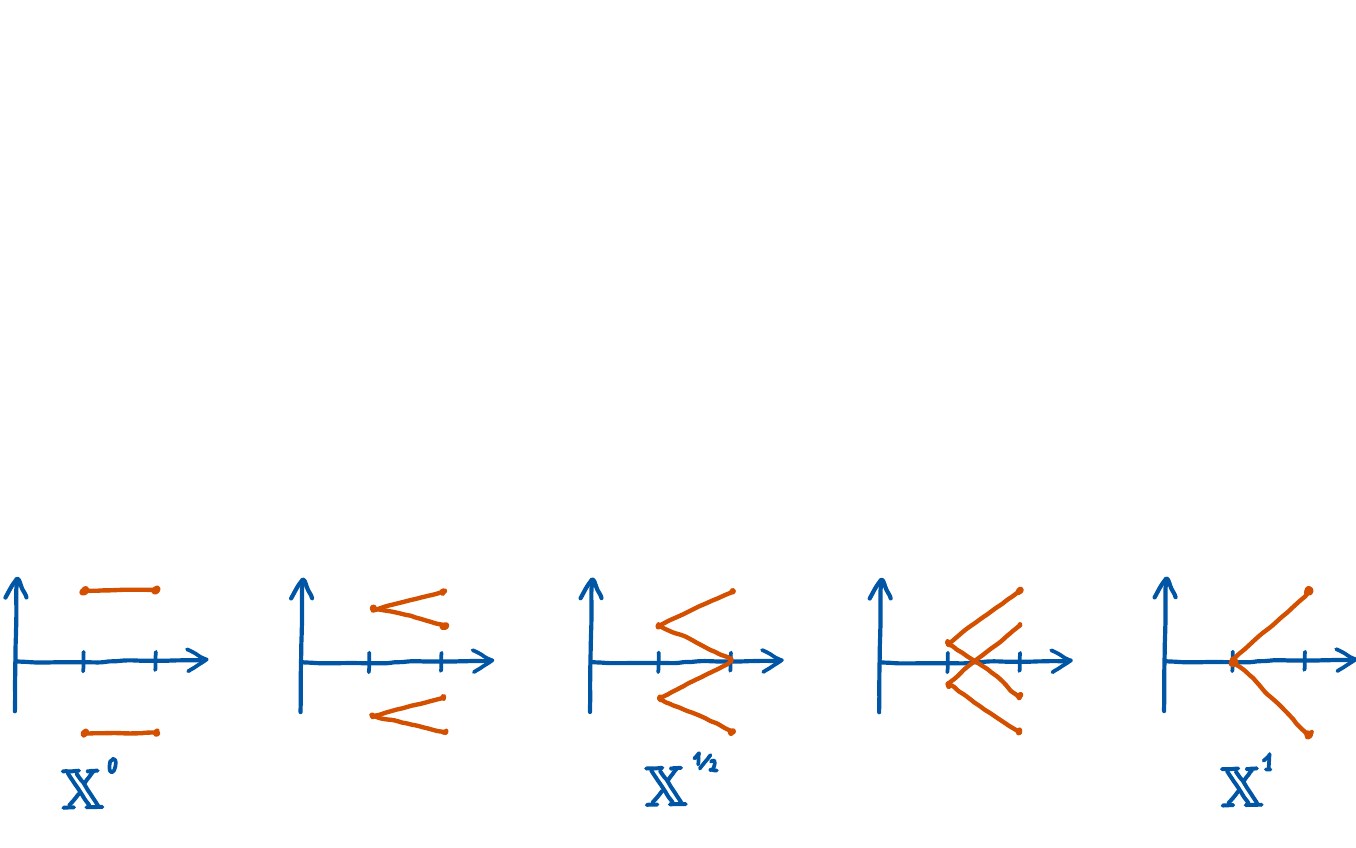} 
  \caption{Interpolation between two simple martingales $\fp{X}^0$ and $\fp{X}^1$.}
 \label{fig:AW.interpolation}
 \end{figure}
 \vspace{-02mm}
 
We emphasize that the usual Wasserstein interpolation on $\mathcal P_p(\R^N)$ is not compatible with concepts one would like to consider for stochastic processes, e.g.\ stochastic optimization problems are not continuous along geodesics, the set of martingales is not displacement convex, etc.

We also note that the set $\mathcal P_p(\R^N)$ is not $\AWA_p$-displacement convex: even if $\P, \Q\in \mathcal P_p(\R^N)$ are  laws of relatively regular processes, the respective geodesic does in general not lie in $\mathcal P_p(\R^N)$, see  Example \ref{ex:plain.not.geodesic}.
This further underlines the importance of  considering processes together with their filtration. 

\subsection{Equivalence of filtered processes}
We briefly discuss the equivalence relation induced by
\begin{align}\label{AWAequi}
\AWA_p(\fp{X}, \fp{Y})=0.
\end{align}
Intuitively, one would hope that processes with zero distance are equivalent in the sense that they have identical properties from a probabilistic perspective. 
In fact, based on formalizing  what assertions belong to the `language of probability', Hoover--Keisler \cite{HoKe84}  have made  precise what it should  mean that two processes $\fp X, \fp Y$ have the \emph{same probabilistic properties}. 
$\fp X$ and$ \fp Y$ are  then called   \emph{equivalent in  adapted distribution}, in signs $\fp{X} \sim_\infty \fp{Y}$. 
We will establish below that equivalence in adapted distribution can be expressed in terms of adapted Wasserstein distance:

\begin{theorem}\label{HKequiv}
Let $\fp{X}, \fp{Y} \in \FP_p$. 
Then $\fp{X} \sim_\infty \fp{Y} $  if and only if $ \AWA_p(\fp{X}, \fp{Y})=0$.
\end{theorem}

Informally, Theorem \ref{HKequiv} asserts that the equivalence classes in $\FFP_p$ collect precisely all representatives of a  process that should be considered identical from a probabilist's  point of view. 

Other (more familiar) notions of equivalence on $\FP$ are \emph{equivalence in law}, in symbols $\sim_0$, and Aldous' notion of \emph{synonymity}, in symbols $\sim_1$. 
Both of these are strictly coarser than $\sim_\infty$ and may identify processes that have different probabilistic properties. 
For example, there are filtered processes $\fp{X},\fp{Y}$ with $ \fp{X}\sim_0 \fp{Y}$ where $\fp{X}$ is a martingale while $\fp{Y}$ is not. Similarly, we will construct examples of processes $\fp{X},\fp{Y}$ where $\fp{X}\sim_1\fp{Y}$ but  optimal stopping problems written on $\fp{X},\fp{Y}$ lead to different results (see Theorem \ref{thm:stopping.rank}).

\subsection{Continuity of Doob-decomposition, optimal stopping, Snell-envelope}

In line with  Theorem \ref{HKequiv},  $\AWA_p(\fp{X},\fp{Y})=0$ implies  that $\fp{X}$ and $ \fp{Y}$ have the same Doob-decomposition,  that optimal stopping problems of the form 
\begin{align}
\label{OptStop}
\sup_\tau\E[ G_\tau (X_1, \ldots, X_\tau)],
\end{align} 
where $\tau$ runs through all $(\F_t)_{t=1}^N$-stopping times,  yield the same optimal value and have the same Snell-envelop. 
Moreover, we show that the above operations are continuous w.r.t.\  the weak adapted topology, indeed  we establish:

\begin{theorem} \label{thm:OptStopIntro}
The mapping that assigns to a filtered process its Doob-decomposition is Lipschitz continuous. 
If $G_t\colon\R^t\to \R$ is bounded and continuous (resp.\  Lipschitz) for each $t$, then \eqref{OptStop} is continuous (resp.\  Lipschitz) in $\fp{X}$. 
\end{theorem}

In Section \ref{sec:applications} we collect further statements of similar flavour as Theorem \ref{thm:OptStopIntro}. 
It is important to note that comparable results do not hold w.r.t.\ to other (coarser) topologies  for filtered processes. Specifically,  convergence in Aldous' extended weak topology is strictly weaker than convergence in $\AWA_p$ and is not strong enough to obtain continuity of optimal stopping problems (see Section \ref{sec:two.processes}). This seems remarkable, since Aldous \cite[page 105]{Al81} deliberates the question which framework is natural to study continuity of optimal stopping. 

\subsection{Canonical representatives  of filtered processes}
A slightly altered variant $\Z$  of the Polish space $\mathcal{V}$ appearing in Theorem \ref{WREP} also plays an important role in finding canonical representatives for the equivalence classes in $\FFP$: 
In Section \ref{sec:FP} below we will show that there exists   $(\Z,\F^{\Z},  (\F_t^{\Z})_{t=1}^N, (Z_t)_{t=1}^N) $ with $\Z$ Polish,
such that every filtered process $\fp{X}$ is represented in a canonical way via a probability on $\Z$, i.e.\ there exists $\Q^{\fp{X}}\in\PP(\Z)$  such that 
\[ (\Z,  \F^{\Z}, \Q^{\fp{X}}, (\F_t^{\Z})_{t=1}^N, (Z_t)_{t=1}^N ) \sim_\infty \fp{X}.\]
In particular, all information about the process $\fp{X}$ is stored in the corresponding measure $\Q^{\fp X}$, while  the underlying probability space $\Omega$, the representing stochastic process $Z$ and the respective filtration do not depend on $\fp{X}$. 
In this sense the situation is analogous to the canonical representation of stochastic processes via probabilities on the path space.
In view of Theorem \ref{HKequiv} this also implies that one can assume without loss of generality  that a given filtered process is defined on a Polish probability space.

\subsection{Prohorov-type result and barycenter of processes}

An extremely useful  property of the usual weak topology is the abundance of (pre-)compact sets based on Prohorov's theorem. Remarkably, this  carries over to `adapted' topologies. This was first established by Hoover \cite[Theorem 4.3]{Ho91}, see also \cite[Lemma 1.7]{BaBaBeEd19a}. In the present context this fact can be expressed as follows: 

\begin{theorem}\label{CompactPre}
\label{thm:compact}
A set $K\subseteq \FFP_p$ is $\AWA_p$-precompact if and only if the respective set of laws in  $\mathcal P_p(\mathbb{R}^N)$ is $\WA_p$-precompact.
\end{theorem}

Note that by Prohorov's theorem, $\WA_p$-precompactness in $\mathcal P_p(\R^N)$ is equivalent to tightness plus uniform $p$-integrability, see for instance \cite{Vi09}.

Theorem \ref{CompactPre} is relevant  in several proofs given below and has important consequences for the applications of our results presented in Section \ref{sec:applications}. For instance, it allows us to
 establish  the existence of barycenters of stochastic processes: 
Famously, Agueh and Carlier \cite{AgCa11} introduced the concept of barycenters w.r.t.\ Wasserstein distance which has striking consequences in machine learning (e.g.\ \cite{RaPeDeBe12, CuDo14}), statistics (e.g.\ \cite{PaZe19, BaFoRi18}) as well as in pure mathematics (e.g.\ \cite{LeLo17, KiPa17}). 
In Theorem \ref{cor:barycenter attainment} we show that for filtered processes $\fp X^1, \ldots, \fp X^k\in \FFP_p$ and convex weights $\lambda_1, \ldots, \lambda_k$ there exists a \emph{barycenter process} 
i.e.\ a 
filtered process $\fp X^*\in \FFP_p$ which minimizes  
$$\inf_{\fp X }  \lambda_1 \AW_p^p( \fp X^1, \fp X) + \ldots + \lambda_k \AW_p^p( \fp X^k, \fp X).$$ 

\subsection{Applications and extensions}

As already noted above, the adapted Wasserstein distance improves over the classical weak topology / Wasserstein distance in that it guarantees stability of basic operations such as the Doob-decomposition and optimal stopping. Naturally we expect similar results  for other probabilistic problems with inherent time structure. 
In this line, we describe applications to stability of stochastic optimal control, utility maximization and pricing / hedging,  robust finance in the realm of American options, 
conditional McKean-Vlasov control, and weak optimal transport, see Sections \ref{ssec:optimal stopping} - \ref{ssec:Doob-decomposition} below. 
In view of applications it is  relevant that  adapted Wasserstein distance can be efficiently computed numerically as well as  estimated from given data; we comment on this in Section \ref{ssec:numerics}.  

While the focus of the present article lies on stochastic processes in finite discrete time, extensions  to more general cases are intriguing. 
In Appendix \ref{sec:infinite.time} we consider the case of infinite discrete time, i.e.\ the set $\FFP^{(\infty)}$ of processes $(X_t)_{t=1}^\infty$ whose paths lie in $\R^\infty$ (or a countable product of Polish spaces). 
We obtain results very similar to the finite discrete time case, mainly based on limiting arguments.\footnote{We thank an anonymous referee for pointing us to this direction.} 
A notable difference is that the path space $\R^\infty$ is not geodesic and hence 
$\FFP^{(\infty)}$ is not geodesic either. 

Concerning continuous time processes $(X_t)_{t\in [0,T]}$, it is known from stochastic analysis, that different applications require the use of different topologies / metrics on the path space. 
This fact appears even more noticeable  when also information is taken into account. 
In Appendix  \ref{sec:cont.time} we briefly present adapted topologies for continuous time stochastic processes that have been used in the literature or seem sensible. We describe how $\AWA$ needs to be altered to fit the respective choices 
and comment on some strengths and weaknesses of the emerging theories. 

 \vspace{-0mm}
 \begin{figure}
    \centering
    \includegraphics[page=1,width=0.9\textwidth]
        {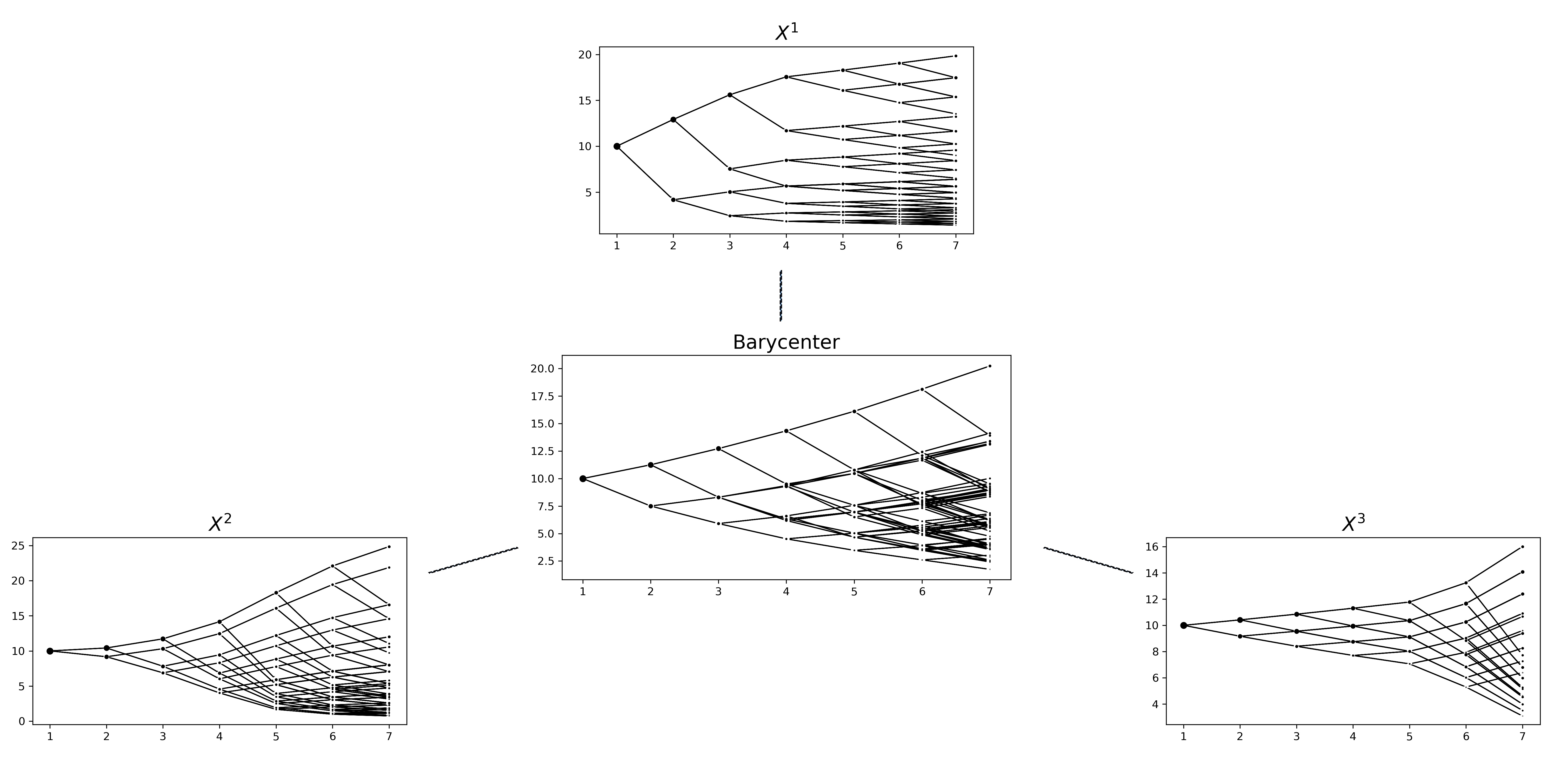} 
  \caption{Three CRR-type models as considered in finance and econometrics  and their barycenter (implemented by Stephan Eckstein).}
 \label{fig:barycenter}
 \end{figure}
 \vspace{-02mm}

\subsection{Remarks on related literature.}

Imposing a `causality' constraint on a transport plan between laws of processes  seems to go back to the Yamada--Watanabe criterion for stochastic differential equations \cite{YW} and is used under the name `compatibility' by Kurtz \cite{Ku07}.

A systematic treatment and use of  causality as an interesting property of abstract transport plans between filtered probability spaces and their associated optimal transport problems was initiated by  Lassalle \cite{La18} and Acciaio, Backhoff, and Zalashko \cite{AcBaZa20}. 

As noted above, different groups of authors have introduced similar `adapted' variants of the Wasserstein distance, this includes the works of 
 Vershik \cite{Ve70, Ve94}, 
R\"{u}schendorf \cite{Ru85},  Gigli \cite[Chapter 4]{Gi04} (see also  \cite[Section 12.4]{AmGiSa08}), Pflug and Pichler \cite{PfPi12}, Bion-Nadal and Talay \cite{BiTa19}, and Nielsen and Sun \cite{NiSu20}. Pflug and Pichler's nested distance has had particular impact in multistage programming, see \cite{PfPi14, PfPi15, GlPfPi17, KiPfPi20, PiSh21}  among others. 

In addition to these distances, 
extensions of the weak topology that account for the flow of information were introduced by Aldous in stochastic analysis \cite{Al81} (based on Knight's prediction process \cite{Kn75}), 
 Hoover and Keisler in mathematical logic \cite{HoKe84, Ho91} and Hellwig in economics \cite{He96}. Very recently, an approach using higher rank signatures was given by Bonner, Liu, and Oberhauser \cite{BoLiOb23}, in particular providing a metric for convergence in adapted distribution in the sense of Hoover--Keisler.

The idea to represent information (in the sense of filtrations) using conditional distributions  originated in the theory of dynamical systems and Vershik's program to classify filtrations whose time horizons starts at $-\infty$,  see e.g.\ \cite{Ve94}   and in particular the survey \cite{Ve17}. For a more probabilistic account of this line of research we refer to \cite{EmSc01}. 
Independently, Pflug \cite{Pf09}  introduced  this idea in stochastic optimization and defined the space of `nested distributions'.

Recently, there has been significant interest in adapted / causal  transport problems in discrete time or with a finite number of hierarchical levels. 
A goal of the present article is to provide the theoretical framework for these  emerging lines of research and we briefly indicate some of these directions:
In mathematical finance, the use  of weak adapted topologies was initiated in the context of game options  \cite{Do14}. Further contributions apply adapted transport and adapted  Wasserstein distances to questions of insider trading and enlargement of filtrations \cite{AcBaZa20},  stability of pricing / hedging and utility maximization  \cite{GlPfPi17, BaDoGu20, BaDoDo20, BaBaBeEd19a, AcBaPa22, BeJoMaPa21a, BeJoMaPa21b} and  interest rate uncertainty \cite{AcBePa20}.
Adapted transport is used in \cite{BaWi23} to study the sensitivity of multiperiod optimization problems and distributionally robust optimization problems. In \cite{Ha23} it is applied to time-dynamic matching problems  \cite{BaHa23a}, and in \cite{SaLeLiHoDaLy21, HoLeLiLySa23} for the computational resolution of optimal stopping and other filtration-dependent problems.
In \cite{ChLiMeWaWa23} a connection of adapted transport  to the Weisfeiler-Lehman distance is reveiled.
Machine learning algorithms based on adapted or hierarchical structures are studied in the context of image processing \cite{KaGuMaSi23}, text processing and hierarchical domain translation \cite{YuClChMiSo19, ElBeFa22}, causal graph learning \cite{AkGaKi23}, video prediction and generation \cite{XuLiMuAc20,XuAc21}, and universal approximation \cite{AcKrPa23}.
In \cite{ChEc23}, adapted transport is used as the starting point to develop a framework for more general causal dependence structures.

\subsection{Organization of the paper}
In Section \ref{ssec:notation} we introduce some important concepts and in particular the notions of (bi-) causality and adapted Wasserstein distance.

In Section \ref{sec:FP} we formally discuss the Wasserstein space of filtered processes $\FFP_p$ as a preparation  to establish Theorem \ref{MainTheorem} and Theorem \ref{WREP} subsequently. 
In Subsection \ref{ssec:canonical.space} we construct a canonical filtered space that supports for each equivalence class $\fp X \in \FFP_p$ a canonical representative.
Building on the foregoing subsections, we establish in Subsection \ref{ssec:isometry} an  isometric isomorphism between filtered processes, their canonical counterparts, and a classical Wasserstein space.

Section \ref{sec:AF.PP} links the weak  adapted topology to the concept of adapted functions and the prediction process by Hoover and Keisler.

Section \ref{sec:aspects} deals with topological and geometric aspects.
We prove a  compactness criterion, show that $\FFP_p$ is the completion of $\Pc_p(\R^N)$ and  prove that finite state Markov processes are dense in $\FFP_p$, prove that $\FFP_p$ is a geodesic space for $1 < p < \infty$, and show that martingales form a closed, geodesically convex subset of $\FFP_p$. 

In Section \ref{sec:applications} we discuss applications and  comment on numerical aspects related to  $\AWA$.

Section \ref{sec:two.processes} is concerned with an example that, among other things, shows that Aldous' extended weak topology fails to guarantee continuity of  optimal stopping problems.

Finally, in Appendix \ref{sec:block approximation}, we discuss a notion of `block approximation' of couplings, which is an auxiliary concept required  to prove the results in  Subsection \ref{ssec:isometry} for probability spaces that are not necessarily Polish. 
In Appendix \ref{sec:infinite.time} and \ref{sec:cont.time} we discuss extensions of the present setting to the case of stochastic processes index by infinite discrete time and continuous time, respectively.

\section{Notational conventions} 
\label{ssec:notation}

Throughout this article, we fix a time horizon $N\in\N$ and $1 \le p < \infty$.
For each time $1\leq t\leq N$, let $\X_t$ be a Polish space with a fixed compatible complete metric $d_{\X_t}$. 
If $\X_t=\R^d$, then $d_t(x,y)=|x-y|$ where $|\cdot|=\|\cdot\|_2$ is the Euclidean norm.
Given a finite family of sets $(A_u)_{u = 1}^{t}$ and $1\leq s\leq t$, we use the following abbreviation for its product
\[	A_{s:t} := A_s \times \ldots \times A_t.	\]
The same convention applies to vectors.
For  $s \le r \le t$, the projection onto the $r$-th coordinate of $A_{s:t}$ is denoted by $\proj_r \colon A_{s:t} \to A_r$.
Using this convention we are interested in stochastic processes taking values in the path space $\X := \X_{1:N}$. 
Processes on $\X$ are usually denoted by capital letters, i.e., $X = (X_t)_{t=1}^N$, whereas specific elements of the path space are denoted by lower case, i.e., $(x_t)_{t=1}^N \in \X_{1:N}$.

\vspace{0.5em}
\emph{Distances}: 
For a Polish space $\mathcal{A}$ with fixed compatible complete metric $d_\mathcal{A}$, we write $\Pc(\mathcal{A})$ for the set of Borel probability measures on $\mathcal{A}$, and $\Pc_p(\mathcal{A})$ for the subset whose elements integrate  $d^p_\mathcal{A}(\cdot,a_0)$ for some (and hence all) $a_0 \in \mathcal{A}$.
If $\mathcal{B}$ is another Polish space and  $\mu\in \Pc(\mathcal{A})$, $\nu \in \Pc(\mathcal{B})$, we write $\cpl(\mu,\nu)$ for the set of all couplings with marginals $\mu,\nu$, that is $\pi \in \cpl(\mu,\nu)$ if $\pi \in \Pc(\mathcal{A} \times \mathcal{B})$ 
and its first marginal equals $\mu$ and its second $\nu$.
We equip $\Pc(\mathcal{A})$ with the topology of weak convergence and $\Pc_p(\mathcal{A})$ with the ($p$-th order) Wasserstein distance $\mathcal{W}_{\Pc_p(\mathcal{A})}$, that is
\[
	\W_{\Pc_p(\mathcal{A})}^p(\mu,\nu):= \inf_{\pi \in \cpl(\mu,\nu)} \int d_\mathcal{A}^p(a,\hat a) \, \pi(da,d\hat a).
\]
This renders $\Pc(\mathcal{A})$ and $\Pc_p(\mathcal{A})$ Polish spaces.
Note that for a bounded metric, the weak convergence topology and the one induced by $\W_p$ coincide.
Whenever clear from context, we will omit excessive subscripts and simply write $d$ and $\W_p$ for $d_{\mathcal{A}}$ and $\mathcal{W}_{\Pc_p(\mathcal{A})}$ respectively.

\vspace{0.5em}
\emph{Filtrations:} 
For a filtered process ${\fp X}=(\Omega^{\fp X}, \F^{\fp X}, \P^{\fp X}, (\F^{\fp X}_t)_{t=1}^N, (X_t)_{t=1}^N)$, we use the convention that $\F^{\fp X}_0:=\{\emptyset,\Omega^{\fp X} \}$.
It is important to note that since the processes start at time $t=1$, this convention is only notational and does not imply that the initial $\sigma$-algebra is trivial.
Frequently we consider multiple products of $\sigma$-algebras.
To prevent  notation getting out of hand, we write 
\[\F_{t,s}^{\fp X,\fp Y} := \F_t^\fp X \otimes \F_s^\fp Y\quad\text{for } 0\leq s,t\leq N\]
for two filtered processes $\fp X$ and $\fp Y$.
Moreover, we will often identify $\mathcal{F}^{\fp X,\fp Y}_{0,t}$ with $\F^{\fp Y}_t$; e.g.\ an $\F^{\fp X,\fp Y}_{0,t}$-measurable functions is naturally associated on $\Omega^{\fp X}\times\Omega^{\fp Y}$ with an $\F^{\fp Y}_t$-measurable function depending only on the second coordinate, and vice versa.
In a similar manner, for a function $f\colon A_t\to\Y$, we continue to write $f\colon A_{1:n}\to\Y$ for the function $f\circ\mathrm{pr}_t$.

\vspace{0.5em}
\emph{Couplings}:
In optimal transport couplings are the central tool for comparing probability measures.
For filtered processes this role is taken by bicausal couplings, i.e.\ couplings which respect the information structure of the underlying filtered probability spaces. 

\begin{definition}[Causal couplings]
\label{def:causal}
 	Let $\fp{X}$, $\fp{Y}$ be filtered processes.
 	A probability $\pi$ on $(\Omega^\fp X \times \Omega^\fp Y, \F^\fp X \otimes \F^\fp Y)$  is called \emph{coupling} between $\fp X$ and $\fp Y$ if its marginals are $\P^\fp X$ and $\P^\fp Y$.
	We call $\pi$ 
 	\begin{enumerate}[label=(\alph*)]
	\item
 	\emph{causal} (or causal from $\fp{X}$ to $\fp{Y}$) if, for every $1\leq t\leq N$, conditionally on $\F^{\fp{X},\fp{Y}}_{t,0}$ we have that $\F^{\fp{X},\fp{Y}}_{N,0}$ and $\F^{\fp{X},\fp{Y}}_{0,t}$ are independent,
 	\item
 	\emph{anticausal} (or causal from $\fp{Y}$ to $\fp{X}$) if, for every $1\leq t\leq N$, conditionally on $\F^{\fp{X},\fp{Y}}_{0,t}$ we have that $\F^{\fp{X},\fp{Y}}_{0,N}$ and $\F^{\fp{X},\fp{Y}}_{t,0}$ are independent,
	\item
 	\emph{bicausal} if it is both, causal and anticausal.
 	\end{enumerate}
 	We  write $\cpl(\fp X,\fp Y)$, $\cpla(\fp X,\fp Y)$ and $\cplba(\fp X, \fp Y)$ for the set of couplings, causal couplings, and bicausal couplings, respectively.
 \end{definition}
 
In case that the underlying spaces are path spaces equipped with canonical filtration / processes, these definitions correspond precisely to the classical definitions given in the literature, see \cite{PfPi12, PfPi14, BaBeLiZa16,   BaBeEdPi17, BaBaBeEd19b} among others. 
In the context of space with more general filtrations causality and causal transport are considered in \cite{La18, AcBaZa20}.

The following lemma provides useful characterizations of causality which we will frequently use throughout the article.

 \begin{lemma}[Causality]
 \label{lem:causal CI}
 	Let $\pi$ be a coupling between two filtered processes $\fp X$ and $\fp Y$.
 	Then the following are equivalent.
 	\begin{enumerate}[label=(\roman*)]
 		\item \label{it:causal CI}$\pi$ is causal (from $\fp X$ to $\fp Y$).
 		\item \label{it:causal CI1}
 		$\E_\pi[ U | \F^{\fp X,\fp Y}_{t,t} ]
 		= \E_\pi [ U | \F^{\fp X,\fp Y}_{t,0} ]$ for all $1\leq t \leq N$ and bounded $\F_N^\fp X$-mb.\ $U$.
 		\item \label{it:causal CI2}
 		$\E_\pi[ V | \F^{\fp X,\fp Y}_{N,0} ]
 		= \E_\pi[ V | \F^{\fp X,\fp Y}_{t,0} ]$
 		for all $1\leq t\leq N$ and bounded $\F_t^\fp Y$-mb.\ $V$.
 	\end{enumerate}
 	Moreover, $U$ in \ref{it:causal CI1} can be allowed to be $\F^{\fp X,\fp Y}_{t,N}$-measurable and $V$ in \ref{it:causal CI2} can be allowed to be $\F^{\fp X,\fp Y}_{t,t}$-measurable.
 \end{lemma}
 
	In words, \ref{it:causal CI1} says that given the past of $\fp X$, the past of $\fp Y$ does not provide additional information about the future of $\fp X$ and \ref{it:causal CI2} says that given the past of $\fp X$, the future of $\fp X$ does not provide additional information about the past of $\fp Y$.
  
 \begin{proof}[Proof of Lemma \ref{lem:causal CI}]
 	The equivalence between \ref{it:causal CI}-\ref{it:causal CI2} is a consequence of \cite[Proposition 5.6]{Ka97}.
 	The second statement follows from  a standard application of the  monotone class theorem.
 \end{proof}

The adapted Wasserstein distance between filtered processes $\fp X$ and $\fp Y$ taking values in $\X$ is then defined as 
\[\AW_p^p(\fp X,\fp Y) :=\inf_{\pi\in\cplba(\fp X, \fp Y) } \E_\pi[d_{\X,p}^p(X,Y)]
\quad\text{where } d_{\X,p}(x,y)=\Big(\sum_{t=1}^N d_{\X_t}^p(x_t,y_t) \Big)^{\frac{1}{p}}.\]
When clear from context, we write $d$ instead of $d_{\X,p}$.
Similarly, we write $\E[f(X)]$ instead of $\E_{\P^{\fp X}}[f(X)]$ etc.

\vspace{0.5em}
{\it Kernels and product measures:}
For two Polish spaces $\mathcal{A}$ and $\mathcal{B}$, the term kernel refers to a Borel-measurable mapping $k\colon\mathcal{A}\to\mathcal{P}_p(\mathcal{B})$.
For $\mu\in\mathcal{P}_p(\mathcal{A})$ and a kernel $k$, we write $\mu\otimes k\in\mathcal{P}_p(\mathcal{A}\times\mathcal{B})$ for the measure given by $\mu\otimes k(A\times B)=\int_A k^a(B)\,\mu(da)$.
If $\nu\in\mathcal{P}_b(\mathcal{B})$ we write $\mu\otimes\nu\in\mathcal{P}_p(\mathcal{A}\times\mathcal{B})$ for the product measure. 
For a measure $\mu\in\mathcal{P}(\A)$ and a Borel-measurable mapping $f\colon\A\to\B$, the push-forward of $\mu$ under $f$ is denoted by $f_\ast\mu$.

\section{The Wasserstein space of stochastic processes} \label{sec:FP}

\subsection{The canonical filtered space}
\label{ssec:canonical.space}

In order to prove Theorem \ref{MainTheorem} we introduce the canonical space of  filtered processes.
The classical canonical space of a stochastic process $X$ is the triplet consisting of path space, Borel-$\sigma$-algebra, and its induced law.
Clearly, this triplet is adequate if one is interested solely in trajectorial properties of the process.
However, the filtration is a major part of filtered processes $\fp X \in \FP$ and  therefore we need to capture the information contained in its filtration $\F^\fp X$ in a canonical way.
Thus we need to define a canonical space which is capable to carry besides the path properties also the relevant informational properties of $\fp X$.

 As an instructional example, consider two 1-step filtered processes $\fp X$ and $\fp Y$ taking values in $\X = \X_1 \times \X_2 = \{ 0 \} \times \{-1,1\}$.
 We write $\mathcal{G}_1$ for the trivial $\sigma$-algebra, $\mathcal{G}_2=\mathcal{B}$ for the Borel $\sigma$-algebra on $\X$, and $X = (X_t)_{t = 1}^2$ for the coordinate process on $\X$.
 Let $\P := \frac{1}{2}( \delta_{(0,-1)} + \delta_{(0,1)} )$, and define 
 \begin{align*}
	 \fp X &:= \left( \X , \mathcal{B}, \fp P, (\mathcal{G}_1,\mathcal{G}_2), X  \right)
	 \quad\text{and}\quad
	 \fp Y := \left( \X , \mathcal{B}, \P, (\mathcal{G}_2,\mathcal{G}_2),  X  \right).
\end{align*}
 Even though the laws of the paths of the two processes coincide, their probabilistic behavior is very different due to their different filtrations, specifically we have
 \begin{equation}
	 \Law(X_2 | \F_1^{\fp X})(\omega) = \frac{1}{2}(\delta_{-1} + \delta_1) 
	 \text{ whereas } 
	 \Law(X_2 | \F_1^{ \fp Y})(\omega) = \delta_{X_2(\omega)}.
 \end{equation}
 From this perspective, the product $\X_1 \times \Pc(\X_2)$ is adequate to capture $X_1$ and additionally the information on $X_2$ we can witness at time 1 based on the filtration, that is $(X_1, \Law(X_2 | \F_1^{\fp X}))$.

This observation leads to the definition of what we baptize the canonical filtered space, the information process, and the canonical filtered process  below:
 
 \begin{definition}[Canonical space]
	\label{def:canonical.space.Z}
	Fix $p \in [1,\infty)$. We iteratively define a sequence of nested spaces. 
	We write $(\Z_N,d_{\Z_N}):=(\X_N,d_{\X_N})$ and recursively for $t=N-1,\dots,1$ 
 \begin{equation} \label{def:Zt}
 	\Z_t := \Z_t^-\times\Z_t^+ :=\X_t\times\Pc_p(\Z_{t+1}) 
 \end{equation}
 with metric $d_{\Z_t}^p:= d_{\X_t}^p + \W_{p,\Pc(\Z_{t+1})}^p$.
 The elements of $\Z_t$ are denoted by $z_t = (z_t^-,z_t^+) \in \Z_t^- \times \Z_t^+$.
 The \emph{canonical filtered space} is given by the triplet
	\begin{equation} \label{eq:CFS}
		\left( \Z, \F^\Z, (\F^\Z_t)_{t=1}^N \right),
	\end{equation}
	 where $\Z_{1:N}$ is denoted by $\Z$, elements of $\Z$ by $z = z_{1:N}$, the Borel-$\sigma$-algebra on $\Z$ by $\F^\Z$, and $\sigma(z \mapsto z_{1:t})$ by $\F_t^\Z$.
	In the context of the canonical filtered space the map $Z^- \colon \Z \to \X_{1:N}$ denotes the evaluation map
	 \begin{equation}
		 \label{def:evaluation map}
		 Z^-(z) := (Z^-_t(z))_{t = 1}^N := (z_t^-)_{t = 1}^N.
	 \end{equation}
\end{definition}

The spaces introduced in Definition \ref{def:canonical.space.Z} are Polish as all operations involved in their definition preserve this property.
The next definition associates to a filtered process $\fp{X}$ its canonical counterpart -- the information process $\ip(\fp X)$ defined on $\Omega^{\fp X}$ and taking values in $\Z$.
We shall later see that the information process selects all information contained in the original filtration relevant for the process; hence its name.

\begin{definition}[The information process]
	\label{def:ip}
	To each $\fp X\in\FP$, we associate its \emph{information process} $\ip(\fp X)=(\ip_t(\fp X))_{t=1}^N$  defined by setting
		\begin{align*}
		\ip_N(\fp X)\colon\Omega^\fp X&\to\mathcal{Z}_{N},\quad 	\omega \mapsto X_N(\omega)
		\end{align*}		
		and, recursively  for $t=N-1,\dots,1$,
		\begin{align*}
		\ip_t(\fp X) =(\ip_t^-(\fp{X}), \ip_t^+(\fp{X}))  \colon\Omega^\fp X &\to \mathcal{Z}_t, \quad
		\omega \mapsto \left( X_t(\omega), \Law(\ip_{t+1}(\fp{X})|\F_t^\fp{X})(\omega) \right).
		\end{align*}	
\end{definition}

 Note that $\ip_t^-(\fp X)$ is $\Z_t^-$-valued and $\ip_t^+(\fp X)$ is $\Z_t^+$-valued.
 In particular this implies that the information process is well-defined: 
 as $(\Z_t)_{t=1}^N$ consists of Polish spaces, the conditional probabilities appearing in the recursive definition of $\ip(\fp X)$ above exist.
 
 The information processes will be an essential ingredient when we  define canonical representatives of filtered processes (see Definition \ref{def:CFP.associated.to.FP} below).
 The next lemma can be seen as a first justification of its name:
 
 \begin{lemma}[The information process is self-aware]
 \label{lem:ip.self.aware}
 	For every bounded, Borel (continuous) function $f\colon\Z\to\R$ and every $1\leq t\leq N$ there is a bounded, Borel (continuous) function $g\colon\Z_{1:t}\to\R$ such that
 	\[ \E[f(\ip(\fp{X}))  | \F_t^\fp{X}] = g(\ip_{1:t}(\fp{X})) \quad\text{for all } \fp X\in \FP.\]

	More generally, let $\mathcal A$ be a Polish space.
	For every Borel (continuous) function $f \colon \Z \to \mathcal A$ and every $1 \le t \le N$ there is a Borel (continuous) function $g \colon \Z_{1:t} \to \Pc(\mathcal A)$ such that
	\[
		\Law(f(\ip(\fp X)) | \F_t^\fp X ) = g(\ip_{1:t}(\fp X))\quad\text{for all }\fp X \in \FP.
	\]
 \end{lemma}
 
 This lemma is an easy consequence of properties of the unfold operator introduced below, see Lemma \ref{lem:unfoldprops.measures}.
	Nevertheless, to boost the reader's intuition, we want to include a direct proof for the first statement in the notationally lighter case $N=2$:
 
 \begin{proof}[Sketch of proof for $N=2$] 
	For $t=2$ there is nothing to do, as $\ip(\fp{X})$ is $\F_2^\fp{X}$-measurable.
 	Let $t=1$. 
 	For simplicity, we assume first that $f$ is a product $f(z_1,z_2)=f_1(z_1)f_2(z_2)$ for suitable $f_1$ and $f_2$.
 	Then we can write
 	\begin{align*}
 	\E[f(\ip(\fp{X}))  | \F_1^\fp{X}]
 	&= f_1(\ip_1(\fp{X})) \E[ f_2(\ip_2(\fp{X})) | \F_1^\fp{X}] = f_1(\ip_1(\fp{X})) \int f_2(z_2)\, \Law(\ip_2(\fp{X}) | \F_1^\fp{X})(dz_2),
 	\end{align*}
 	where we used that $\ip_1(\fp{X})$ is $\F_1^\fp{X}$-measurable.
 	By definition $\ip_1^+(\fp{X})=\Law(\ip_2(\fp{X}) | \F_1^\fp{X})$, whence
 	\begin{align*}
 	\E[f(\ip(\fp{X}))  | \F_1^\fp{X}]
 	&= g (\ip_1(\fp{X})) \quad\text{for }  \quad
 	g(z_1)
 	:=f_1(z_1)\int f_2(z_2)\, z_1^+(dz_2).
 	\end{align*}
 	For general $f$ not necessarily of product form, a straightforward application of the monotone class theorem concludes the proof.
\end{proof}

	In what follows, we are often dealing with mappings between nested spaces of probability measures with  different algebraic structures.
	The unfold operator, introduced below, is an essential tool in reducing bookkeeping to a comprehensible level.

\begin{definition}[Unfold]
	\label{def:unfold}
		For every $1\leq t\leq N-1$, we define
		\begin{align} \label{eq:unfold1}
			\begin{split}
			\unfold_t &\colon \Pc_p(\Z_t) \to \Pc_p(\Z_{t:N}),\\
			\mu &\mapsto \mu(dz_t)\, z_t^+ (dz_{t+1}) \,z_{t+1}^+(dz_{t+2})\ldots \,z_{N-1}^+(dz_N),
			\end{split}
		\end{align}
		and call $\unfold_t $ the \emph{unfold} operator (at time $t$).
		For $t=N$, we define $\unfold_N$ to be the identity map.
	\end{definition}

	Recall the definition of $\ip(\fp X)$, see Definition \ref{def:ip}. 
	As $\ip_t^+(\fp{X})$ is a random variable taking values in $\Pc(\Z_{t+1})$, the unfold operator can be applied pointwise, i.e., we may consider $\unfold_{t+1}(\ip_t^+(\fp{X}))$.
	The following lemma explores properties of $\unfold_t$ in view of the information process $\ip(\fp X)$:

	\begin{lemma} \label{lem:unfoldprops.measures}
		For $1\leq t \leq N-1$ the following hold:
		\begin{enumerate}[label = (\roman*)]
			\item \label{it:unfold.recursive}
				For $\mu\in\Pc_p(\Z_t)$ we have
				\begin{equation} \label{eq:unfold1'}
						\unfold_t(\mu)(dz_{t:N}) =\mu(dz_t) \, \unfold_{t+1}(z_t^+)(dz_{t+1:N}).
				\end{equation}
			\item \label{it:unfoldprops2}
				For $\fp X \in \FP_p$ we have
				\begin{equation} \label{eq:ip unfold connection}
					\Law(\ip(\fp{X})|\F_t^{\fp{X}}) = \delta_{\ip_{1:t}(\fp{X})} \otimes\unfold_{t+1}(\ip_t^+ (\fp{X})).
				\end{equation}
			In other words, for all bounded Borel functions $f\colon \Z \to \R$ we have
			\begin{equation} \label{eq:ip unfold connection'}
				\E[f(\ip(\fp{X})) \mid \F_t^{\fp{X}}]
				= \int f (\ip_{1:t}(\fp{X}),z_{t+1:N} ) \, \unfold_{t+1}(\ip_t^+ (\fp{X})) (dz_{t+1:N}).
			\end{equation}		
			\item \label{it:unfoldprops3}
			$\unfold_t$ is Lipschitz continuous from $\Pc_p(\Z_{t})$ to $\Pc_p(\Z_{t:N})$.
		\end{enumerate}
	\end{lemma}
			
	\begin{proof}
		\begin{enumerate}[label=(\roman*),wide]
		\item Equation \eqref{eq:unfold1'} is another way of expressing \eqref{eq:unfold1}.
		\item The statements in \eqref{eq:ip unfold connection} and \eqref{eq:ip unfold connection'} are clearly equivalent and we shall therefore only prove the latter one via a backward induction.
		
		Since there is nothing to do for $t=N-1$, we assume that the statement is true for $2\leq t \leq N-1$.
		Let $g\colon\Z_{1:t}\to\R$ be the Borel function defined as
		\begin{equation}
			\label{eq:definition of g unfoldprops}
			g(z_{1:t}):=\int f(z_{1:t},z_{t+1:N})\, \unfold_{t+1}(z_t^+)(dz_{t+1:N}),
		\end{equation}
		whereby the inductive hypothesis now reads  $\E[f(\ip(\fp{X})) | \F_t^{\fp{X}}] =g(\ip_{1:t}(\fp{X}))$.
   
		Then $\ip_{t-1}^+(\fp{X}) = \Law(\ip_t(\fp{X})|\F_{t-1}^\fp{X})$ and the tower property implies
		\begin{align} \nonumber
			\E[f(\fp{X})|\F^\fp{X}_{t-1}] &=\E[ \E[f(\fp{X})|\F^\fp{X}_t] |\F^\fp{X}_{t-1}] = \E[g(\ip_{1:t}(\fp{X})) |\F^\fp{X}_{t-1}]
		\\ \label{eq:something.124}
			&= \int g(\ip_{1:t-1}(\fp{X}),z_t)\,\ip_{t-1}^+(\fp{X})(dz_t).
		\end{align}
		Recalling \eqref{eq:unfold1'} and \eqref{eq:definition of g unfoldprops}, we conclude that the last term in \eqref{eq:something.124} equals
		\begin{align*}
	&	\int\left(\int  f (\ip_{1:t-1}(\fp{X}),z_t,z_{t+1:N} ) \, \left( \unfold_{t+1}(z_t^+ \right) (dz_{t+1:N}) \right) \,\ip_{t-1}^+(\fp{X})(dz_t)\\
		&=\int  f (\ip_{1:t-1}(\fp{X}),z_{t:N} ) \, \unfold_{t}(\ip_{t-1}^+(\fp{X})\big) (dz_{t:N}),
		\end{align*}
		which completes the proof of \ref{it:unfoldprops2}. 
   
		\item The assertion is shown via a backward induction over $t$. 
		
		For $t=N$ the unfold operator is given as the identity map, which is in particular Lipschitz continuous.
		Assume that the claim is true for some $2\leq t+1\leq N$.
		By \eqref{eq:unfold1'} we can write $\unfold_t(\mu)=\mu\otimes k_t$, where $k_t\colon \Z_t\to\Pc_p(\Z_{t+1:N})$. 
		This map is explicitly given by $k_t(z_t) := \unfold_{t+1}(z_t^+)$.
		An application of Lemma \ref{lem:lipschitz kernel} below implies that $\unfold_t$ is again Lipschitz continuous with a new constant.
		\qedhere
		\end{enumerate}
	\end{proof}

\begin{lemma}
	\label{lem:lipschitz kernel}
		Let $\mathcal{A}$ and $\mathcal{B}$ be  Polish spaces and let $k\colon \mathcal{A} \to \Pc_p(\mathcal{B})$ be  $L$-Lipschitz.
		Then the map
		\begin{align*}
		\Pc_p(\mathcal{A}) &\to \Pc_p(\mathcal{A}\times  \mathcal{B}), \quad  \mu \mapsto \mu \otimes k 
		\end{align*}
		is $(1+L^p)^{1/p}$-Lipschitz.
	\end{lemma}
	\begin{proof}
		A short computation shows that there are $a_0\in\mathcal{A}$, $b_0\in\mathcal{B}$ and a constant $c$ such that $\int d^p(b,b_0)\,k^a(db)\leq  c(1+d^p(a,a_0))$ for all $a\in\mathcal{A}$; in particular $\mu \otimes k \in\Pc_p(\mathcal{A}\times  \mathcal{B})$ is well-defined for every $\mu\in\Pc_p(\mathcal{A})$.		
		
		Let $\mu,\hat{\mu}\in\Pc_p(\mathcal{A})$ and denote by $\pi$ the $\mathcal{W}_p$-optimal coupling between them.
		For every pair $(a,\hat a) \in \mathcal A \times \mathcal A$, let $\gamma^{a,\hat a} \in \cpl(k^a,k^{\hat a})$ be optimal for $\W_p(k^a,k^{\hat a})$.
		Using the Jankov-von Neumann theorem \cite[Theorem 18.1]{Ke95}, standard arguments show that $a,\hat{a}\mapsto \gamma^{a,\hat{a}}$ can be chosen universally measurably.
		Then
		\[ \Pi(da,db,d\hat{a},d\hat{b}):=\pi(da,d\hat{a}) \gamma^{a,\hat a}(db,d\hat b)\]
		defines a coupling between $\mu\otimes k$ and $\hat{\mu}\otimes k$.
		Thus we can estimate
		\begin{align*}
			\mathcal{W}_p^p(\mu \otimes k, \hat{\mu} \otimes k )
			&\leq \int d^p(a,\hat{a}) + d^p(b,\hat b) \,\Pi(da,db,d\hat{a},d\hat{b})\\
			&= \int d^p(a,\hat{a}) + \mathcal{W}_p^p(k^a,k^{\hat{a}}) \,\pi(da,d\hat{a})\leq (1+L^p) \mathcal{W}_p^p(\mu,\hat{\mu}),
		\end{align*}
		where the last inequality holds by $L$-Lipschitz continuity of $k$.
	\end{proof}

	\begin{proof}[Proof of Lemma \ref{lem:ip.self.aware}]
		Let $f \colon \Z \to \mathcal A$ be bounded and Borel measurable (continuous), and $1 \le t \le N - 1$.
		As $\unfold_{t+1}$ is continuous by Lemma \ref{lem:unfoldprops.measures} \ref{it:unfoldprops3}, we obtain continuity of
		\[
			F \colon \Z_{1:t} \to \Pc_p(\Z),\quad z_{1:t} \mapsto \delta_{z_{1:t}} \otimes \unfold_{t+1}(z_t^+).
		\]
		
		Define the map $G \colon \Z_{1:t} \to \Pc(\mathcal A)$ as the push-forward $G(z_{1:t}) := f_\ast F(z_{1:t})$.
		Obviously, when $f$ is continuous, $G$ is also continuous as the composition of continuous functions.
		By Lemma \ref{lem:unfoldprops.measures} \ref{it:unfoldprops2}, we obtain
		\[
			\Law \left( f(\ip(\fp X)) | \F_t^\fp X\right) = f_\ast \Law\left(\ip(\fp X) | \F^\fp X_t \right) = f_\ast F(\ip_{1:t}(\fp X)) = G(\ip_{t+1}(\fp X)).
		\]
		We have thus proved the second assertion.

		To obtain the first assertion, let $\mathcal A = \R$.
		Define $g\colon \Z_{1:t} \to \R$ by
		$
			g(z_{1:t}) := \int a \, G(z_{1:t})(da),
		$
		which is well-defined since $f$ is bounded with values in a compact set, say, $K\subset \R$.
		Hence, we can view $G$ as a function mapping into $\Pc(K)$.
		By the first part of the proof, we obtain
		\[
			\E\left[ f(\ip(\fp X)) | \F^\fp X_t \right]	 = \int a \, G(\ip_{t+1}(\fp X))(da) = g(\ip_{t+1}(\fp X)).
		\]
		The map $p \mapsto \int a \, p(da)$ is continuous on $\Pc(K)$ and we conclude that $g$ is continuous if $f$ is.
	\end{proof}

	Equipped with the unfold operator we define our real object of interest:

	\begin{definition}[Canonical filtered processes]
		We call $\fp X \in \FP$ a \emph{canonical filtered process}, in symbols $\fp X \in \CFP$, if
		\begin{equation} \label{eq:CFP}
			\fp X = \left( \Z, \F^\Z, \unfold_1(\bar{\mu}), (\F_t^\Z)_{t = 1}^N, Z^- \right),
		\end{equation}
		where $(\Z,\F^\Z,(\F_t^{\Z})_{t = 1}^N)$ is the canonical filtered space, see \eqref{eq:CFS},  $Z^-$ is the evaluation map \eqref{def:evaluation map}, and $\bar{\mu} \in \Pc(\Z_1)$. 
		As usual we write $\CFP_p$ for the subset of processes whose laws have finite $p$-th moment.
	\end{definition}

	Using the concept of information process we can associate to an arbitrary  filtered process  a unique element in $\CFP$:
	\begin{definition}[Associated canonical filtered processes]
		\label{def:CFP.associated.to.FP}
		Let $\fp X\in\FP_p$, and let $\overline{\fp X}\in\CFP_p$ be given by \eqref{eq:CFP} with $\bar{\mu}:=\Law(\ip_1 (\fp X))$, or, according to Lemma \ref{lem:unfoldprops.measures}, equivalently
		\begin{equation}
			\label{eq:aCFP}
			\overline{\fp X} = \left( \Z, \F^\Z, \Law(\ip(\fp X)), (\F^\Z_t)_{t = 1}^N, Z^- \right).
		\end{equation}
		We call $\overline{\fp X}$ the \emph{canonical filtered process associated to} $\fp X$.
	\end{definition}

	We want to stress at this point that, as stated in \eqref{eq:aCFP}, all information of $\Law(\ip(\fp X)) \in \Pc(\Z)$ is already contained in $\Law(\ip_1(\fp X)) \in \Pc(\Z_1)$ by Lemma \ref{lem:unfoldprops.measures}.
	The relations between $\CFP_p$ and $\Pc_p(\Z_1)$ become apparent in Theorem \ref{thm:isometry} below, and the relation of filtered processes to their canonical counterparts becomes apparent through Lemma \ref{lem:process to canonical}.

\begin{lemma} 
	\label{lem:process to canonical}
	Let $\fp X, \fp Y \in \FP_p$, and let $\overline{\fp X},\overline{\fp Y} \in \CFP_p$ be their associated canonical processes.
	The following hold.
	\begin{enumerate}[label=(\roman*)]
		\item \label{it:process to canonical1}
		$(\id, \ip(\fp X))_\ast \P^\fp X \in \cplba(\fp X,\overline{\fp X})$.
		\item \label{it:process to canonical2}
		If $\pi \in \cpla(\fp X,\fp Y)$, then $(\ip(\fp X),\ip(\fp Y))_\ast \pi \in \cpla(\overline{\fp X}, \overline{\fp Y})$.
		\item \label{it:process to canonical3}
		If $\pi \in \cplba(\fp X,\fp Y)$, then $(\ip(\fp X),\ip(\fp Y))_\ast \pi \in \cplba(\overline{\fp X},\overline{\fp Y})$.
	\end{enumerate}
\end{lemma}

\begin{proof}
	To show \ref{it:process to canonical1}, we write $\gamma := (\id, \ip(\fp X))_\ast \P^\fp X$ and first check causality of $\gamma$ using \ref{it:causal CI2} of the characterization of causality given in Lemma \ref{lem:causal CI}. 
	To that end, let $V\colon \Z \to \R$ be bounded and $\F_t^\Z$-measurable.
	From the definition of $\gamma$ we see that $\gamma$-almost surely $V = V(\ip(\fp X))$.
	Thus $V(\ip(\fp X))$ is $\F_t^\fp X$-measurable and causality of $\gamma$ from $\fp X$ to $\overline{\fp X}$ follows from 
	\begin{align*}
		\E_\gamma\left[ V \middle| \F_{N,0}^{\fp X,\Z} \right] &= \E_\gamma\left[ V(\ip(\fp X)) \middle| \F_{N,0}^{\fp X,\Z} \right] = V(\ip(\fp X))
	= \E_\gamma\left[ V(\ip(\fp X)) \middle| \F_{t,0}^{\fp X,\Z} \right] = \E_\gamma\left[ V \middle| \F_{t,0}^{\fp X,\Z} \right].
	\end{align*}
	
	To see causality of $\gamma$ from $\overline{\fp X}$ to $\fp X$, we will again use Lemma \ref{lem:causal CI}, this time item \ref{it:causal CI1}.
	Let $U \colon \Z \to \R$ be bounded and $\F_N^\Z$-measurable.
	Again, due to the structure of $\gamma$ it is readily verified that $\gamma$-almost surely $U = U(\ip(\fp X))$ and
	\begin{align}
		\label{eq:idipcpl1}
		\E_\gamma\left[ U \middle| \F_{t,t}^{\fp X,\Z}\right] 
		&= \E_\gamma \left[ U \middle| \F_{t,0}^{\fp X,\Z} \right],
	\\	\label{eq:idipcpl2}
	\E_\gamma\left[ U \middle| \F_{0,t}^{\fp X,\Z}\right] &= \E_\gamma \left[ U \middle| \ip_{1:t}(\fp X) \right].
	\end{align}
	By the self-awareness property of the information process, see Lemma \ref{lem:ip.self.aware}, \eqref{eq:idipcpl1} and \eqref{eq:idipcpl2} we find
	\begin{align*}
		\E_\gamma\left[ U \middle| \F_{t,t}^{\fp X,\Z} \right] &= \E_\gamma\left[ U(\ip(\fp X)) \middle| \F_{t,0}^{\fp X,\Z} \right] = \E_\gamma\left[ U(\ip(\fp X)) \middle| \ip_{1:t}(\fp X) \right]
	\\	&= \E_\gamma\left[ U(\ip(\fp X)) \middle| \F_{0,t}^{\fp X,\Z} \right] = \E_\gamma\left[ U \middle| \F_{0,t}^{\fp X,\Z} \right],
	\end{align*}
	which completes the proof of item \ref{it:process to canonical1}.

	To verify \ref{it:process to canonical2}, we write $\overline{\pi} := (\ip(\fp X),\ip(\fp Y))_\ast \pi$, $\eta := (\id,\ip(\fp X),\ip(\fp Y))_\ast \pi$, and let $U \colon \Z \to \R$ be as above.
	By Lemma \ref{lem:causal CI} and Lemma \ref{lem:ip.self.aware}, we have $\eta$-almost surely
	\begin{equation}
		\label{eq:idipcpl3}
		\E_{{\pi}} \left[U(\ip(\fp X)) \middle| \F_{t,t}^{\fp X,\fp Y}\right] = \E_\pi \left[  U(\ip(\fp X)) \middle| \F_{t,0}^{\fp X,\fp Y} \right] =  \E_\pi \left[  U(\ip(\fp X)) \middle| \ip_{1:t}(\fp X) \right].
	\end{equation}
	Using \eqref{eq:idipcpl3} yields $\eta$-almost surely
	\[
		\E_{\overline{\pi}} \left[ U \middle| \F_{t,0}^{\Z,\Z} \right] = \E_{\pi} \left[ U(\ip(\fp X)) \middle| \F_{t,t}^{\fp X,\fp Y} \right].
	\]
	We conclude by the tower property $\eta$-almost surely
	\begin{align*}
		\E_{\overline{\pi}} \left[ U \middle| \F_{t,0}^{\Z,\Z} \right] 
		&= \E_\pi \left[ \E_\pi \left[ U(\ip(\fp X)) \middle| \F_{t,t}^{\fp X,\fp Y} \right] \middle| \ip_{1:t}(\fp X), \ip_{1:t}(\fp Y) \right]
	\\	
	&= \E_\pi\left[ U(\ip(\fp X)) \middle| \ip_{1:t}(\fp X), \ip_{1:t}(\fp Y)  \right] 
	= \E_{\overline{\pi}} \left[ U \middle| \F_{t,t}^{\Z,\Z} \right],
	\end{align*}
	which completes the proof of \ref{it:process to canonical2}.

	Finally, for symmetry reasons \ref{it:process to canonical2} implies \ref{it:process to canonical3}.
\end{proof}

\subsection{The isometry} \label{ssec:isometry}

Based on the preparatory work from the preceding subsection, we are able to establish Theorem \ref{thm:isometry}.
From this, we derive that $\AW_p$ naturally induces a complete metric on the factor space $\FFP_p$, and that $\FFP_p$ is isometrically isomorphic to the (classical) Wasserstein space $(\Z_1,\W_p)$, thereby establishing Theorem \ref{WREP}.

\begin{theorem} \label{thm:isometry}
	Let $\fp X,\fp Y \in \FP_p$ and let $\overline{\fp X},\overline{\fp Y} \in \CFP_p$ be the associated canonical processes.
	Then  \begin{equation} \label{eq:isometry}
		\AW_p\left( \fp X, \fp Y \right) 
		= \AW_p\left( \cfp X, \cfp Y\right) 
		= \W_p\left(\Law( \ip_1(\fp X)), \Law( \ip_1(\fp Y) )\right).
	\end{equation}
	In particular, $\AW_p$ is a pseudo-metric on $\FP_p$ and the embedding $\cfp X \mapsto \Law(\ip_1(\cfp X))$ is an isometric isomorphism of $\CFP_p$ and $\Pc_p(\Z_1)$.
\end{theorem}

\begin{proof}
	The first equality in \eqref{eq:isometry} is a direct consequence of Lemma \ref{lem:process to canonical} and Theorem \ref{thm:adapted block} in the Appendix.
	For the convenience of the reader, we present an alternative proof of the first equality under the assumption that the probability spaces of $\fp X$ and $\fp Y$ are Polish, thereby omitting the technical result in Theorem \ref{thm:adapted block}.

	By Lemma \ref{lem:process to canonical} \ref{it:process to canonical3} we find
	\[
		\AW_p(\fp X,\fp Y) \ge \AW_p(\cfp X, \cfp Y).
	\]
	To see the reverse inequality, let $\overline{\pi} \in \cplba(\cfp X,\cfp Y)$ and write
	\[
		\gamma := \left( \id, \ip(\fp X) \right)_\ast \P^\fp X \text{ and }\hat{\gamma} := \left( \id, \ip(\fp Y) \right)_\ast \P^\fp Y.
	\]
	These couplings are bicausal by Lemma \ref{lem:process to canonical} \ref{it:process to canonical1} and admit disintegrations $(\gamma_{z})_{z \in \Z}$ and $(\hat{\gamma}_{\hat{z}})_{\hat{z} \in \Z}$ since the considered probability spaces are Polish by assumption.
	Consider the probability
	\[
		\pi(d\omega,d\hat{\omega}) := \int \gamma_z(d\omega) \hat{\gamma}_{\hat{z}}(d\hat{\omega}) \, \overline{\pi}(dz,d\hat{z}).
	\]
	For symmetry reasons we will only show that $\pi$ is causal from $\fp X$ to $\fp Y$.
	By Lemma \ref{lem:causal CI} it suffices to show that for any bounded, $\F_t^\fp Y$-measurable $V$ we have
	\[
		\E_\pi\left[ V \middle| \F_{N,0}^{\fp X,\fp Y} \right]	
		= \E_\pi \left[ V \middle| \F_{t,0}^{\fp X,\fp Y} \right].
	\]
	As $\hat \gamma$ is bicausal, Lemma \ref{lem:causal CI} asserts that
	$
		\hat z \mapsto 	\int V(\hat \omega) \, \hat \gamma_{\hat z}(d\hat \omega)
	$
	is $\F_{t}^{\cfp Y}$-measurable. By the same reasoning, we obtain $\F_t^\fp X$-measurability of
	\[
		\omega \mapsto W(\omega) := \iiint V(\hat \omega) \, \hat \gamma_{\hat z}(d \hat \omega) \, \overline{\pi}_z(d\hat z) \, \gamma_{\omega}(dz),
	\]
	where $\gamma_\omega(dz) := \Law(\ip(\fp X) | \F_{N}^{\fp X})(\omega)$.
	Hence, by the definition of $\pi$ and the tower property we get
	\[
		\E_\pi \left[ V \middle| \F_{N,0}^{\fp X,\fp Y} \right] = W = \E_\pi\left[ V \middle| \F_{t,0}^{\fp X,\fp Y} \right].
	\]
	As $V$ was arbitrary, this yields  $\pi \in \cpla(\fp X,\fp Y)$ and by symmetry $\pi \in \cplba(\fp X, \fp Y)$.
	Moreover, we have $\E_\pi[d^p(X,Y)] = \E_{\overline\pi}[d^p(\overline{X},\overline{Y})]$ and conclude that $\AW_p(\fp X,\fp Y) = \AW_p(\cfp X,\cfp Y)$.

	It remains to show the second equality.
	Write $\mu := \Law(\ip_1(\fp X))$ and $\nu:=\Law(\ip_1(\fp Y))$.
	By Lemma \ref{lem:bicausal cfp rep}, we have that 
	\begin{align*} 
	\AW_p^p(\cfp X,\cfp Y)
	=\inf_{\pi_1 \in \cpl(\mu,\nu)} \inf_{(k_t)_{t=1}^{N-1}} \int \sum_{t=1}^N d^p(z_t^-,\hat z_t^-)  \, (\pi_1\otimes k_1 \otimes \ldots \otimes k_{N-1})(dz,d\hat z),
	\end{align*}
	where the second infimum is taken over all kernels 
	\begin{equation}
		\label{eq:kernel}
 		k_t \colon \Z_{1:t} \times \Z_{1:t} \to \Pc_p(\Z_{t+1} \times \Z_{t+1}) 
		\text{ with } 
		k_t^{z_{1:t},\hat z_{1:t}} \in \cpl(z_t^+,\hat z_t^+).
	\end{equation}
	Now, for every $1\leq t\leq N-1$, let $k^\ast_t$ be a kernel as in \eqref{eq:kernel} that is an optimal coupling (w.r.t.\ $\W_p$) between its marginals.
	Their existence follows from a standard measurable selection argument.
	Then, for every $1\leq t \leq N-1$, $z_{1:t},\hat z_{1:t}\in \Z_{1:t}$, and every kernel $k_t$ as in \eqref{eq:kernel}, we have that
	\begin{align}
	\label{eq:ddp.optimal.kernel}
	\begin{split}
	d^p(z_t, \hat z_t)
	&=d^p(z_t^-,\hat z_t^-) + \W_p^p(z_t^+, \hat z_t^+) \\
	&\leq d^p(z_t^-,\hat z_t^-) + \int d^p(z_{t+1},\hat z_{t+1})\, k_t^{z_{1:t},\hat z_{1:t}}(dz_{t+1},d\hat z_{t+1}) 
	\end{split}
	\end{align}
	with equality if $k_t=k_t^\ast$.
	In particular, for every $\pi_1\in\cpl(\mu,\nu)$, an iterative application of \eqref{eq:ddp.optimal.kernel} shows that
	\begin{align*}
	\int d^p(z_1, \hat z_1) \, \pi_1(dz_1,d\hat{z}_1)
	&\leq \int \sum_{t=1}^N d^p(z_t^-,\hat z_t^-) \, (\pi_1\otimes k_1 \otimes \ldots \otimes k_{N-1})(dz,d\hat z) 
	\end{align*}
	with equality if $k_t=k_t^\ast$ for every $1\leq t\leq N-1$.
	Optimizing over $\pi_1\in\cpl(\mu,\nu)$ yields the claim.
\end{proof}

\begin{definition}[Wasserstein space of stochastic processes]
	We call the quotient space 
	\[\FFP_p:=\FP_p /_{\AW_p} \]
	the Wasserstein space of stochastic processes.
	$\FFP_p$ is equipped with $\AW_p$ (which is by Theorem \ref{thm:isometry} well-defined on $\FFP_p$ independent of the choice of representative).
\end{definition}

Our canonical choice of a representative of $\fp X \in \FFP_p$ is the associated canonical process $\cfp X \in \CFP_p$.
From now on, whenever we use the probability space of (the equivalence class of filtered processes) $\fp X$, we refer to the filtered probability space provided by $\cfp X$ if not stated otherwise.

\begin{corollary}
	\label{cor:weak lsc}
	The map $(\fp X, \fp Y) \mapsto \AW_p(\fp X, \fp Y)$ is lower semicontinuous w.r.t.\ the weak adapted topology\footnote{The weak adapted topology relates to the adapted Wasserstein distance the same way the weak topology of measures relates to the Wasserstein distance: The weak adapted topology is induced by $\AW_p$ when, for each $1 \le t \le N$, the metric $d_{\X_t}$ is replaced by the bounded metric $d_{\X_t} \wedge 1$.} on $\FFP_p \times \FFP_p$.
\end{corollary}

\begin{proof}
	The result follows from combining \cite[Corollary 6.11]{Vi09}, that is the `non-adapted' analogon of Corollary \ref{cor:weak lsc} from the classical OT theory, with Theorem \ref{thm:isometry}.
\end{proof}

\section{Adapted functions and the prediction process}
\label{sec:AF.PP}

This section relates the adapted Wasserstein distance to the existing concepts of prediction processes and adapted functions introduced by Knight \cite{Kn75},  Aldous \cite{Al81}, and Hoover and Keisler \cite{HoKe84}. 
The main result of this section, Theorem \ref{thm:information.process.AF} below, shows that all concepts induce the same relation on  filtered processes.

Before recalling the definition of adapted functions from \cite{Ho87} (see also \cite{HoKe84}), let us say that, intuitively, an adapted function is an operation that takes a filtered processes as argument and returns a random variable defined on the underlying probability space of this filtered process.
Simple examples of adapted functions are  $\fp X\mapsto \sin(X_1)$ and $\fp X\mapsto \E[\min\{\exp(X_2X_4), 1\}|\F^{\fp X}_2]$.

\begin{definition}[Adapted functions]
\label{def:AF}
We call $f$ an \emph{adapted function
-- we write $f\in \AF$ -- if it can be built using the following three operations:}
 	\begin{enumerate}[label=(AF\arabic*)]
 	\item \label{it:AF1} If $\Phi\colon \mathcal{X} \to\mathbb{R}$ is continuous bounded, then $\Phi\in \AF$; we set $\Phi(\fp X):=\Phi(X)$.
 	\item \label{it:AF2} If $m\in\mathbb{N}$, $f_1,\dots,f_m\in \AF$, and $\varphi \in C_b(\R^m)$, then $\varphi(f_1, \ldots, f_m)\in \AF$; we set  \linebreak $\varphi(f_1, \ldots, f_m)(\fp X):=\varphi(f_1(\fp X), \ldots, f_m(\fp X))$.
 	\item \label{it:AF3} If $1\leq t \leq N$ and $g\in\AF$, then  $(g|t)\in \AF$; we set $(g|t)(\fp X) := \E[g(\fp X) | \F^{\fp X}_t]$.
 	\end{enumerate}	
 	Further define the \emph{rank} of an adapted function inductively as follows: the rank of $\Phi$ is 0;  the rank of $\varphi(f_1,\dots,f_m) $ is the maximal rank of $f_1,\dots, f_m$; and the rank of  $(g|t)$ is the rank of $g$ plus 1.
 	
 	The set of all adapted functions of rank at most $n \in \N \cup \{0\}$ is denoted by $\AF[n]$.
\end{definition}

Moreover, we can naturally embed $\AF[n]$ into $\AF[n+1]$ by identifying $f\in\AF[n]$ with $(f|N)$, since $f(\fp X) = (f|N)(\fp X)$ for all $\fp X \in \FP_p$.
Consequently, we may assume without loss of generality in item \ref{it:AF2} that $f_1,\ldots,f_m$ all have the same rank.

Adapted functions were defined in \cite{HoKe84} in a continuous time setting.
The present discrete time setting permits to give the following, perhaps clearer, representation:
 
 \begin{lemma} \label{lem:rep.AF}
 	Let $f \in \AF$ and $n\in\N$.
 	Then
			$f \in \AF[n]$ if and only if  for every  $k=1,\dots,N$ there is $m_k\in\N$ and an $m_k$-dimensional vector $\vec{g}_k$ consisting of elements in $\AF[n-1]$, and there is $F\in C_b(\mathbb{R}^{\sum_{k=1}^N m_k})$ such that
			\begin{equation} \label{eq:rep.AF}
			f(\fp{X}) = F \big( \E[\vec{g}_1(\fp{X})|\F_1^\fp{X}],\dots,\E[\vec{g}_N(\fp{X})|\F_N^\fp{X}]\big)\quad \text{for all } \fp X \in \FP_p.
			\end{equation}
 \end{lemma}

 \begin{proof}
 	It turns out to be useful to keep track of the depth of an adapted function, a notion that we now introduce:
	Loosely speaking, for $f \in \AF[n]$, its depth is the number of times \ref{it:AF2} was applied to a \emph{base element} of the form $(g|t)$ with $g \in \AF[n-1]$.
 	The \emph{depth} (at rank $n$) of a `base element' $(g|t)$ with $g \in \AF[n-1]$ is defined as $0$, i.e.,
	\begin{equation} 
		\label{eq:depth of base}
		\depth((g|t)) := 0\quad \text{for all }g\in\AF[n-1] \text{ and }1 \le t \le N.
	\end{equation}
	Recursively, we assign to $\varphi(f_1, \ldots, f_m) = f \in \AF[n]$ with $\phi \in C_b(\R^m)$ and $f_i \in \AF[n]$, $i = 1,\ldots,m$, its depth
 	\begin{equation}
		 \label{eq:depth else}
		 \depth(f) := \max_{i = 1,\ldots,m} \depth(f_i) + 1.
	 \end{equation}
	Note that by the iterative construction of any formation $f \in \AF[n]$, $f$ is either a base element or of the form detailed in \ref{it:AF2}, and its depth is well-defined by \eqref{eq:depth of base} and \eqref{eq:depth else}.
		 
 	Let $f \in \AF[n]$.
	We begin the induction at depth $0$.
	Then $f$ has $\depth(f) = 0$ if and only if it is a base element, in which case \eqref{eq:rep.AF} holds true.
	Now assume that $k:=\depth(f)>0$ and that \eqref{eq:rep.AF} applies to all $g \in \AF[n]$ with $\depth(g)<k$.
	We write $f =\phi(f_1, \ldots, f_m)$ where $\phi \in C_b(\R^m)$ and all $f_i \in \AF[n]$ have depth less than $k$.\
	By the inductive hypothesis, for every $1 \le i \le m$ there are vectors $\vec{g}_1^i,\ldots,\vec{g}_N^i$ consisting of elements in $\AF[n-1]$, and $F^i \in C_b(\R^{m^i})$ such that
	\[
		f_i(\fp X) =  F^i\left( \E\left[ \vec{g}^i_1(\fp X) | \F^\fp X_1 \right], \ldots, \E\left[ \vec{g}^i_N(\fp X) | \F^\fp X_N \right] \right)\quad\text{for all }\fp X \in \FP_p.
	\]
 	Collecting and sorting all the terms of the vectors $\vec{g}^i_t$ for $1 \le t \le N$ gives
	\[
		\vec{g}_t := \vec{g}^{1:m}_t = (\vec{g}_t^1,\ldots, \vec{g}_t^m).
	\]
	Finally, let $\sigma$ be the permutation with the property
	\[
	 	\sigma(\vec{g}_1,\ldots,\vec{g}_N) = (\vec{g}^1,\ldots,\vec{g}^m),
	\]
	then $F = \phi \circ (F^1,\ldots,F^k) \circ \sigma$ together with $(\vec{g}_1,\ldots,\vec{g}_N)$ satisfies \eqref{eq:rep.AF}.
 \end{proof}

 \begin{definition}[Adapted distribution]\label{def:adapted.distribution}
 	Two filtered processes $\fp{X},\fp{Y}\in\FP_p$ have the same \emph{adapted distribution (of rank $n\geq 0$)} if $\E[f(\fp{X})]=\E[f(\fp{Y})]$ for every adapted function $f\in\AF$ (resp.\ $f \in \AF[n]$);
 	we write $\fp{X}\sim_\infty\fp{Y}$ (resp.\ $\fp{X}\sim_n\fp{Y}$).
 \end{definition}

 \begin{remark} \label{rem:AF.lipschitz.Borel}
 	In the definition of adapted functions, we started in \ref{it:AF1} with the base set of continuous and bounded functions from $\X$ to $\R$.
 	It is possible to vary this base set, without changing the induced equivalence relations $\sim_\infty$ and $\sim_n$, see Definition \ref{def:adapted.distribution}.
	One may replace \ref{it:AF1} with any of the following choices:
 	\begin{enumerate}[label=(AF1\alph*)]
 	\item \label{it:AF1a} if $\Phi\colon \mathcal{X} \to\mathbb{R}$ is bounded and Borel measurable, then $\Phi\in \AF$;
 	\item if $\Phi\colon \mathcal{X} \to\mathbb{R}$ is bounded and Lipschitz continuous, then $\Phi\in \AF$;
 	\item if $1\leq t\leq N$ and $\Phi\colon \mathcal{X}_t \to\mathbb{R}$ is bounded and continuous, then $\Phi\circ\proj_t\in \AF$.
 	\end{enumerate}
 	In a similar manner, we may consider in \ref{it:AF2} solely Lipschitz continuous / Borel measurable and bounded $\phi$, and still preserve the equivalence relations introduced in Definition \ref{def:adapted.distribution}.
	We shall prove this further down below.
 \end{remark}

	The purpose of the next example is twofold: 
	first, to show where adapted distributions and adapted functions naturally appear, and also to familiarize the reader with the latter.
 	
 \begin{example}[Martingales and optimal stopping]
 \label{ex:martingales.opt.stop}
 	Let $\fp X,\fp Y\in\FP$.
 	\begin{enumerate}[label=(\alph*),wide]
 	\item \label{it:ex.martingale}	If $\fp X$  is a martingale $\fp Y\sim_1\fp X$, then so is $\fp Y$ as already observed in \cite{Al81}.
 	Indeed, for $1\leq t\leq N$,
 	\[ f_t:=  | \proj_{t} - (\proj_{t+1}|t)| \]
	is an element of $\AF[1]$.\footnote{
 	For demonstrative purposes, we disregard that only bounded functions are allowed in the definition of adapted functions.
 	Indeed, this is only a technical issue and all terms are well-defined by standard approximation arguments.}
 	Its evaluation yields
 	\[ \E\left[\left|Y_t- \E[Y_{t+1}|\F^\fp Y_t]\right| \right]
 	= f_t(\fp Y) 
 	= f_t(\fp X)
 	=\E\left[\left|X_t- \E[X_{t+1}|\F^\fp{X}_t]\right| \right] = 0,\]
 	that is the martingale property of $\fp Y$.

	\item Another important property preserved by $\sim_1$ is Markovianity.
	This also holds for the important property of being `plain' defined in \eqref{eq:plain} below.
 	\item \label{it:ex.opt.stop}	Let $c\colon\mathcal{X}\times\{1,\dots,N\}\to\mathbb{R}$ be nonanticipative.\footnote{That is, $c_t(x)=c(x,t)$ depends only on $x_{1:t}$ when $1\leq t\leq N$.}
 	By the Snell-envelope theorem we have
 	\[ v_c(\fp X)
 	:=\inf_{\tau \text{ is } (\mathcal{F}^{\fp X}_t)_{t=1}^N\text{-stopping time}} \E[c_\tau(X)] 
 	= \E[S_1],\]
 	where $S_1$ is defined by backward induction starting with  $S_N:=c_N(X)$ and
 	\[ S_t:= c_t(X) \wedge \E[S_{t+1}|\F^\fp{X}_t] \quad \text{for }t=N-1,\ldots,1.\]
 	Thus, each $S_t$ equals the value of an adapted function of rank $N-t$, from where it follows that $\fp X\sim_{N-1}\fp Y$ implies $v_c(\fp X)=v_c(\fp Y)$.
 	\end{enumerate}

 	We will come back to \ref{it:ex.martingale} and \ref{it:ex.opt.stop} in Section \ref{sec:aspects}.
 \end{example}

 	Closely related to adapted functions is the \emph{prediction process}:
	\begin{definition}[Prediction process]\label{def:pred.process}
		For $\fp X \in \FP_p$
the \emph{first order prediction processes} is given by
		\begin{align*}
		\pp^1(\fp{X})\colon\Omega^\fp X &\to \mathcal{M}_1:=\Pc_p(\X)^N, \\
		\omega&\mapsto \left(\Law(X|\F_t^\fp{X})(\omega)\right)_{t=1}^N
		\end{align*}
		Iteratively, the \emph{$n$-th order prediction process} is  given by
		\begin{align*}
		\mathrm{pp}^{n}(\fp{X})\colon\Omega^\fp X&\to\mathcal{M}_{n}:=\mathcal{P}_p(\mathcal{M}_{n-1})^N,\\
		\omega &\mapsto\left(\Law( \pp^{n-1}(\fp{X}) |\F_t^\fp{X})(\omega)\right)_{t=1}^N.
		\end{align*}
		Finally, the \emph{prediction process} is defined as the $\bigtimes_{n\in\mathbb{N}} \mathcal{M}_n=:\mathcal M$-valued random variable
		\[\pp(\fp{X}):=(\pp^n(\fp{X}))_{n\in\mathbb{N}}.\]
	\end{definition}
		For convenience, we set the zero-th order prediction process $\pp^0(\fp{X}):=X$ and $\mathcal{M}_0:=\mathcal{X}$ so that the iterative scheme of Definition \ref{def:pred.process} is valid for $n\geq 0$.
	
\begin{lemma} \label{lem:pp is continuous function of ip}
		For every $n \in \N$ there is a continuous function $F^n \colon \Z \to \mathcal M_n$ such that
		\[
			\pp^n(\fp X) = F^n(\ip(\fp X))	\quad \text{for all } \fp X \in \FP_p.
		\]
\end{lemma}

 \begin{proof}
	 Let $F^0 \colon \Z \to \mathcal M_0 = \X$ be the corresponding projection, that is $F^0(z) = z_{1:N}^-$.
	 Therefore, $F^0$ is continuous with $\pp^0(\fp X) = X = F^0(\ip(\fp X))$.
	 
	 Let $n \in \N$, and consider the inductive hypothesis that there is a continuous map $F^{n-1}\colon\Z\to\mathcal{M}_{n-1}$ such that
	 \begin{equation}\label{eq:pp inductive ass}
		 \pp^{n-1}(\fp X) = F^{n-1}(\ip(\fp X))\quad\text{for all } \fp X \in \FP_p.
	 \end{equation}
	 By definition of the $n$-th order prediction process and the hypothesis \eqref{eq:pp inductive ass}, we have
	 \[
	 	\pp^n(\fp X) = \left( \Law\left( F^{n-1}(\ip(\fp X)) | \F^{\fp X}_t \right) \right)_{t = 1}^N\quad \text{for all } \fp X \in \FP_p.
	 \]
	 By Lemma \ref{lem:ip.self.aware}, for every $1 \le t \le N$, there is a continuous map $G_t^n \colon \Z \to \Pc_p(\mathcal M_{n-1})$ such that
	 \[
	 	\Law\left( F^{n-1}(\ip(\fp X)) | \F^{\fp X}_t   \right) = G^n_t(\ip(\fp X))\quad\text{for all }\fp X \in \FP_p.
	 \]
	 The proof is completed be setting $F^n \colon \Z \to \mathcal M_n = \Pc_p(\mathcal M_{n-1})^N$ to be $F^n:=(G^n_1,\ldots,G^n_N)$.
 \end{proof}

 It is worth pointing out that $\mathrm{pp}^{n}$ contains `at least as much information' as its predecessor $\mathrm{pp}^{n-1}$.
 In fact, we shall later see in Proposition \ref{prop:process.different.rank} that for $n<N-1$, it contains strictly more information in general.

 \begin{lemma} \label{lem:ppn contains ppk}
	 Let $1\leq k\leq n$.
	 There is a 1-Lipschitz function $F \colon \mathcal M_n \to \mathcal M_k$ such that
	 \[
		F(\pp^n(\fp X)) = \pp^k(\fp X)\quad\text{for all }\fp X \in \FP_p.
	 \]
	 In particular, if $\fp X,\fp Y \in \FP_p$  are such that $\pp^n(\fp X)$ and $\pp^n(\fp Y)$ have the same distribution, then $\pp^k(\fp X)$ and $\pp^k(\fp Y)$ have the same distribution as well.
 \end{lemma}
 \begin{proof}
	For $m \geq 0$ consider the isometric injections
	\begin{align*}
		\iota_{m} \colon \mathcal M_{m} &\to \Pc_p(\mathcal M_{m}),\qquad 
		p \mapsto \delta_p.
	\end{align*}
	For $m\geq 1$, since $\pp^{m}_N(\fp X)=\delta_{\pp^{m-1}(\fp X)}$, we may apply $\iota_{m-1}^{-1} \circ \proj_N$ to $\pp^m$, and obtain $ \pp^{m-1}(\fp X)=\iota_{m-1}^{-1}(\pp^m_N(\fp X))$ for all $\fp X \in \FP_p$.
	Moreover, $\iota_{m}^{-1}$ admits a 1-Lipschitz extension $I_m^{-1}\colon \Pc_p(\mathcal M_m) \to \mathcal M_m$ given by
	\[
		I_m^{-1}(P) := \left( \int p_1 \, P(dp), \ldots, \int p_N \, P(dp) \right),
	\]
	where we write $p = (p_1,\ldots,p_N) \in \mathcal M_m = \Pc_p(\mathcal M_{m-1})^N$.
	Indeed, $I_m^{-1}$ is $1$-Lipschitz as, for $P,Q \in \Pc_p(\mathcal M_m)$, we have by Jensen's inequality
	\begin{align*}
		\WA_{\Pc_p(\mathcal M_m)}^p\left( P, Q \right) &= \int \sum_{t = 1}^N \WA_{\Pc_p(\mathcal M_{m-1})}^p(p_t, q_t) \, \pi^\ast(dp, dq) \\
		&\ge \sum_{t = 1}^N \WA_{\Pc_p(\mathcal M_{m-1})}^p\left( \int p_t \, \pi^\ast (dp,dq), \int q_t \, \pi^\ast(dp,dq) \right) = d_{\mathcal M_m}^p\left( I_m^{-1}(P), I_m^{-1}(Q) \right),
	\end{align*}
	where $\pi^\ast$ is an $\WA_{\Pc_p(\mathcal M_m)}$-optimal coupling of $P$ and $Q$.
	We denote by $F^m \colon \mathcal M_m = \Pc_p(\mathcal M_{m-1})^N \to \mathcal M_{m-1}$ the composition $I_{m-1}^{-1} \circ \proj_N$.

	Finally, the mapping
	\[ F :=F^{k+1} \circ \ldots \circ F^{n} \colon \mathcal M_n \to \mathcal M_{k} \]
	 is 1-Lipschitz and satisfies
	 \begin{align*}
	 	F(\pp^n(\fp X)) 
	 	&= F^{k+1} \circ \ldots \circ F^{n-1}(\pp^{n-1}(\fp X)) = \ldots = \pp^k(\fp X)
	 \end{align*}
	 for all $\fp X \in \FP_p$.
	 This completes the proof.
 \end{proof}

 Let us remark that, when $n\geq 1$, the process $\mathrm{pp}^n(\fp{X})$ is a measure-valued martingale w.r.t.\ $(\F^\fp{X}_t)_{t=1}^N$ which is terminating at $\delta_{\mathrm{pp}^{n-1}(\fp X)}$.

 A version of the next lemma can be found in \cite{HoKe84},  though the proof is different due to differences in the definition of adapted functions (as multi-time stochastic processes).

 \begin{lemma}
 \label{lem:rank.adapted.functions.prediction.process.same}
 	Let $\fp{X},\fp{Y}\in\FP_p$ and let $n\in\N$. 
 	Then the following are equivalent:
 	\begin{enumerate}[label=(\roman*)]
 	\item \label{it:rank.adapted.functions.prediction.process.same1}
 	$\fp{X}\sim_n \fp{Y}$;
 	\item \label{it:rank.adapted.functions.prediction.process.same2} 
 	$\mathrm{pp}^n(\fp{X})$ and $\mathrm{pp}^n(\fp{Y})$ have the same distribution.
 	\end{enumerate}
 \end{lemma}

 Before proving the lemma, we want to point out that the same proof with obvious modifications also works to obtain Remark \ref{rem:AF.lipschitz.Borel}.
 For example, replace at every instance `continuous' with `Borel-measurable' for \ref{it:AF1a}.

 \begin{proof}
 	We start with the easier direction that \ref{it:rank.adapted.functions.prediction.process.same2} implies \ref{it:rank.adapted.functions.prediction.process.same1}.
 	Clearly, it is sufficient to show:
	
	\emph{Claim}: For every $f\in\AF[n]$ there is $F\in C_b(\mathcal{M}_n)$ such that 
 	\begin{equation} \label{eq:rank.AF.pp.claim1}
		f(\fp X)=F(\pp^n(\fp X)) \quad\text{for all } \fp X \in\FP_p.
	 \end{equation}
 	
 	For $n=0$ the claim is trivially true as $\mathcal M_0 = \X$ and due to item \ref{it:AF1}.
 	Assume now that the claim is true for $n \in \N \cup \{0\}$, and let $f\in\AF[n+1]$.
 	Using Lemma \ref{lem:rep.AF}, we may represent $f$ as in \eqref{eq:rep.AF}.
 	Thus, it suffices to show \eqref{eq:rank.AF.pp.claim1} for $f=(g|t)$ where $g\in\AF[n]$ and $1\leq t\leq N$.
 	By the inductive hypothesis there is $G\in C_b(\mathcal{M}_{n})$ such that $g(\fp X)=G(\mathrm{pp}^{n}(\fp X))$ for all $\fp X \in\FP_p$.
	Therefore
	 	\begin{align*}
 		f(\fp{X}) 
 		&= \E[G(\mathrm{pp}^{n}(\fp{X}))|\F_t^\fp{X}] = \int G \,d \pp^{n+1}_t(\fp X)
 		= H(\mathrm{pp}^{n+1}(\fp{X}))
 	\end{align*}
 for all $\fp X \in \FP_p$, where $H \colon \mathcal{M}_{n+1}\to\R$, $p\mapsto \int G\,d p_t$ is continuous and bounded.
 This shows \eqref{eq:rank.AF.pp.claim1} and thus that \ref{it:rank.adapted.functions.prediction.process.same2} implies \ref{it:rank.adapted.functions.prediction.process.same1}.

 	We proceed to show that \ref{it:rank.adapted.functions.prediction.process.same1} implies \ref{it:rank.adapted.functions.prediction.process.same2}.
 	To that end, we interject two preliminary statements.
 	Define $\mathcal{S}_0:= C_b(\mathcal{M}_0)$ and inductively define $\mathcal{S}_{n}\subset C_b(\mathcal{M}_{n})$ as the set of all functions of the form
 	\begin{equation}\label{eq:form.function.Sn}
 			p \mapsto \psi \left(\int \vec{G}_1\,dp_1,\dots,\int \vec{G}_N\,dp_N\right),
 	\end{equation}
 	where, for every $1\leq t \leq N$, $\vec{G}_t$ is a vector of functions in $\mathcal{S}_{n-1}$ and $\psi\colon\mathbb{R}^m\to\mathbb{R}$ (with adequate $m$) is continuous and bounded.

	\emph{Claim}: $\mathcal{S}_n$ is an algebra which separates points in $\mathcal{M}_n$.

 	As usual, we proceed by induction.
	The claim follows trivially for $n = 0$, since $\mathcal S_0=C_b(\X)$.
 	Assume now that the claim is true for $n$.
 	Clearly, $\mathcal{S}_{n+1}$ is an algebra. 
 	To see the second part of the claim, namely that it separates points, let $p=(p_t)_{t=1}^N$ and $q=(q_t)_{t=1}^N$ be two distinct elements $\mathcal{M}_{n+1}$, that is, $p_{t_0} \neq q_{t_0}$ for some $t_0$.
 	By the inductive hypothesis $\mathcal{S}_{n}$ is an algebra which separates points in $\mathcal{M}_{n}$, therefore \cite[Theorem 4.5]{EtKu09} provides $G \in \mathcal S_{n}$ with
	\[ \int G\,dp_t\neq\int G\,dq_t, \]
 	whence, $\mathcal{S}_{n+1}$ separates points in $\mathcal M_{n+1}$.

	\emph{Claim}: For $F\in\mathcal{S}_n$ there is $f\in\AF[n]$ with
	\begin{equation} \label{eq:rank.AF.pp.claim3}
		F(\pp^n(\fp X))=f(\fp X) \quad\text{for all }\fp X \in \FP_p.
	\end{equation}

	Again, the assertion is trivial for $n=0$ as $\pp^0(\fp X) = X$.
	Assume that the claim holds for $n$, and let $F \in \mathcal{S}_{n+1}$ be represented by $\psi$ and $\vec{G}_1,\dots,\vec{G}_N$ as in \eqref{eq:form.function.Sn}.
 	By definition of the prediction process $\mathrm{pp}^n$ we have 
 	\[
		F(\pp^{n+1}(\fp{X})) = \psi\left(\E[\vec{G}_1(\mathrm{pp}^{n}(\fp{X}))|\F^\fp{X}_1],\dots,\E[\vec{G}_N(\mathrm{pp}^{n}(\fp{X}))|\F^\fp{X}_N]\right)
	 \]
	for all $\fp X \in \FP_p$.
 	By assumption there are vectors $\vec{g}_t$ of adapted functions in $\AF[n]$ with 
 	\[\vec{G}_t(\mathrm{pp}^{n}(\fp X))=\vec{g}_t(\fp X) \quad\text{for all }\fp X\in \FP.\]
 	Similarly as in the proof of Lemma \ref{lem:rep.AF} we collect all terms and obtain some $f\in\AF[n+1]$ with $F(\pp^{n+1}(\fp X)) = f(\fp X)$, which shows the claim.

 	With our two preliminary claims already established, we are ready to show that \ref{it:rank.adapted.functions.prediction.process.same1} implies \ref{it:rank.adapted.functions.prediction.process.same2}.
 	By \eqref{eq:rank.AF.pp.claim3} we have that 
	 \[
		\E\left[ F(\pp^n(\fp X)) \right] 
		= \E\left[ F(\pp^n(\fp Y)) \right]\quad\text{for all }F \in \mathcal S_n.
	 \]
 	As $\mathcal{S}_n$ is an algebra which separates points, it follows e.g.\ from  \cite[Theorem 4.5]{EtKu09} that $\pp^n(\fp X)$ and $\pp^n(\fp Y)$ have the same distribution.
 \end{proof}

 \begin{lemma} \label{lem:ip as continuous function of pp}
	 There is a 1-Lipschitz map $F \colon \mathcal M_{N-1} \to \Z$ such that
	 \[
		F(\pp^{N-1}(\fp X)) = \ip(\fp X) \quad \text{for all }\fp X \in \FP_p.	 
	 \]
 \end{lemma}

 \begin{proof}
	By Lemma \ref{lem:ppn contains ppk} there are 1-Lipschitz maps $G^{k,n} \colon \mathcal M_{n} \to \mathcal M_k$, $k < n$ with $G^{k,n} \circ \pp^{n} = \pp^k$.

	\emph{Claim}: For $1 \le t \le N$ there is a 1-Lipschitz map $F^t \colon \mathcal M_{N - t} \to \Z_{t}$ with
	\begin{equation}
		\label{eq:ppN-1 implies ip.induction}
		F^{t}(\pp^{N-t}(\fp X)) = \ip_{t}(\fp X)\quad\text{for all }\fp X \in \FP_p.
	\end{equation}

	Clearly, \eqref{eq:ppN-1 implies ip.induction} is satisfied when $t = N$.
	Indeed,  $\ip_N = \pp^0_N$ whereby $F^{N} := \proj_N$ fulfills \eqref{eq:ppN-1 implies ip.induction}.

	To establish \eqref{eq:ppN-1 implies ip.induction} for general $t$, we proceed by induction.
	Assuming that the claim holds true for $2 \le t \le N$, we find by the definition of the information process, see Definition \ref{def:interpolation process}, for $\fp X \in \FP_p$
	\begin{align*}
		\ip_{t - 1}(\fp X) 
		&= \left( X_{t - 1}, \Law\left( \ip_t(\fp X) | \F_{t - 1}^\fp X \right) \right) \\ 
		&= \left( \pp^0_{t-1}(\fp X), \Law\left( F^t(\pp^{N-t}(\fp X)) | \F^\fp X_{t - 1} \right) \right)  \\
		&= \left( \pp^0_{t - 1}(\fp X), (F^t)_\ast (\pp^{N - t + 1}_{t - 1}(\fp X)) \right) \\ 
		&= (\proj_{t - 1} \circ G^{0,N-t+1},\proj_{t - 1} \circ H^{t - 1}) \circ \pp^{N - t + 1}(\fp X),
	\end{align*}
	where $H^{t - 1} \colon \mathcal M_{N - t + 1} \to \Pc_p(\Z_t)$ is given by
	\[
		H^{t - 1}(p) := (F^t_\ast p_1,\ldots, F^t_\ast p_N).	
	\]
	Since $F^t$ is 1-Lipschitz  by assumption, the same holds true for $H^{t - 1}$.
	Therefore, $F^{t - 1} :=  (\proj_{t - 1} \circ G^{0,N-t+1},\proj_{t - 1} \circ H^{t - 1})$ is 1-Lipschitz and satisfies \eqref{eq:ppN-1 implies ip.induction}, which yields the claim.
	
	Finally, by the previously shown claim, the map 
	\[F := (F^1,F^2 \circ G^{N-2,N-1},\ldots, F^N \circ G^{0,N-1}) \]
	has the desired properties.
 \end{proof}
 

 \begin{theorem} \label{thm:information.process.AF}
 	Let $\fp X,\fp Y\in \FP_p$.
 	All of the following are equivalent:
 	\begin{enumerate}[label=(\roman*)]
 	\item \label{it:information.process.AF1a}
	 	$\fp{X} \sim_\infty \fp{Y}$.
 	\item \label{it:information.process.AF1b}
	 	$\fp{X} \sim_{N-1}\fp{Y}$.
	\item \label{it:information.process.AF2a}
	 	$\mathrm{pp}(\fp{X})$ and $\mathrm{pp}(\fp{Y})$ have the same distribution.
	\item \label{it:information.process.AF2b}
	 	$\pp^{N-1}(\fp{X})$ and $\pp^{N-1}(\fp{Y})$ have the same distribution.
 	\item \label{it:information.process.AF3a}
	 	$\ip(\fp{X})$ and $\ip(\fp{Y})$ have the same distribution.
	 \item \label{it:information.process.AF3anew}
	  	$\ip_1(\fp{X})$ and $\ip_1(\fp{Y})$ have the same distribution.
	\item \label{it:information.process.AF3b}
		$\AW_p(\fp X,\fp Y) = 0$.
 	\end{enumerate}
 \end{theorem}
 
	In Proposition \ref{prop:process.different.rank} we shall further prove that for every $1\leq n\leq N-1$, the relation $\sim_n$ strictly refines $\sim_{n-1}$: 
 there are $\fp X, \fp Y \in \FP_p$ with $\fp X \sim_{n-1} \fp Y$ but $\fp X\not\sim_{n} \fp Y$ (and especially $\AW_p(\fp X, \fp Y) > 0$).
	Importantly, these refinements are essential even for seemingly simple applications as we shall show in Theorem \ref{thm:stopping.rank}: Only the relation $\fp X \sim_{N-1} \fp Y$ guarantees that two processes $\fp X$ and $\fp Y$ have the same values for optimal stopping problems. 	

 \begin{proof}[Proof of Theorem \ref{thm:information.process.AF}]
	In a first step, note that \ref{it:information.process.AF1a} implies \ref{it:information.process.AF1b}; that \ref{it:information.process.AF2a} implies \ref{it:information.process.AF2b}; and that
	\ref{it:information.process.AF3a} implies \ref{it:information.process.AF3anew}.
	Further, Lemma \ref{lem:rank.adapted.functions.prediction.process.same} shows that \ref{it:information.process.AF1a} and \ref{it:information.process.AF2a} are equivalent and that \ref{it:information.process.AF1b} and \ref{it:information.process.AF2b} are equivalent.
	Theorem \ref{thm:isometry} shows that \ref{it:information.process.AF3a} and \ref{it:information.process.AF3b} are equivalent.
	Lemma \ref{lem:pp is continuous function of ip} shows that \ref{it:information.process.AF3a} implies \ref{it:information.process.AF2a}. 
	Finally, Lemma \ref{lem:ip as continuous function of pp} shows that \ref{it:information.process.AF2b} implies \ref{it:information.process.AF3a}.
	This concludes the proof.
 \end{proof}

\section{Topological and geometric properties of $\FFP_p$} \label{sec:aspects}

\subsection{Compactness in $\FFP_p$} \label{ssec:relcomp}

To develop a comprehensive understanding of a topology, 
it is essential to get a hold on compact sets.
For the weak topology, this is bestowed on us by Prokhorov's theorem which gives an easy to check tightness-criterion for relative compactness.
Theorem \ref{thm:rel.compact} implies that, perhaps surprisingly, the very same tightness-criterion also implies relative compactness for stochastic processes in $\FFP_p$

\begin{theorem}[Prokhorov's theorem] \label{thm:rel.compact}
	For a subset $\Pi \subseteq\FFP_p$, the following are equivalent.
	\begin{enumerate}[label=(\roman*)]
	\item $\Pi$ is relatively compact in $\FFP_p$.
	\item \label{it:prokhorov.2} $\{ \Law(X) \colon \fp X \in \Pi \}$ is relatively compact in $\Pc_p(\X)$.
\end{enumerate}	 
\end{theorem} 
	It is worthwhile to recall that condition \ref{it:prokhorov.2} is equivalent to tightness plus uniform integrability (see, e.g., \cite{Vi09}), that is,  for every $x_0\in\mathcal{X}$ and $\varepsilon>0$ there is a compact set $K\subset \mathcal{X}$ such that
	\[\sup_{\fp X\in \Pi} \E \left[ (1+d^p(x_0,X) ) 1_{\{X\notin K\}}\right]\leq \varepsilon .\]
As a consequence of the nested structure of $\Z_1$, the following \emph{intensity operator} plays an important role in the proof of Theorem \ref{thm:rel.compact}: 
for two Polish spaces $\mathcal{A}$ and $\mathcal{B}$ we define $\widehat{I} \colon \Pc_p(\mathcal{A}\times \Pc_p(\mathcal{B})) \to \Pc_p(\mathcal{A}\times \mathcal{B})$ via
	\begin{equation}\label{eq:def.intensity.hat}
		\int f(a,b) \, \widehat I(\pi)(da,db) = \iint f(a,b) \, p(db) \, \pi(da,dp)
	\end{equation}
for $f \in C_b(\mathcal{A}\times\mathcal{B})$.
	The intensity map $\widehat I$ closely relates relatively compact sets of its domain and its range in the sense of the subsequent lemma.

\begin{lemma}[c.f.\ Lemma 5.7 in \cite{BeJoMaPa21b}] \label{lem:rel.comp.intensity}
	Let $\mathcal{A}$ and $\mathcal{B}$ be two Polish spaces.
	For $\Pi \subseteq \Pc_p(\mathcal{A} \times \Pc_p(\mathcal{B}))$ are the following equivalent:
	\begin{enumerate}[label=(\roman*)]
		\item \label{it:rel.comp.intensity1} 
		$\Pi \subseteq \Pc_p(\mathcal{A} \times \Pc_p(\mathcal{B}))$ is relatively compact.
		\item \label{it:rel.comp.intensity2} 
		$\widehat I(\Pi)  \subseteq \Pc_p(\mathcal{A}\times \mathcal{B})$ is relatively compact.
	\end{enumerate}
\end{lemma}

\begin{proof}[Proof of Theorem \ref{thm:rel.compact}]
	As we know by Theorem \ref{thm:isometry} that $\FFP_p$ is isometrically isomorphic to $\Pc_p(\Z_1)$, we obtain that $\Pi$ is relatively compact if and only if the set $\{ \Law(\ip_1(\fp X)) \colon \fp X \in \Pi \}$ is relatively compact.
	On the other hand, note that for $1 \leq t < N$, $\mathcal A := \X_{1:t}$, and $\mathcal B := \Pc_p(\Z_{t+1})$ we have by the nested definition of $\ip_t$ that for all $\fp X \in \FFP_p$
	\[
		\widehat I\left( \Law(X_{1:t-1},\ip_t(\fp X)) \right) = \Law(X_{1:t},\ip_{t+1}(\fp X)).
	\]
	Applying Lemma \ref{lem:rel.comp.intensity} yields equivalence of the following statements:
	\begin{itemize}
		\item $\{ \Law(X_{1:t-1},\ip_t(\fp X)) \colon \fp X \in \Pi \}$ is relatively compact;
		\item $\{ \Law(X_{1:t},\ip_t(\fp X)) \colon \fp X \in \Pi \}$ is relatively compact;
	\end{itemize}
	Hence, by applying this argument iteratively, we find that $\{ \Law(X) \colon \fp X \in \Pi \}$ is relatively compact if and only if $\{ \Law(\ip_1(\fp X)) \colon \fp X \in \Pi \}$ is relatively compact, which we wanted to show.
	The final assertion is a direct consequence of the classical Prokhorov's theorem and the characterization of Wasserstein convergence in $\Pc_p(\mathcal A \times \mathcal B)$, see \cite[Definition 5.8]{Vi09}.
\end{proof}

\subsection{Denseness of simple processes} \label{ssec:dense}
A canonical way of embedding $\Pc_p(\X)$ into $\FFP_p$ is the following:
we can associate to each law $\P \in \Pc_p(\X)$ the processes
\begin{equation}
	\label{eq:plain}
	\fp X \equiv \left( \X, \mathcal B(\X), \P, (\sigma(X_{1:t}))_{t = 1}^N,  X \right),
\end{equation}
where $X$ denotes the coordinate process on $\X$, and call this type of process \emph{plain}.
The set of all plain processes is denoted by $\Lambda_{\mathrm{plain}} \subseteq \FFP_p$.

\begin{proposition} \label{prop:plain}
	Let $(\Omega,\F,\Q)$ be an arbitrary probability space, let $Y \colon \Omega \to \X$ be a $\F$-measurable map such that $\Law(Y) \in \Pc_p(\X)$, and denote by $\fp X$ the plain process associated to $\Law(Y)$; i.e.\ $\fp X$ is given by \eqref{eq:plain} with $\P=\Law(Y)$.
	Then
	\begin{equation}
		\label{eq:plain rep}
		\tilde{\fp X} := \left( \Omega, \F, \Q, (\sigma(Y_{1:t}))_{t=1}^N, Y \right)
	\end{equation}
	satisfies  $\AW_p(\fp X,\tilde{\fp X}) = 0$.
	In particular, if $\hat{\fp X},\hat{\fp Y} \in \FFP_p$ are plain and $\Law(\hat X) = \Law(\hat Y)$, then $\hat{\fp X} = \hat{\fp Y}$.
\end{proposition}

\begin{proof}
	The coupling $(\textrm{id}_\Omega,Y)_\ast \Q$ is bicausal between $\tilde{\fp X}$ and $\fp X$ as well as $\tilde{\fp Y}$ and $\fp Y$. Thus $\AW_p(\fp X,\tilde{\fp X}) = 0 = \AW_p(\fp Y, \tilde{\fp Y})$.
\end{proof}

Among others, a purpose of this subsection is to show that the space of filtered processes $\FFP_p$ naturally appears as the completion of all plain processes $\Lambda_{\mathrm{plain}}$.

We call a probability space $(\Omega,\mathcal{F},\P)$  \emph{finite} if $\Omega$ consists of finitely many elements.

\begin{theorem}
	\label{thm:dense}
	If $\X$ has no isolated points, then the set
	\[
		\left\{ \fp X \in \FFP_p \colon \fp X \text{ is Markov and has a representative on a finite probability space} \right\}
	\]
	is dense in $\FFP_p$.
	In particular, the plain processes $\Lambda_{\mathrm{plain}}$ are a dense subset.
\end{theorem}

This theorem follows from the next proposition.

\begin{proposition}
	\label{prop:markov.dense}
	Let $\fp X \in \FP_p$ and let $\varepsilon > 0$.
	Then there is 
	\begin{equation}
		\label{eq:approx FP}
		\fp Y := (\Omega^\fp X, \F^\fp X,\mathbb P^\fp X,(\F^\fp Y_t)_{t = 1}^N,Y)\in\FP_p,
	\end{equation}
	where for every $1\leq t\leq N$, $Y_t$ is a function of $\ip_t(\fp X)$ and $\mathcal F_t^{\fp Y}$ is a finite subset of $\F_t^\fp X$, such that
	\begin{equation}
		\AW_p(\fp X,\fp Y) < \varepsilon.
	\end{equation} 
	In particular, if $\X$ has no isolated points, then $\fp Y$ can be chosen Markovian.
\end{proposition}

The proof relies on the following result, which is essentially shown in \cite[Lemma 4.8]{BeLa20}.
Recall here that for $\mu,\nu \in\mathcal{P}_p(\X)$, the term $\AW_p(\mu,\nu)$ refers to $\AW_p(\fp X,\fp Y)$ where $\fp X$ and $\fp Y$ are plain processes, see \eqref{eq:plain}, distributed according to $\mu$ and $\nu$, respectively.
In a similar fashion, we will use here $\cplba(\mu,\nu)$ to denote the set of all bicausal couplings between the corresponding plain processes.

\begin{lemma}
	\label{lem:approximation of plain process}
	For each $1 \leq t \leq N$, let $(\mathcal{Y}^m_t)_{m \in \N}$ be an increasing family of finite subsets of Polish spaces $\Y_t$ such that
	$\cup_{m \in \N} \mathcal{Y}^m_t \text{ is dense in }\Y_t$
	and let $\phi_t^m\colon \Y_t \to \mathcal{Y}_t^m$ be the map which assigns each point its nearest point in $\mathcal{Y}_t^m$ (with ties broken arbitrarily but measurably).
	Then, for any $\mu \in \Pc_p(\Y_{1:N})$, we have
	\begin{equation}
		\label{approx for plain processes}
		\lim_{m \to \infty} \AW_p \left( \mu, (\phi_{1:N}^m)_\ast \mu \right) = 0.
	\end{equation}
\end{lemma}

\begin{proof}
	For every $m$, set 
	\[\mu^m
	:=(\phi_{1:N}^m)_\ast \mu 
	= \big( x_{1:N} \mapsto (\phi_1^m(x_1),\ldots,\phi_N^m(x_N)) \big)_\ast \mu.\]
	By \cite[Lemma 1.4]{BaBaBeEd19b}, \eqref{approx for plain processes} is equivalent to 
	\begin{gather}
		\label{plain approx weak convergence}
		\inf_{\pi \in \cplba(\mu,\mu^m)} \int d(x,y) \wedge 1 \, \pi(dx,dy) \to 0 \text{ and}
	\\	\label{plain approx moments} \int d^p(x,x^0) \, \mu^m(dx) \to \int d^p(x,x^0) \, \mu(dx),
	\end{gather}
	as $m\to\infty$, where $x^0\in \Y_{1:N}$ is some arbitrary but fixed element.
	For convenience, we choose $x^0\in \Y_{1:N}^0$.
	
	If each $\Y_t$ were compact, \eqref{plain approx weak convergence} would follow directly from \cite[Lemma 4.8]{BeLa20}.
	However, compactness in \cite[Lemma 4.8]{BeLa20} was only used to additionally obtain the rate of convergence, and the same proof shows that in the present setting \eqref{plain approx weak convergence} holds true.
	As for \eqref{plain approx moments}, note that the reverse triangle inequality shows that
	\[ |d(\phi_{1:N}^m(x),x^0) - d(x,x^0)|
	\leq d(\phi_{1:N}^m(x),x) \]
	for every $x\in\Y_{1:N}$.
	Thus, $d(\phi_{1:N}^m(x),x)$ decreases to zero as $m\to\infty$ and is bounded by $d(x^0,x)$ due to the definition of $\phi^m$ and $x^0\in \Y_{1:N}^0$.
	Dominated convergence shows that  \eqref{plain approx moments} holds and completes the proof.
\end{proof}

 \begin{proof}[Proof of Proposition \ref{prop:markov.dense}]
	Consider $\mu := \Law(\ip(\fp X)) \in \Pc_p(\Z)$ and interpret $\Z = \Z_{1:N}$ as a path space. 
	By Corollary \ref{lem:approximation of plain process} there is a family of laws $(\mu^m)_{m \in \N}$ on $\Pc_p(\Z)$ with
	\[ \lim_{m \to \infty}\AW_p(\mu,\mu^m) = 0 \]
	and for every $m \in \N$, $\mu^m$ is finitely supported.

	Fix $\varepsilon>0$ and $m \in \N$ such that $\AW_p(\mu,\mu^m) < \varepsilon$.
	Let $\phi_t^m \colon \Z_t \to \Z_t$, $1 \le t \le N$ be the family of maps introduced in Corollary \ref{lem:approximation of plain process}.
	We write $\overline{Y}_t := \proj_{\X_t} \circ \phi^m_{1:N}$.
	The maps $\phi_t^m (\ip_t(\fp X))$ and $Y_t := \overline{Y}_t(\ip_t(\fp X))$ are both $\F^\fp X_t$-measurable and
	\[
		\mathcal F_t^\fp Y := \sigma\left( \phi_s^m (\ip_s(\fp X)) \colon 1 \leq s \leq t \right),
	\]
	is a finite sub-$\sigma$-algebra of $\F_t^\fp X$.
	The filtered process $\fp Y$ is given as in \eqref{eq:approx FP} with process $Y$ and filtration $(\F_t^\fp Y)_{t = 1}^N$.
	By virtue of Lemma \ref{lem:causal CI} it is readily verified that the coupling $( \id_{\Omega^\fp X}, \phi^m_t(\ip_t(\fp X))_{t = 1}^N ))_\ast P^\fp X$ is bicausal between $\fp Y$ and
	\[
		\overline{\fp Y} := \left( \Z, \F^\Z, \mu, \left(  \sigma(\phi^m_{1:t})\right)_{t = 1}^N,  \overline{Y}_{1:N} \right),
	\]
	whence $\AW_p(\overline{\fp Y},\fp Y) = 0$.
	Similarly, the coupling $(\id_\Z, \phi^m_{1:N})_\ast \mu$ is bicausal between $\overline{\fp Y}$ and
	\[
		\fp Z := (\Z,\F^\Z,\mu^m,(\F_t^\Z)_{t=1}^N,Z) \text{ where }Z_t = \proj_{\X_t},
	\]
	thus, $\AW_p(\overline{\fp Y},\fp Z) = 0$ and $\AW_p(\fp Y,\fp Z) = 0$.
	By Theorem \ref{thm:isometry} we conclude
	\begin{align*}
		\AW_p(\fp X,\fp Y) &= \AW_p(\fp X,\fp Z) = \left( \inf_{\pi \in \cplba(\mu,\mu^m)} \int d^p(z^-,\hat z^-) \, \pi(dz,d\hat z) \right)^\frac{1}{p} \\
		&\le \left( \inf_{\pi \in \cplba(\mu,\mu^m)} \int d^p(z,\hat z) \, \pi(dz,d\hat z) \right)^\frac{1}{p} =  \AW_p(\mu,\mu^m) < \varepsilon. \qedhere
	\end{align*}
 \end{proof}
 
 \subsection{Martingales} 
\label{ssec:martingale}

Assume for this subsection that $\X_t=\R^d$ for each $1\leq t\leq N$.

\begin{proposition}[Martingales]
\label{cor:martingales}
	The set of all martingales
	\[ \mathrm{M}_p:= \{ \fp X\in \FFP_p : \fp X \text{ is a martingale} \}  \]
	is closed w.r.t.\ $\AW_p$.
\end{proposition}

There is a multitude of ways how to prove Proposition \ref{cor:martingales}:
\begin{enumerate}[label=(\roman*)]
	\item as a consequence of the continuity of Doob-decomposition (Proposition \ref{prop:doob} below), 
	\item as a consequence of the continuity of optimal stopping,
	\item as a consequence of Example \ref{ex:martingales.opt.stop},
	\item by characterizing martingales as those processes $\fp X \in \FFP_1$ for which $\Law(\ip_1(\fp X))$ is concentrated on a particular closed subset of $\Z_1$,
	\item or directly by coupling arguments.
\end{enumerate}

We will present the last variant.

\begin{proof}[Proof of Proposition \ref{cor:martingales}]
	Let $(\fp X^n)_{n \in \N}$ be a sequence in ${\rm M}_p$ converging to $\fp X$.
	Fix $n \in \N$, let $\pi\in\cplba(\fp X^n,\fp X)$ and $1\le t \le s \le N$.
	By Lemma \ref{lem:causal CI} and the martingale property of $\fp X^n$ we have 
	\begin{equation*}
		\E\left[X_{s} \middle| \mathcal{F}^\fp X_t\right] 
		= \E_\pi\left[X_{s} \middle|  \mathcal{F}^{\fp X,\fp X^n}_{t,t} \right] \text{ and } X^n_t 
		= \E_\pi\left[X_{s}^n \middle| \mathcal{F}^{\fp X,\fp X^n}_{t,t} \right].
	\end{equation*}
	Thus, letting $\Delta^n := (X_t^n-X_t) + (X_{s}^n-X_{s})$, Jensen's inequality yields that
	\begin{align*}
	\E\left[ \left|X_t-\E[X_s \middle| \mathcal{F}^{\fp X}_t]\right|\right] 
	&= \E_\pi\left[\left|X_t^n - \E_\pi\left[X_{s}^n + \Delta^n \middle| \mathcal{F}^{\fp X,\fp X^n}_{t,t}\right]\right|\right] \\
	&\leq 	\E_\pi[ |\Delta^n | ] \leq \E_\pi[ |X^n - X|].
	\end{align*}
	As $\pi\in\cplba(\fp X^n,\fp X)$ was arbitrary, $\E[|X_t - \E[X_{s} | \F_t^\fp X]]$ is dominated by $\AW_1(\fp X^n,\fp X)$ which is arbitrarily small for $n$ sufficiently large.
	Hence, $\E[|X_t-\E[X_{s} | \mathcal{F}^{\fp X}_t]|]=0$ showing that $\fp X$ is a martingale.
\end{proof}

\subsection{$(\FFP_p,\AW_p)$ is a geodesic space} 
\label{ssec:geodesic}

The purpose of this section is to show Theorem \ref{thm:geodesic}, and in particular that $(\FFP_p,\AW_p)$ is a geodesic space.
For concise notation, we shall assume throughout that $\X_t=\mathbb{R}^d$ for every $1\leq t\leq N$ (but see also Remark \ref{rem:geodesic.X.not.R}).

\begin{definition}[Constant speed geodesic] \label{def:constant speed geodesic}
	A family $(\fp Z^{u})_{u\in[0,1]}$ in $\FFP_p$ is said to be a \emph{constant speed geodesic} connecting $\fp X, \fp Y \in\FFP_p$ if 
	\begin{enumerate}[label=(\alph*)]
	\item \label{it:def geodesic1}
	$\fp Z^0 = \fp X$ and $\fp Z^1 = \fp Y$,
	\item \label{it:def geodesic2}
	$\AW_p(\fp Z^u, \fp Z^v)=|u-v|\AW_p(\fp X,\fp Y)$ for all $u,v\in[0,1]$.
	\end{enumerate}
\end{definition}

A tangible way of defining constant speed geodesics is -- in analogy to the classical $\W_p$-displacement interpolation -- by means of geodesics on the state space and optimal couplings.
To that end, recall that $\Z$ is the canonical space defined in Definition \ref{def:canonical.space.Z}.

\begin{definition}[Interpolation process]
 \label{def:interpolation process}
	Let $\fp X, \fp Y \in \FFP_p$ and let $\pi \in \cplba(\fp X, \fp Y)$.
	We call the family $(\fp Z^{\pi,u})_{u \in [0,1]}$ given by
	\begin{equation}
		\label{eq:interpolation process}
		\fp Z^{\pi,u} \
		:= \left( \Z \times \Z, \F^\Z \otimes \F^\Z, \pi,(\F_{t,t}^{\Z,\Z})_{t = 1}^N,  ((1-u)X_t+uY_t)_{t = 1}^N  \right)
	\end{equation}	
	the \emph{interpolation process} between $\fp X$ and $\fp Y$ (w.r.t.\ the coupling $\pi$).
\end{definition}

The following is the main result of this section.

\begin{theorem}[Filtered processes form a geodesic space]
\label{thm:geodesic}
	Let $p\in(1,\infty)$.
	\begin{enumerate}[label=(\roman*)]
	\item \label{it:geodesic.1}
	The space $(\FFP_p,\AW_p)$ is a geodesic space, that is, for every $\fp X,\fp Y\in \FFP_p$ there is a constant speed geodesic connecting them.
	\item \label{it:geodesic.2}
	A family $(\fp Z^u)_{u\in[0,1]}$ in $\FFP_p$ is a constant speed geodesic between $\fp X,\fp Y\in \FFP_p$ if and only if the family
	\[ (\gamma^u)_{u\in[0,1]}:= (\Law(\ip_1(\fp Z^u)))_{u\in[0,1]}
	\] is a $\W_p$-constant speed geodesic\footnote{
		A family $(\gamma^u)_{u\in[0,1]}$ in $\mathcal{P}_p(\Z_1)$ is said to be  a $\W_p$-constant speed geodesic connecting $\mu,\nu\in \mathcal{P}(\Z_1)$ if $\gamma^0=\mu$, $\gamma^1=\nu$, and $\W_p(\gamma^u,\gamma^v)=|u-v|\W_p(\mu,\nu)$ for every $u,v\in[0,1]$.}
		between $\Law(\ip_1(\fp X))$ and $\Law(\ip_1(\fp Y))$.
	\item \label{it:geodesic.3}
	If $\pi$ is an optimal bicausal coupling for $\AW_p(\fp X,\fp Y)$, then the interpolation process $(\fp Z^{\pi,u})_{u \in [0,1]}$ is a constant speed geodesic between $\fp X$ and $\fp Y$.
	\end{enumerate}
\end{theorem}

In a forthcoming paper, it will be shown that not only does the interpolation process constitute a geodesic, but actually \emph{all} geodesics can be described as interpolation processes, at least once one is willing to allow for external randomization, that is, extending the probability space by independent randomness.

\begin{proof}
	We start by proving \ref{it:geodesic.1} and \ref{it:geodesic.2}.
	As $\Z_N=\X_N$ is a geodesic space, it follows from \cite{Li07} that $\mathcal{P}_p(\X_N)$ is a geodesic space, too.
	Moreover, the product (endowed with the $l^p$-norm) of two geodesic spaces remains geodesic, hence $\Z_{N-1}=\X_{N-1}\times \mathcal{P}_p(\Z_N)$ is a geodesic space.
	Repeating this argument inductively shows that $\Pc_p(\Z_1)$ a geodesic space.
	Claim \ref{it:geodesic.1} and \ref{it:geodesic.2} now follow from the isometry between $\FFP_p$ and $\mathcal{P}_p(\Z_1)$ given in Theorem \ref{thm:isometry}.

	We now show \ref{it:geodesic.3}.
	Bicausality of $\pi$ and Lemma \ref{lem:causal CI} immediately show part \ref{it:def geodesic1} of Definition \ref{def:constant speed geodesic}, and it remains to deal with part \ref{it:def geodesic2}.
	To that end, note that the coupling $\Pi := \left( \id, \id \right)_\ast \pi$ is bicausal between $\fp Z^{\pi,u}$ and $\fp Z^{\pi,v}$.
	Then, for every $u,v\in[0,1]$, as $Z^u-Z^v=(u-v)(X-Y)$, we compute
 	\begin{align}
		\label{eq:interpolation process is geodesic}
		\begin{split}
		\AW_p(\fp Z^{\pi,u},\fp Z^{\pi,v}) 
		&\leq \E_\Pi[ \| Z^u-Z^v\|_p^p]^{1/p}= |u-v| \E_\pi[ \|X-Y\|_p^p]	^{1/p} 
		= |u-v| \AW_p(\fp X,\fp Y),
		\end{split}
	\end{align}
	where the last equality holds by optimality of $\pi$.
	A straightforward application of the triangle inequality shows that
	there cannot be strict inequality in \eqref{eq:interpolation process is geodesic}, and whence the claim follows.
\end{proof}

The next example illustrates that even in case of geodesics between plain processes (c.f.\ \eqref{eq:plain}) it is necessary to consider general filtrations (instead of the  filtrations generated by the processes).

\begin{example}
\label{ex:plain.not.geodesic}
	Let $N = 2$.
	We consider two processes with paths in $\R^2$: $\fp X$ and $\fp Y$ are the plain process associated to the laws (on the path space $\R^2$)
	\[
		\mu := \frac{1}{2} \left( \delta_{(1,-2)} + \delta_{(-1,2)} \right) \text{ and } \nu := \frac{1}{2} \left( \delta_{(1,1)} + \delta_{(-1,-1)} \right),
	\]
	respectively.
	It is then easy to verify that there exists a unique constant speed geodesic and that it is given by the interpolation process of $\fp X$ and $\fp Y$ w.r.t.\ the unique $\AW_p$-optimal coupling $\pi \in \cplba(\fp X,\fp Y)$.
	This coupling sends the mass from $(1,-2)$ to $(-1,-1)$ and the mass from $(-1,2)$ to $(1,1)$.
	Therefore, at time $1/2$, $\fp Z^{\pi,1/2}$, is not a plain process, since
	\[
		\Law(\ip_1(\fp Z^{\pi,1/2})) = \frac{1}{2} \left( \delta_{(0,\delta_{3/2})} + \delta_{(0,\delta_{-3/2})} \right);
	\]
	or put differently, even though $Z^{\pi,1/2}_1 = 0$ we know at $t = 1$ already precisely where we will end up at $t = 2$, thus, $\fp Z^{\pi,1/2}$ is not plain.
\end{example}

\begin{theorem}
	For $p\in(1,\infty)$, the set  ${\rm M}_p$ (of martingales in $\FFP_p$) forms a closed, geodesically convex subset of $\FFP_p$. 
\end{theorem}

 \vspace{-20mm}
 \begin{figure}[H]
    \centering
    \includegraphics[page=1,width=0.6\textwidth]
        {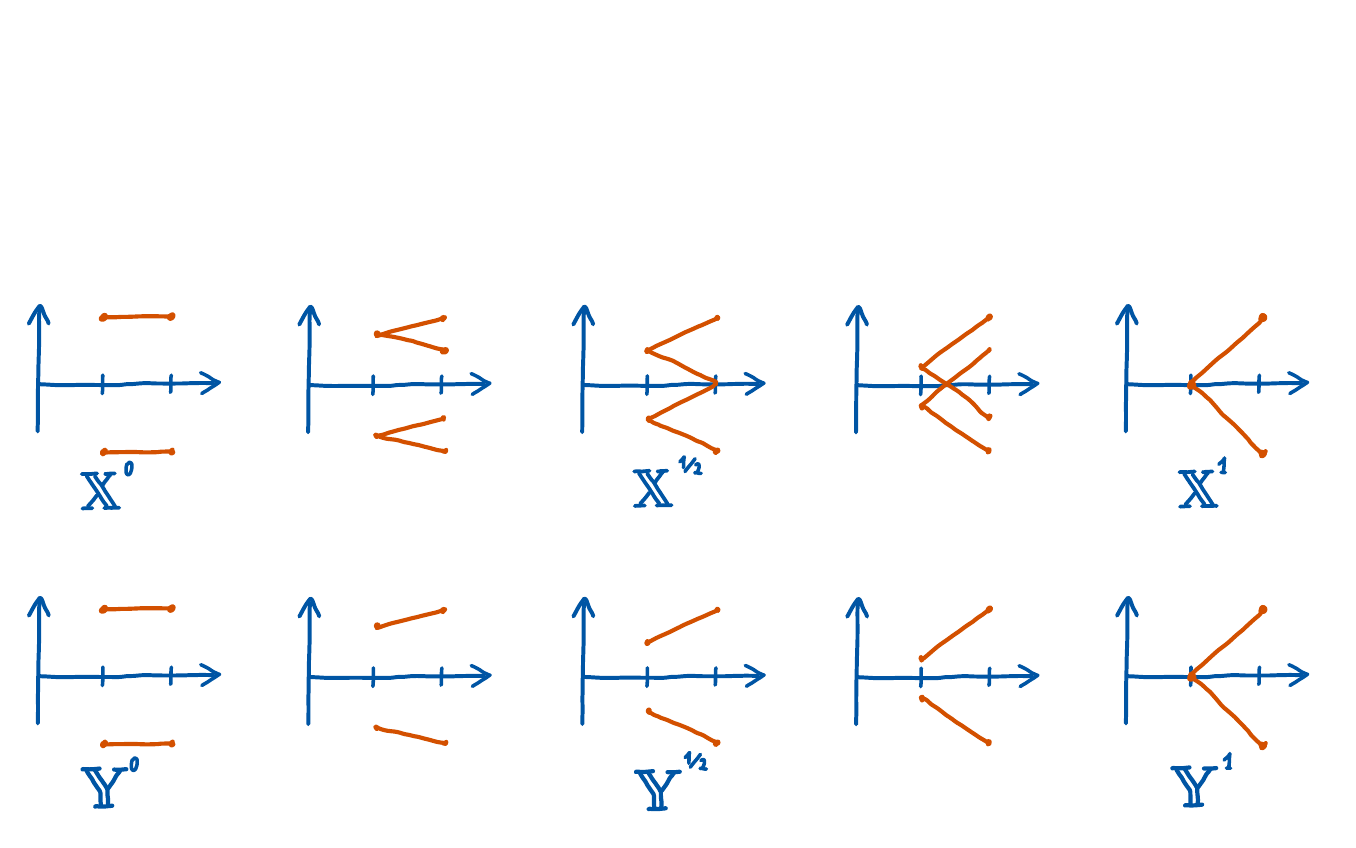} 
  \caption{Comparison of the adapted Wasserstein interpolation $(\fp X^u)_{u\in [0,1]}$ and a classical Wasserstein interpolation $(\fp Y^u)_{u\in [0,1]}$ where the geodesic between martingales does not consist of martingales.}
 \label{fig:usual.wasserstein}
 \end{figure}
 \vspace{-02mm}

	For ease of exposition, we will only check that the set of martingales is geodesically convex when restricting to geodesics given by interpolation processes.
	Knowing that all geodesics can be characterized this way modulo an external randomization, this assumption can in fact be made without loss of generality.
	Alternatively, a proof via our description of geodesic processes as geodesics on $\mathcal{P}_p(\Z_1)$ in Theorem \ref{thm:geodesic} is possible too, but less informative, and therefore left to the ambitious reader.

\begin{proof}[Simplified proof]
	By Proposition \ref{cor:martingales}, it remains to show that ${\rm M}_p$ is geodesically convex. 
	To that end, let $\fp X,\fp Y\in {\rm M}_p$  and let $(\fp Z^u)_{u\in[0,1]}$ be a constant speed geodesic connecting them, which, as already explained, is assumed to be given as the interpolation processes w.r.t.\ some $\pi\in\cplba(\fp X,\fp Y)$.	
	
	We need to show that, for given fixed $u\in[0,1]$, the processes $\fp Z^u$ is a martingale, too.
	To that end, let $1\leq s\leq t\leq N$ and write
	\begin{align*}
	\E\left[Z^u_t \middle| \mathcal{F}^{\fp Z}_s\right]
	&=\E_\pi\left[(1-u)X_t+uY_t \middle| \mathcal{F}^{\fp X,\fp Y}_{s,s}\right]=(1-u)\E\left[X_t \middle|\mathcal{F}^{\fp X}_{s}\right]+u \E\left[Y_t \middle|\mathcal{F}^{\fp Y}_{s}\right]
	= Z_s^u
	\end{align*}
	where we use bicausality of $\pi$ in the form of Lemma \ref{lem:causal CI}, and the martingale property of both $\fp X$ and $\fp Y$.
	Hence $\fp Z^u$ is a martingale which completes the proof.
\end{proof}

\begin{remark}
\label{rem:geodesic.X.not.R}
	We chose $\X_t=\mathbb{R}^d$ in this section to lighten notation.
	However, Theorem \ref{thm:geodesic} remains valid if all of the $\X_t$'s are geodesic space, with the obvious modifications, such as replacing $(1-u)X_t+uY_t$ in Definition \ref{def:interpolation process} by (appropriately measurable selections of) geodesics between $X_t$ and $Y_t$.
\end{remark}

\section{Continuity w.r.t.\ $\mathcal{AW}_p$ and applications}
\label{sec:applications}

We have argued in the introduction that the weak adapted distribution governs `all' probabilistic aspects of a stochastic process.
In this section we highlight this claim, and further show that several (optimization) problems involving stochastic processes \emph{continuously} depend on the adapted distribution.
And moreover, that \emph{quantitative} estimates w.r.t.\ $\AW_p$ are possible.
Throughout this section we assume that  $\mathcal{X}_t=\R^d$ for every $1\leq t\leq N$.

\subsection{Optimal stopping} \label{ssec:optimal stopping}
Fix a non-anticipative function $c\colon (\R^d)^N\times\{1,\dots,N\}\to\mathbb{R}$ and set
\begin{align}
\label{eq:def.opt.stop}
v_c(\fp X):=\inf_{\tau\in\mathrm{ST}(\fp{X})} \E[c_\tau(X)] 
\end{align}
for $\fp X\in \FP$, where $\mathrm{ST}(\fp{X}) $ denotes the set of all $(\F_t^\fp{X})_{t=1}^N$-stopping times taking values in $\{1,\dots,N\}$.
In Example \ref{ex:martingales.opt.stop} we have already seen that $v_c$ is well-defined on $\FFP_p$ in the sense that the value of $v_c$ does not depend on the choice of representative.
 
 \begin{proposition}[Optimal stopping]
 \label{lem:stopping.lipschitz}
	Let $\fp X, \fp Y\in\FFP$  such that $c_{1:N}(X)$ and $c_{1:N}(Y)$ are integrable.
	Then we have
 	\begin{align*}
 	v_c(\fp X) - v_c(\fp Y)
 	&\leq \inf_{\pi\in\cpla(\fp{X},\fp{Y})}  \E_\pi\Big[ \max_{1\leq t\leq N} |c_t(X)-c_t(Y)| \Big] .
 	\end{align*}
 In particular,  the following hold.
 \begin{enumerate}[label=(\roman*)]
 \item If $c_t$ is continuous bounded for every $t$, then $\fp X\mapsto v_c(\fp X)$ is continuous on $\FFP$ w.r.t.\ the weak adapted topology.
 \item If $c_t$ is continuous and satisfies that $|c_t(\cdot)|\leq \alpha( 1 + |\cdot|^p)$ for every $t$ and some $\alpha>0$, then $\fp X\mapsto v_c(\fp X)$ is continuous on $\FFP_p$ w.r.t.\ $\mathcal{AW}_p$.
 \item If $c_t$ is Lipschitz  for every $t$, then $\fp X\mapsto v_c(\fp X)$ is $\AW_1$-Lipschitz  on $\FFP_1$.
 \end{enumerate}
 \end{proposition}

 In Theorem \ref{thm:stopping.rank} we will construct an example showing that the whole adapted distribution is required to control optimal stopping problems.
 Note that the non-quantitative version of Proposition \ref{lem:stopping.lipschitz} (i.e.\ continuity of optimal stopping for continuous $c$ that satisfies an adequate growth condition) already follows from Example \ref{ex:martingales.opt.stop}.
 
 \begin{proof}[Proof of Proposition \ref{lem:stopping.lipschitz}] 
 	The `in particular' statement follows by symmetry from the first statement, so we shall only prove the first one.
 	Its proof is similar to \cite{BaBaBeEd19b} and is included here for the convenience of the reader.
 	Let $\varepsilon>0$ and $\tau^\ast\in\mathrm{ST}(\fp{Y})$ be a stopping time such that $\E[c_{\tau^\ast}(Y)]\leq v_c(\fp Y) + \varepsilon$.
 	Further let $\pi\in\cpla(\fp{X},\fp{Y})$ and, for every $u\in[0,1]$, define
 	\[ \sigma_u:= \min\left\{ t\in\{1,\dots,N\} : \pi\left(\tau^\ast \leq t \middle| \mathcal{F}^{\fp X,\fp Y}_{N,0}\right)\geq u\right\}. \]
 	By causality, c.f.\ Lemma \ref{lem:causal CI}, we have $\sigma_u\in\mathrm{ST}(\fp{X})$, hence
 	\begin{align*}
		v_c(\fp X)&\leq \inf_{u \in [0,1]} \E_\pi\left[ c_{\sigma_u}(X) \right] \le \int_{[0,1]} \E_\pi[c_{\sigma_u}(X)]\,du \\
		&=\sum_{t=0}^N \int_{[0,1]} \E_\pi\left[  c_t(X) 1_{\pi\left(\tau^\ast\leq t \middle|\mathcal{F}^{\fp X,\fp Y}_{N,0}\right)\geq u> \pi\left(\tau^\ast\leq t-1|\mathcal{F}^{\fp X,\fp Y}_{N,0}\right)}\right] \, du = \E_\pi[c_{\tau^\ast}(X)].
 	\end{align*}
	In conclusion we obtain
 	\[
		v_c(\fp X) - v_c(\fp Y)	\leq \E_\pi\left[c_{\tau^\ast}(X)-c_{\tau^\ast}(Y)\right] + \varepsilon \leq \E_\pi \left[ \max_{1\leq t\leq N} |c_t(X)-c_t(Y)| \right] + \varepsilon,
 	\]
 	which, as $\varepsilon > 0$ and $\pi \in \cpla(\fp X,\fp Y)$ were arbitrary, yields the assertion.
 \end{proof}

\subsection{American options and robust pricing}

Working in the setup of e.g.\ \cite{BeHePe12}, $X_t\in \R_+$ stands for the (discounted) price of a financial asset at a  time $t\in \{1, \ldots, N\}$ and  $x_1 = X_1\in \R_+$ denotes the current price. 
One assumes that there exists  a family of \emph{European}  derivatives, described by a (continuous and linearly bounded) family of functions $\phi_i:(\R_+)^N\to \R, i\in I$  which are liquidly traded in the market, meaning that the respective prices $p_i, i\in I$ are specified from externally given data. 
The  set of all \emph{calibrated models}\  consists of all martingales with mean $x_1$ which correctly reproduce the prices given by the market, i.e.\ in mathematical terms
\begin{align}\label{eq:CaMo}
\mathrm{M}_I:=\{\fp X\in \mathrm{M}_1 : X_1= x_1, \E [\phi_i (X)]= p_i, i\in I\}.
\end{align}
A common assumption is  
\begin{align}\label{eq:CC} \tag{\text{CC}}
\text{$\{\phi_i\colon i \in I\}$ contains  all  call options written on $X_N$ and ${\rm M}_I\neq\emptyset$.}
\end{align} 
Going back to a famous observation of Breeden-Litzenberger \cite{BrLi78},
this implies  the following basic  fundamental fact:
\begin{proposition}\label{basic:MOT}
Under assumption \eqref{eq:CC} the set $\{\Law(\fp X): \fp X\in  \mathrm{M}_I \}$ is   $\W_1$-compact.
\end{proposition}
A direct consequence is that for a further (continuous, linearly bounded) European derivative  $\Phi:(\R_+)^N\to \R, i\in I$ the set of possible arbitrage free prices consists of the  interval 
$\big[\inf_{\fp X\in  \mathrm{M}_I}  \E[\Phi(X)],$ $  \sup_{\fp X\in  \mathrm{M}_I}  \E[\Phi(X)]\big]$, where the endpoints are attained for `extremal models'. 
Proposition \ref{basic:MOT} as well as various extensions of it play a crucial role for the duality theory as well as the characterization of extremal models in robust finance, see \cite{BeNuTo16, BeCoHu14, GuOb19, ChKiPrSo20} among many others.

A particular limitation of Proposition \ref{basic:MOT} is that it allows only to consider derivatives with a European payoff structure, but neglects derivatives with an American exercise structure, where the buyer may choose when to exercise and the buyers price in  a model $\fp X$ equals
\begin{align}
\label{eq:AmEx} 
\sup \{ \E[ \Phi(X) ] : \tau\in {\rm ST}(\fp X)\}.
\end{align}
Apart from a few important exceptions (see \cite{BaZh19, HoNe17, AkDeObTa19}) the robust finance literature is focused on the case of European derivatives.
The simple reason is that going beyond the standard European case requires an adequate topology on processes with a \emph{non trivial} filtration, which was hitherto unavailable.
As derivatives with American exercise structure are \emph{more} common than European derivatives it is highly desirable to extend the existing theory to this case.
 As a consequence of our results we obtain.

\begin{proposition}
Assume that $\{\phi_i: i\in I\}$ is a family of derivatives with linearly bounded continuous payoffs and European or American exercise structures. 
Under assumption \eqref{eq:CC}, the set $\{\fp X: \fp X\in  \mathrm{M}_I \}$  is $\AW_1$-compact.
Moreover,  for a European or American  derivative with continuous  linearly bounded payoff $\Phi:(\R_+)^N\to \R$, the lower/upper pricing bounds for the derivative $\Phi$ are attained.
\end{proposition}

\begin{proof}
This is a direct consequence of Prokhorov's theorem in our setting (Theorem \ref{thm:rel.compact}),  Proposition \ref{eq:AmEx}, and the continuity of optimal stopping (Proposition \ref{lem:stopping.lipschitz}).
\end{proof} 

\subsection{Utility maximization}

Let $ U\colon \R\to\R$ be an increasing concave (utility) function, and denote by $U'$ the left-continuous version of the derivative.
Denote by $\mathcal{H}(\fp X)$  the set of all $(\F^{\fp X}_t)_{t = 1}^N$-predictable processes $H$ that are bounded by 1, and by $(H\cdot X)_t:=\sum_{s=1}^{t-1} H_{s+1}(X_{s+1}-X_{s})$ the discrete-time stochastic integral of $H$ w.r.t.\ $X$.
For $C\colon (\R^d)^N\to \R$ and  $\fp X\in\FFP_p$, denote by $u(\fp X)$  the value of the utility maximization problem  with random endowment $C$, that is,
\[ u(\fp X) :=\sup_{H\in\mathcal{H}(\fp X)} \E[ U( C(X) + (H\cdot X)_N ) ].\]
In \cite[Theorem 1.8]{BaBaBeEd19a} it is shown that $u(\fp X)$ depends continuously on $\fp X$ (w.r.t.\ $\mathcal{AW}_p$) when restricting to plain processes (i.e.\ to $\fp X$ whose filtration is generated only by their paths).
In the context of utility maximization, however, it is of central importance to also understand the effect that changes of the information/filtration of $\fp X$ have to $u(\fp X)$, and the following result extends \cite[Theorem 1.8]{BaBaBeEd19a} to the general setting.

\begin{theorem}
\label{thm:utility.max}
	Let $C\colon\R^N\to\R$ be Lipschitz continuous and assume that there exists $\alpha$ such that $U'(\cdot)\leq \alpha(1+|\cdot|^{p-1})$.
	Then, for every $R>0$ there is a constant $K$ (depending only on $R$, $\alpha$ and the Lipschitz constant of $C$) such that  
	\[  \left| u(\fp X) - u (\fp Y) \right|
	\leq K \cdot \mathcal{AW}_p(\fp X,\fp Y)\]
	for every  $\fp X,\fp Y\in\FFP_p$ with $\mathcal{AW}_p(\fp X, 0),  \mathcal{AW}_p(\fp Y, 0) \leq R$.
\end{theorem}
\begin{proof}
	For simplicity, we assume that $C=0$ and focus on $p=1$ (i.e.\ $U$ is Lipschitz); the modifications needed for the general case are minimal and follow e.g.\ as detailed in \cite{BaBaBeEd19a}.
	Let $H^\ast\in\mathcal{H}(\fp X)$ be (almost) optimal for $u(\fp X)$ and let $\pi\in\cplbc(\fp X,\fp Y)$ be (almost) optimal for $\mathcal{AW}_1(\fp X,\fp Y)$.
	Define $ G_t:= \E_\pi[ H_t^\ast |\mathcal{F}^{\fp X, \fp Y}_{0,N}]$ for every $t$.
	Clearly $G$ is bounded by 1, and by bicausality of $\pi$, $G_t$ is $\F^{\fp Y}_t$-measurable (see Lemma \ref{lem:causal CI});  hence $G\in\mathcal{H}(\fp Y)$.
	It follows that
	\[ u(\fp Y)
\geq \E[U((G\cdot Y)_N) ]
 =  \E_\pi[ U( \E_\pi[  (H^\ast \cdot Y)_N  \mid| \F^{\fp X,\fp Y}_{0,N} ] )  ].
\] 
	By concavity of $U$ and Jensen's inequality, followed by Lipschitz continuity of $U$,
	\begin{align*}
	u(\fp Y)
	\geq \E_\pi\left[ U\left(   (H^\ast \cdot Y)_N  \right)  \right]
	&=\E_\pi\left[ U\left(   (H^\ast \cdot X)_N + (H^\ast\cdot (Y-X))_N  \right)  \right]
	\\
	&\geq \E_\pi\left[ U\left(   (H^\ast \cdot X)_N \right)  \right] -  L \E_\pi\left[ |(H^\ast \cdot (Y-X))_N | \right]
	\end{align*}
	where $L$ is the Lipschitz constant of $U$.
	It remains to note that $|(H^\ast\cdot (Y-X))_N |\leq 2\sum_{t=1}^N |X_t-Y_t|$, hence $u(\fp Y)\geq u(\fp X) - 2L\mathcal{AW}_1(\fp X,\fp Y)$.
	Reversing the roles of $\fp X$ and $\fp Y$ proves the claim. 
\end{proof}

\begin{remark}
	As explained in \cite[Section 3.3]{BaBaBeEd19a}, in the context of optimization problems involving stochastic integrals and (semi-)martingales, it is more natural to define $\AW_p$ with a different cost function instead of $\E[d^p(X,Y)]$, namely  with $\E[ |\langle M^{\fp X} - M^{\fp Y} \rangle|^{p/2} + |A^{\fp X} - A^{\fp Y}|_{\rm var}^p]$ where $\fp X = M^{\fp X} + A^{ \fp X}$ denotes the Doob-decomposition, $\langle\cdot\rangle$ the quadratic variation, and $|\cdot|_{\rm var}$ the first variation norm.
	Clearly, this modification of $\AW_p$ is also possible (and reasonable) in the  present setting.
\end{remark}

\subsection{Stochastic control}

Optimal stopping and utility maximization are basic stochastic control  problems involving  processes and we have shown in Proposition \ref{lem:stopping.lipschitz} and Theorem \ref{thm:utility.max} that their values are continuous w.r.t.\ the adapted Wasserstein distance.
As it happens, this is the general principle for  stochastic control problems.

For example, if $J\colon (\R^d)^N\times \mathbb{R}^{N-1}\times (\mathbb{R}^d)^{N} \to \mathbb{R}$ is convex in the second and third argument, a similar reasoning as used for the proof of Theorem \ref{thm:utility.max} shows that under suitable continuity and growth assumptions on $J$,
	\[\fp X\mapsto \inf_{H\in\mathcal{H}(\fp X)} \E\Big[ J\Big(X,((H\cdot X)_t)_{t=2}^N, H \Big)\Big], \]
	is (locally Lipschitz) continuous w.r.t.\ $\AW_p$.
	
	We refer to \cite{BaBaBeEd19a} for a more elaborate analysis of such problems in a mathematical  finance context.
	In particular, using the results of the present paper, the results of \cite{BaBaBeEd19a} for the stability of superhedging, risk based headging and utility indifference pricing (which are formulated only for processes endowed with their raw filtration) extend to processes with arbitrary filtrations.
	
\subsection{Conditional McKean-Vlasov control}

	Denote by $\mathcal{A}(\fp X)$ the set of all $(\F^{\fp X}_t)_{t = 1}^N$-adapted processes that are bounded by 1 and let $\fp B$ be a discrete-time Brownian motion\footnote{That is, $B_1=0$, and for every $t$ the increments $(B_{t+1}-B_t)$ are standard normal and independent of $\mathcal{F}^{\fp B}_t$.}.
	For $\alpha\in\mathcal{A}(\fp B)$ consider the controlled process $\fp X^{\fp B,\alpha}$, defined recursively via $X^{\fp B,\alpha}_1=0$ and 
	\begin{align*}
	X^{\fp B, \alpha}_{t+1}
	:=G_{t+1}\left( X^{\fp B, \alpha}_t,\alpha_t,\mathcal{L}(X_{t}^{\fp B,\alpha} ), B_{t+1}-B_t\right), 	
	\end{align*}
	where $G_{t+1}\colon\R^d\times\R^d\times\mathcal{P}_p(\R^d)\times\R^d\to [0,1]$ is some fixed continuous function prescribing the dynamics of $X^{\fp B,\alpha}$.
	Further let $J\colon (\R^d)^T\times (\R^d)^T\times \mathcal{P}_p((\R^d)^T)\to [0,1]$ be continuous and consider the McKean-Vlasov control problem in its weak formulation:
	\begin{align}
	\label{eq:MCV.control}
	 \inf_{\fp B \text{ is Bronwian motion and }\alpha\in\mathcal{A}(\fp B)} \E\left[ J\left(X^{\fp B,\alpha}, \alpha, \mathcal{L}(X^{\fp B,\alpha}|\mathcal{F}^{\fp B}_t)_{t=1}^N \right) \right];
	 \end{align}
	we refer to \cite{PhWe16} for more background on problems of the type \eqref{eq:MCV.control}.
	Based on Prokhorov's theorem for filtered processes, it is straightforward to show  that a  solution to  \eqref{eq:MCV.control} exists:
	
	\begin{proposition}
	The infimum over all Brownian motions $\fp B$ and controls $\alpha\in\mathcal{A}(\fp B)$ in \eqref{eq:MCV.control} is attained.
	\end{proposition}
	\begin{proof}[Sketch of proof]
	Let $(\fp B^k,\alpha^k)_k$ be a minimizing sequence for \eqref{eq:MCV.control} (in particular, $\alpha^k\in\mathcal{A}(\fp B^k)$), and set $	\fp Y^k := ( \Omega^{\fp B^k}, \F^{\fp B^k}, \P^{\fp B^k}, (\F_t^{\fp B^k})_{t = 1}^N, (B^k,\alpha^k))$.
	By Theorem \ref{thm:rel.compact}, the set $\{\fp Y^k : k\in\N\}$ is relatively compact, hence  there exists 
	\[
		\fp Y = \left( \Omega, \F,\P,(\F_t)_{t = 1}^N, ( B, \alpha) \right)
	\in\FFP_p\]
	 such that (potentially after passing to a subsequence) $\AWA_p(\fp Y^k,\fp Y)\to 0$.
	 It is straightforward to verify that $\fp B =( \Omega, \F,\P,(\F_t)_{t = 1}^N, B)$ is a Brownian motion and that $\alpha\in\mathcal{A}(\fp B)$.
	 Further, since each $G_t$ is continuous,
	\[ 
		(X^{\fp B^k, \alpha^{k}},{\alpha^{k}},(\Law(X^{\fp B^k, \alpha^k} | \F_t^{\fp B^k}))_{t = 1}^N) 
		\to (X^{\fp B, \alpha},  \alpha, (\Law(X^{\fp B, \alpha} | \F_t))_{t = 1}^N) \quad \text{in distribution},
	\]
	as $k\to \infty$.
	The claim now readily follows form continuity of $J$.
	\end{proof}

\subsection{Weak optimal transport}

 Motivated by applications to functional inequalities, Gozlan et.\ al.\ \cite{GoRoSaTe14} introduced the  \emph{weak optimal transport problem}, which extends  classical transport  to `non-linear'  costs $c\colon\R^d\times\mathcal{P}(\R^d)\to[0,\infty)$. Given marginal distributions $\mu_1, \mu_2$ on $\R^d$, the task is  \begin{align}\label{eq:WOT}\tag{\text{WOT}}  \text{minimize} \quad  \E[ c(X_1,\mathcal{L}(X_2|X_1))] \quad \text{over couplings $(X_1,X_2)$ s.t.\ $\Law(X_1)=\mu_1$, $\Law(X_2)=\mu_2$}.
\end{align}
 Weak transport   preserves enough structure from the classical case  to allow for a useful theory while being sufficiently general to  capture many  problems that lie outside the scope of  transport theory, see \cite{BaPa20} for an overview. Cornerstone results in weak optimal transport (existence of optimizers, duality, geometric characterization of optimal couplings,  stability in the data) were shown in the generality of classical transport  only in \cite{BaBePa19, AcBePa20, BeJoMaPa21b}, relying on  two-period results of the current article and the two period version of Theorem \ref{thm:compact}. In analogy to classical transport, the key idea is to relax \eqref{eq:WOT}
 and consider 
 \begin{align}\label{eq:RWOT}
\inf \big\{ \E[ c(X_1,\mathcal{L}(X_2|\F_1))] : \text{$\fp X\in \FP$ satisfies $\Law(X_1)= \mu_1$, $\Law(X_2)=\mu_2$} \big\}.
 \end{align}
 By Theorem \ref{thm:compact} this is an optimization over a compact set which of course admits minimizers under the usual assumption of lower semi-continuity. Indeed \eqref{eq:RWOT} can be viewed as a classical linear optimization problem over probabilities on $\Z_1$.
 
In several  weak transport problems it is natural to consider not just two marginal constraints, but arbitrarily many (relaxed martingale transport \cite{GuOb19}, robust pricing of VIX-futures \cite{GuMeNu17},  model-independence in fixed income markets following \cite{AcBePa20}, multi-marginal Skorokhod embedding  \cite{CoObTo15, BeCoHu16}). In view of this, we propose the $N$-marginal weak transport problem
\begin{align}\label{eqRWOT}
\inf \big\{ \E[ c(\ip (\fp X))] : \text{ $\fp X\in \FP$ satisfies   $\Law(X_1)=\mu_1$,\ldots ,$\Law(X_N)=\mu_N$} \big\}.
 \end{align}
 As above this is an optimization problem over a compact set, which corresponds to a linear optimization problem of probabilities on $\Z_1$, susceptible to classical convex analysis.

\subsection{$\AW_p$-barycenters}

The $\AW_p$-barycenter of two stochastic processes $\fp Z^0, \fp Z^1 \in \FFP_p$ is the minimizer of
\[
	\inf_{\fp X \in \FFP_p} \frac{1}{2} \left( \AW_p^p(\fp Z^0, \fp X) + \AW_p^p(\fp Z^1,\fp X) \right).
\]
Clearly, this minimization problem is attained by the connecting constant speed geodesic at time $1/2$ (which exist thanks to Theorem \ref{thm:geodesic}).
More generally, one can ask whether a distribution on the stochastic processes $\gamma \in \Pc_p(\FFP_p)$ has an $\AW_p$-barycenter $\fp X^\ast \in \FFP_p$, that is, a minimizer of
\begin{equation}
	\label{eq:def barycenter}
	\inf_{\fp X \in \FFP_p} \int \AW_p^p(\fp Z, \fp X) \, \gamma(d\fp Z).
\end{equation}

\begin{theorem}
	\label{cor:barycenter attainment}
	Let $\gamma \in \Pc_p(\FFP_p)$.
	Then there exists $\fp X^\ast \in \FFP_p$ minimizing \eqref{eq:def barycenter}.
\end{theorem}
\begin{proof}
	Corollary \ref{cor:weak lsc} implies that  $(\fp X, \fp Z) \mapsto \AW_p(\fp X,\fp Z)$ is  lower semicontinuity w.r.t.\ the weak adapted topology.
	Hence, it follows that 
	\[
		F(\fp X) := \int \AW_p^p(\fp Z,\fp X) \, \gamma(d\fp Z),	
	\]
	is lower semicontinuity on $\FFP_p$ w.r.t.\ the weak adapted topology.
	Since $F(\fp X) < \infty$ for $\fp X \in \FFP$ if and only if $\fp X \in \FFP_p$, it suffices to show relative compactness of the sublevel set
	\begin{equation}
		\label{eq:sublevel set}
		\Big\{ \fp X \in \FFP \colon F(\fp X) \le \inf_{\fp Y \in \FFP} F(\fp Y) + \varepsilon \Big\},
	\end{equation}
	where $\varepsilon > 0$, in the weak adapted topology.
	Note that the moments of $\fp X \in \FFP_p$ can be controlled as follows:
	\begin{equation} \label{eq:moment control}
	\sum_{t=1}^N\E \left[ |X_t|^p \right] 
		= \AW_p^p(\fp X, 0) \le 2^{p-1}\! \int\! \AW_p^p(\fp Z, \fp X) \!+\! \AW_p^p(0, \fp Z) \, \gamma(dz) = 2^{p-1} \left( F(\fp X) \!+\! F(0) \right).
	\end{equation}
	By Theorem \ref{thm:rel.compact}, relative compactness of the set \eqref{eq:sublevel set} is equivalent to tightness of the laws.
	Tightness of \eqref{eq:sublevel set} follows from standard arguments since the $p$-moments are uniformly bounded by \eqref{eq:moment control} and $\{ x \in \R^d \colon |x| \le K \}$ is compact for $K > 0$.
\end{proof}

\subsection{The Doob-decomposition} \label{ssec:Doob-decomposition}
	Our final example deals with continuity of the Doob-decomposition.
	Recall that the Doob-decomposition ${\fp D}^{\fp X}$ of a filtered process $\fp X$ is given by
	\[ {\fp D}^{\fp X}:= (\Omega^{\fp X}, \mathcal{F}^{\fp X}, (\mathcal{F}^{\fp X}_t)_{t=1}^N, P^{\fp X},  (M_t,A_t)_{t = 1}^N), \]
 where $M+A=X$ is the unique decomposition of $X$ such that $A_1=0$,  $M=M^{\fp X}$ is a martingale and $A=A^{\fp X}$ is  $(\F^{\fp X}_t)_{t=1}^N$-predictable. (Of course, $X$ is a sub-martingale if and only if the process $A$ is increasing.)

\begin{proposition}
\label{prop:doob}
	The following chain of inequalities holds\footnote{The processes $(M,A)$ takes values in $\mathbb{R}^d\times\mathbb{R}^d$ which we endow with the norm $|(x,y)|_p^p :=|x|^p + |y|^p$.}
	\begin{equation} \label{eq:doob inequality}
		2^\frac{1 - p}{p} \AW_p(\fp X,\fp Y) \leq \AW_p( {\fp D}^{\fp X},{\fp D}^{\fp Y}) \leq c \cdot \AW_p(\fp X,\fp Y)
	\end{equation}
	for all $\fp X, \fp Y\in \FFP_p$, where $c = c(p,N)$ is a constant depending only on $p$ and $N$.
\end{proposition}

	Recall that the predictable process $A^\fp X$ of the Doob-decomposition of $\fp X\in \FFP_p$ is given by
	\begin{align*}
		A^{\fp X}_t:=\sum_{s=1}^{t-1} \E[X_s-X_{s-1}|\F^{\fp X}_s]
	\end{align*}
	for $2\leq t\leq N$ and $M^{\fp X}:= X-A^{\fp X}$.
	Thus, $A_t$ can be viewed as an adapted function of rank 1, that is $A_t\in{\AF}[1]$ and similarly $M_t\in{\AF}[1]$. 
	Building upon this observation, it is not hard to deduce $\AW_p$-continuity of $\fp X\mapsto \fp D^{\fp X}$.

\begin{proof}[Proof of Proposition \ref{prop:doob}]
	Fix $\fp X,\fp Y\in \FFP_p$ and note that, as the filtration of a filtered process and its Doob-decomposition coincide by definition, we have that
	\[ \cplba(\fp X,\fp Y)=\cplba(\fp D^{\fp X},\fp D^{\fp Y}).\]
	The first inequality in \eqref{eq:doob inequality} is immediate from Jensen's inequality.
	The triangle inequality together with Jensen's inequality show 
	\begin{equation} \label{eq:doob.1}
		|(M^{\fp X} - M^\fp Y,A^{\fp X} - A^\fp Y)|_p^p \leq 2^{p-1} |X- Y|^p + (2^{p-1}+1) |A^\fp X - A^\fp Y|^p.
	\end{equation}
	By definition of $A^\fp X$ and $A^\fp Y$ we have for $2 \le t \le N$
	\[ 
		A^{\fp X}_t - A^{\fp Y}_t 
		= \sum_{s=2}^{t} \E_\pi\left[X_s-X_{s-1}-(Y_s-Y_{s-1}) \middle|\F^{\fp X,\fp Y}_{s,s}\right].
	\]
	Again, Jensen's inequality implies 
	\begin{align} \nonumber
		\E_\pi[ |A^{\fp X}  - A^{\fp Y}|^p]^\frac{1}{p} &\leq \Big( \sum_{t=2}^N t^{p-1} \sum_{s=2}^{t} \E_\pi[ |X_s-X_{s-1}-(Y_s-Y_{s-1})|^p] \Big)^\frac{1}{p} \\  \label{eq:doob.2}
		&\leq \big( N 2^p N^{p-1} \E_\pi[ |X - Y|^p] \big)^\frac{1}{p}
	\end{align}
	for every $\pi\in \cplba(\fp X,\fp Y)$.
	In conclusion, suitably combining \eqref{eq:doob.1} and \eqref{eq:doob.2} leads to the second inequality in \eqref{eq:doob inequality}.
\end{proof}

\subsection{Numerical and statistical aspects}\label{ssec:numerics}
In view of applications, the numerical computation of the adapted Wasserstein distance is of crucial importance and has been recently studied in \cite{PiWe21, EcPa22, BaHa23}.
In fact, these papers consider more general adapted transport problems between  stochastic processes $\fp X$ and $\fp Y$ of the form
\begin{equation}
    \label{eq:adapted.transport.problem}
    \inf_{ \pi \in \cplbc(\fp X, \fp Y) } \mathbb E_\pi \left[ c(X,Y) \right],
\end{equation}
where $c$ is a continuous function.
In order to briefly explain the methodology, suppose  that the processes $\fp X$ and $\fp Y$ are both discrete, i.e.\ plain and only take finitely many values.
(We remark that for general processes $\fp X$ and $\fp Y$ that are not necessarily discrete, one may replace them first by discrete approximations e.g.\ as in  Proposition \ref{prop:markov.dense}).
In this case, solving \eqref{eq:adapted.transport.problem} is equivalent to solving a linear program (see e.g.\  \cite[Section 8]{PfPi12} and \cite[Section 3.4]{EcPa22} for more details).
The recursive structure of bicausal couplings allows to solve the linear program via a dynamic programming principle (DPP) involving one-step classical optimal transport problems (see \cite[Chapter 2.10.3]{PfPi14}).
In the current setting, the value functions $V_t \colon \Z_{1:t} \times \Z_{1:t} \to \R$ are recursively given by
\begin{align*}
	V_N(\hat z_{1:N}, \check z_{1:N} ) := c(\hat z_{1:N}^-, \check z_{1:N}^- ), \quad V_t(\hat z_{1:t}, \check z_{1:t} ) := \inf_{\pi \in \cpl(\hat z^+_t,\check z_t^+)} \E_\pi[V_{t + 1}[\hat z_{1:t}, \cdot, \check z_{1:t}^-, \cdot]],
\end{align*}
for $t = N-1,\ldots, 0$ with the convention that $\hat z_0^+ := \Law( \ip_1(\fp X))$ and $\check z_0^+ := \Law(\ip_1(\fp Y))$.
Then $V_0$ is precisely the optimal value \eqref{eq:adapted.transport.problem} and the values of $V_t$ can be computed efficiently e.g.\ via Sinkhorn's algorithm (see \cite{PiWe21}).
Based on the DPP formulation, a fitted value iteration method is proposed in \cite{BaHa23}.
There the value function of the DPP is iteratively empirically estimated and then approximated by neural networks that minimize an empirical loss.
Another recent approach builds on establishing a bicausal version of the Sinkhorn algorithm (see \cite{EcPa22}).
For comprehensive comparisons of these methods, we refer to \cite{EcPa22} and \cite{BaHa23}.
Furthermore, in \cite{XuLiMuAc20}, a minimax reformulation is proposed where the infimum is taken over all couplings and the supremum over an adequate class of functions penalizing the causality constraint. Both families of functions are then parametrized by neural networks and trained by an adversarial algorithm. In a continuous time framework, \cite{BiTa19} establishes a Hamilton-Jacobi-Bellman equation for the value function.

The estimation of processes w.r.t.\ $\mathcal{AW}_p$ from statistical data was studied in \cite{BaBaBeWi20} under the assumption of bounded values and later extended in \cite{AcHo22} (see also  \cite{PfPi16, GlPfPi17}).
Roughly put, while the classical empirical measure does not converge to its population counterpart in the weak adapted topology (nor do e.g.\ the values of the corresponding optimal stopping problems) one can  construct a \emph{modified empirical measure} (based on clustering ideas) which does converge in the weak adapted topology.
In fact, the rate of convergence of that estimator (w.r.t.\ $\AW_p$) is similar to the (optimal) rate of convergence of the  classical empirical measure (w.r.t.\ $\mathcal{W}_p$).

\section{An important example}
\label{sec:two.processes}

Fix $N\geq 2 $ and recall that, by Theorem \ref{thm:information.process.AF}, the relation $\sim_{N-1}$ is equal to the relation $\sim_\infty$, which again coincides with the relation induced by $\AW_p$.

A natural question left open in Section \ref{sec:AF.PP} and Section \ref{sec:aspects} is what rank of the adapted distribution is necessary to govern probabilistic properties of stochastic processes.
For instance, we have seen in Example \ref{ex:martingales.opt.stop} that the adapted distribution of rank 1 suffices for preserving the notion of being a martingale.
However, we will see that rank 1 equivalence is not sufficient in general.
Rather Theorem \ref{thm:information.process.AF} is sharp: $\sim_{N-2}$ does not imply $\sim_{N-1}$. 
In fact, we show that there exist filtered processes that are $\sim_{N-2}$ equivalent but lead to different values for an optimal stopping problem. 

\begin{theorem}
\label{thm:stopping.rank} Let $N\geq 2$. 
	There exist $\fp{X},\fp{Y}\in \FFP_p$ with $\fp X \sim_{N-2} \fp Y$  and a bounded, continuous non-anticipative function $c\colon\mathcal{X}\times\{1,\dots,N\}\to \mathbb{R}$ such that $v_c(\fp X) \neq v_c(\fp Y)$ (compare \eqref{eq:def.opt.stop}  for the definition of $v_c$).
\end{theorem}

In particular, using Proposition \ref{lem:stopping.lipschitz} and Theorem \ref{thm:information.process.AF}, this implies that $\fp X \sim_{N-1} \fp Y$.
More generally, we will show in the following that as long as $k\leq N-1$, the relation $\sim_k$ strictly refines the relation $\sim_{k-1}$.

\begin{proposition}
\label{prop:process.different.rank}
	Let $N\geq 2$ and let $1\leq k \leq N-1$.
	Then there are $\fp X, \fp Y \in \FFP_p$ with representatives defined on the same filtered probability space such that
	\begin{enumerate}[label=(\roman*)]
	\item \label{it:process.different.rank1} $f(\fp{X})=f(\fp{Y})$  for every $f\in\AF[k-1]$,
	\item \label{it:process.different.rank2} $\E[f(\fp{X})]\neq \E[f(\fp{Y})]$ for some  $f\in\AF[k]$.
	\end{enumerate}
\end{proposition}

The processes $\fp X$ and $\fp Y$ will be constructed on a common probability space consisting of $2^{N-1}$ atoms, satisfy $X=Y$ and $X_t=Y_t=0$ for $t\leq N-1$, and only their respective filtrations will differ.

\subsection{Proof of Proposition \ref{prop:process.different.rank}}

This section is devoted to the proof of Proposition \ref{prop:process.different.rank} which will then be used to establish Theorem \ref{thm:stopping.rank}.
In fact, we shall first concentrate on Proposition \ref{prop:process.different.rank} with $k=N-1$, and later conclude for general $k$ via a simple argument.

The construction of $\fp X $ and $\fp{Y}$ is recursive, and we shall start with $N=2$.
Consider the probability space consisting of two elements, say $\{0,1\}$, with the uniform measure, and let $U$ be the identity map, that is, $\P(U = 1) = \P(U = 0)=1/2$.
Now define the filtered processes $\fp X$ and $\fp Y$ via
\begin{align*}
   X:=Y:=(0,U), \quad 
   \F_1^\fp{X}:=\{\emptyset,\Omega\} \text{ and } \F_1^\fp{Y}:=\F_2^\fp{Y}:=\F_2^\fp{X}:=\sigma(U)
\end{align*}
Then the following holds.

\begin{lemma}
	$\fp{X}$ and $\fp{Y}$ satisfy the claim in Proposition \ref{prop:process.different.rank} for $N=2$.
\end{lemma}
\begin{proof}
	Adapted functions of rank $0$ depend only on the process itself and not on the filtration, hence  $f(\fp{X})=f(\fp{Y})$ for all $f\in\AF[0]$.
	That is, part \ref{it:process.different.rank1} of Proposition \ref{prop:process.different.rank} is true.
	On the other hand, one computes
	\begin{align*}
	   \E\left[g(X_2)|\F_1^\fp{X}\right]=\E[g(U)],\quad \quad
	   \E\left[g(X_2)|\F_1^\fp{Y}\right]=g(U)
	\end{align*}
	for every bounded, measurable function $g\colon\mathbb{R}\to\mathbb{R}$.
	Now define $f\in\AF$ by
	\begin{align*} 
	f:=\varphi((g|1)), \quad
	\text{ where }
	\varphi=t\mapsto t^2\wedge 1 \text{ and }g:=x\mapsto (0 \vee x) \wedge 1
	\end{align*}
	so that $f$ has rank $1$.
	Then
	\begin{align*}
	f(\fp{X})=\E\left[X_2|\F_1^\fp{X}\right]^2=1/4, \quad \quad
	f(\fp{Y})=\E\left[X_2|\F_1^\fp{Y}\right]^2=X_2.
	\end{align*}
	As $\E[f(\fp{Y})]=1/2$, this proves the second part of Proposition \ref{prop:process.different.rank}.
\end{proof}

Now assume that, in a $(N-1)$-time step framework, we have constructed filtered processes $\fp{X}$ and $\fp{Y}$ which satisfy Proposition \ref{prop:process.different.rank} (where $N\geq 3$).
We shall construct new filtered processes $\fp{X}^\mathrm{n}$ and $\fp{Y}^\mathrm{n}$ in an $N$-time step framework for which the statement of the lemma remains true.
At this particular instance the superscript `$\mathrm{n}$' stands for `new'.
 
At this point, we enrich the filtered probability space of $\fp X$ and $\fp Y$ by an independent random variable $V$  (i.e.\ $V$ is independent of $\F_{N-1}^\fp{X}\vee\F_{N-1}^\fp{Y}$) in the following manner:
When $(\Omega,\F,\P)$  denotes the probability space of $\fp X$ and $\fp Y$, we consider the new probability space
\begin{equation} \label{eq:def prob space}
	\Big( \Omega^\mathrm{n} := \{0,1\} \times \Omega \,,\,
	\F^\mathrm{n}:= \mathcal B(\{0,1\}) \otimes \F \,,\,
	\P^\mathrm{n}	:= \frac{1}{2} \left( \delta_0 \otimes \P + \delta_1 \otimes \P \right) \Big).
\end{equation}
   We write $V \colon \Omega^\mathrm{n} \to \{0,1\}$ for the projection onto the first coordinate. 
   In other words, we can think of $V$ as a coin-flip which is independent of everything that was constructed so far.
   Further we view $(\mathcal{G}_t^\fp X)_{t = 1}^N:=(\F_{t-1}^\fp X)_{t  = 1}^{N-1}$ as a filtration on \eqref{eq:def prob space}; similarly for $(\mathcal{G}^{\fp Y}_t)_{t=1}^N$ (recall here that $\F_{0}^\fp X=\F_{0}^\fp Y$ are the trivial $\sigma$-algebras by convention).
   Then, by slight abuse of notation, we can view $\fp X$ and $\fp Y$ as processes defined on $\Omega^{\mathrm{n}}$, namely be identifying $\fp X$ with $(0,X)$ and filtration $(\mathcal{G}^{\fp X}_t)_{t=1}^N$ and similarly for $\fp Y$.
   By independence of $V$, this clearly does not affect any properties of the processes.
   Define the processes $\fp{X}^{\mathrm{n}}$ and $\fp{Y}^{\mathrm{n}}$ on the probability space \eqref{eq:def prob space} via 
\begin{align*}
   X^{\mathrm{n}}
   :=&Y^{\mathrm{n}}
   :=(0,X)\\
   \F^{\fp{X}^\mathrm{n}}_1
   :=&\{\emptyset,\Omega^{\mathrm{n}}\} 
   \text{ and } \F^{\fp{Y}^\mathrm{n}}_1:=\sigma(V),\\
   \F^{\fp{X}^\mathrm{n}}_t:=\F^{\fp{Y}^\mathrm{n}}_t
   :=&\left\{ (\{0\}\times A) \cup (\{1\}\times B): A\in\mathcal G_t^\fp{X}, B\in \mathcal G_t^\fp{Y} \right\} \\
   =&\left\{ \left(A\cap \{V = 0\} \right) \cup \left(B\cap \{V = 1\}\right) : A\in\mathcal G_t^\fp{X}, B\in \mathcal G_t^\fp{Y} \right\}
\end{align*}
for $2 \leq t\leq N$.
One can check that $(\F^{\fp{X}^\mathrm{n}}_t)_{t=1}^N$ and $(\F^{\fp{Y}^\mathrm{n}}_t)_{t=1}^N$ are indeed filtrations and that, for every $\P^\mathrm{n}$-integrable random variable $Z$ which is independent of $V$, we have that
\begin{align}
\label{eq:new.filtration.cond}
\begin{split}
\E\left[Z \middle| \F^{\fp{X}^\mathrm{n}}_t\right]
&=\E\left[Z \middle| \mathcal{G}^\fp{X}_t \right]1_{\{V = 0\}} + \E\left[Z \middle| \mathcal{G}^\fp{Y}_t \right]1_{\{V= 1\}} 
=\E\left[Z\middle|\F^{\fp{Y}^\mathrm{n}}_t\right]
\end{split}
\end{align}
for every $1 \le t \le N$.
We spare the elementary proof hereof.

\begin{lemma}
\label{lem:AF.AF.rank.N-3}
	For $0\leq k\leq N-3$ and $f \in \AF[k]$ we have
	\begin{equation}
	   \label{eq:rank.differnt.periods}
	   f(\fp{X}^\mathrm{n})=f(\fp{X})=f(\fp{Y})=f(\fp{Y}^\mathrm{n}) .
	 \end{equation}
	In particular, $\fp X^\mathrm{n} \sim_{N-3} \fp Y^\mathrm n$.
\end{lemma}
\begin{proof}
   The proof is by induction over  $k$:
	for $k=0$, the statement is trivially true (recalling that $X^\mathrm{n}=(0,X)=Y^\mathrm{n}$).
	
	Now assume that it is true for $k-1$ (with $k\leq N-3$), and let $f\in\AF[k]$.
	First assume that $f=(g|t)$ is formed by \ref{it:AF3} only.
	By assumption we have
	\begin{equation} \label{eq:AF.rank.N-3.1}
	   g(\fp{X}^\mathrm{n})=g(\fp{X})=g(\fp{Y})=g(\fp{Y}^\mathrm{n}).
	\end{equation}
	Now use independence of $V$ and $\F_{N}^\fp{X}\vee\F_{N}^\fp{Y}$ and \eqref{eq:AF.rank.N-3.1} to compute
	\begin{align*}
		f(\fp{X}^\mathrm{n}) 
		= \E[g(\fp{X})|\F_t^{\fp{X}^\mathrm{n}}] 
		&=\E\left[ g(\fp X) \middle| \mathcal{G}_t^\fp X \right] 1_{\{V=0\}} + \E\left[ g(\fp Y) \middle| \mathcal{G}_t^\fp Y \right] 1_{\{V=1\}} \\
		&= f(\fp X) 1_{\{V=0\}} + f(\fp Y) 1_{\{V=1\}}  = f(\fp X) = f(\fp Y),
	\end{align*}
   where the last three equalities are due the fact that $\fp{X}$ and $\fp{Y}$ satisfy Proposition \ref{prop:process.different.rank}, that is, $f(\fp{X})=f(\fp{Y})$.
   The same computation is valid for $\fp Y^\mathrm{n}$, whence we have \eqref{eq:rank.differnt.periods} for this particular choice of $f$.
	Finally, by Lemma \ref{lem:rep.AF}, this extends to all $f\in\AF[k]$, not necessarily those formed solely by \ref{it:AF3}.
\end{proof}

\begin{lemma} \label{lem:identity for N-2}
   For $f \in \AF[N-2]$ we have
   \begin{equation}
	   \label{eq:identity for N-2}
	   f(\fp X^\mathrm n) = f(\fp X) 1_{\{V = 0\}} + f(\fp Y) 1_{\{V = 1\}} = f(\fp Y^\mathrm n).
   \end{equation}
\end{lemma}

\begin{proof}
	First, consider $g \in \AF[N-3]$, $1 \le t \le N$, $f = (g|t)$.
	We compute $f(\fp X^\mathrm n)$ by applying \eqref{eq:new.filtration.cond} and Lemma \ref{lem:AF.AF.rank.N-3}
	\begin{align*}
		f(\fp X^\mathrm n) 
		&= \E\left[ g(\fp X^n) \middle| \F_t^{\fp X^\mathrm n} \right] 
		= \E\left[ g(\fp X) \middle| \F_t^{\fp X^\mathrm n} \right] \\
		&= \E\left[ g(\fp X) \middle| \mathcal{G}_t^\fp X \right] 1_{\{V=0\}} + \E\left[ g(\fp Y) \middle| \mathcal{G}_t^\fp Y \right] 1_{\{V=1\}} 
		= f(\fp X) 1_{\{V=0\}} + f(\fp Y) 1_{\{V=1\}}.
	\end{align*}
	We conclude by noticing that the same computation holds true when replacing $\fp X^\mathrm n$ by $\fp Y^\mathrm n$.
\end{proof}

\begin{lemma}
	$\fp{X}^\mathrm{n}$ and $\fp{Y}^\mathrm{n}$ satisfy the claim in Proposition \ref{prop:process.different.rank} with $k = N-1$.
\end{lemma}
\begin{proof}
   Part \ref{it:process.different.rank1} of Proposition \ref{prop:process.different.rank} is a consequence of Lemma \ref{lem:identity for N-2}.
	We proceed to prove part \ref{it:process.different.rank2} of Proposition \ref{prop:process.different.rank}, that is, we construct $f \in \AF[N-1]$ such that 
	\[ \E[f(\fp{X}^\mathrm{n})]\neq \E[f(\fp{Y}^\mathrm{n})].\]
   Since $\fp X \nsim_{N-2} \fp Y$, there is a function $g \in \AF[N-2]$ such that $\E[g(\fp X)] \neq \E[g(\fp Y)]$.
   Set $f:=\min\{(g|1)^2, 1\}\in\AF[N-1]$.
   Using identity \eqref{eq:identity for N-2} from Lemma \ref{lem:identity for N-2} and \eqref{eq:rank.differnt.periods}, we compute
	\begin{align*}
	   f(\fp X^\mathrm n) 
	   &= \E \left[ g(\fp X) 1_{\{V=0\}} + g(\fp Y) 1_{\{V=1\}} \right]^2, \\
	   f(\fp Y^\mathrm n) 
	   &= \E \left[ g(\fp X) 1_{\{V=0\}} + g(\fp Y) 1_{\{V=1\}} \middle| V \right]^2 \\
	   &= \left( \E[g(\fp X)] 1_{\{V=0\}} + \E[g(\fp Y)] 1_{\{V=1\}}\right)^2.
   \end{align*}
	Recalling that $\E[g(\fp X)] \neq \E[g(\fp Y)]$ this implies $\E[f(\fp X^\mathrm n)] < \E[f(\fp Y^\mathrm n)]$.
\end{proof}

At this point we have completed the proof of Proposition \ref{prop:process.different.rank} under the additional assumption that $k=N-1$.
For general $k\leq N-1$, construct $\fp X$ and $\fp Y$ for the $(k - 1)$-time step framework as above.
Recursively repeating the argument detailed below \eqref{eq:def prob space}, we can append $(N - k)$ trivial time steps to the processes $\fp X$ and $\fp Y$, and thereby obtain $(N-1)$-time step processes with the desired properties.
This proves Proposition \ref{prop:process.different.rank} for arbitrary $k\leq N-1$. 

\subsection{Proof of Theorem \ref{thm:stopping.rank}}

As already announced, we use the processes $\fp{X}$ and $\fp{Y}$ (and the notation specific to these processes) constructed for the proof of Proposition \ref{prop:process.different.rank}

The proof will be inductive, starting with $N=2$. Consider the cost function $c \colon \R^2 \times \{1,2\} \to \R$,
\begin{align*}
   c_1(x_1,x_2) & := 1/2 \quad\text{and}\quad
   c_2(x_1,x_2) := (0\vee x_2) \wedge 1 .
\end{align*}
The dynamic programming principle for the optimal stopping problem (also called `Snell envelope theorem') tells us that the optimal stopping values are
\[ v_c(\fp X) = \E\left[c_1(X)\wedge \E[c_2(X)|\F_1^{\fp X}] \right] = 1/2 \text{ and } v_c(\fp Y) = 1/4. \]
This proves the claim for $N=2$.

For the case of general $N$, recall that $\fp{X}^\mathrm{n}$ and $\fp{Y}^\mathrm{n}$ are the processes obtained through the $(N-1)$-step processes $\fp{X}$ and $\fp{Y}$ and an independent coin-flip $V$.

By the previous step, we assume that in a $(N-1)$-step framework we have constructed a cost function $c$ such that $v_c(\fp{X})>v_c(\fp Y)$.
Defining the new cost function $c^{\mathrm{n}}$ in the $N$-step framework by 
\begin{align*}
   c^\mathrm{n}_1(x_{1:N})	&:=  (v_c(\fp{X})+v_c(\fp{Y})) /2,\\
   c^\mathrm{n}_t(x_{1:N}) &:= c_{t-1}(x_{2:N}) \quad\text{ for } 2\leq t\leq N,
\end{align*} 
we claim that $v_{c^\mathrm{n}}(\fp{X}^\mathrm{n})>v_{c^\mathrm{n}}(\fp{Y}^\mathrm{n})$.

To that end, we once more rely on the dynamic programming principle to obtain
\begin{equation}
   \label{eq:opt.stop.dyn.prog.Y}
   v_{c^\mathrm{n}}(\fp{Y}^\mathrm{n})	
   =\E\Big[ c^\mathrm{n}_1 \wedge \inf_{ \tau \in\mathrm{ST}(\fp{Y}^\mathrm{n}),\, \tau\geq 2 } \E\left[c^\mathrm{n}_\tau(Y^\mathrm{n}) \middle|\F^{\fp{Y}^\mathrm{n}}_1 \right] \Big].
\end{equation}
Similarly, additionally using that $\F_1^{\fp{X}^\mathrm{n}}$ is trivial and $c^\mathrm{n}_1$ is deterministic, we get
\begin{equation}
   \label{eq:opt.stop.dyn.prog.X}
   v_{c^\mathrm{n}}(\fp{X}^\mathrm{n}) = c^\mathrm{n}_1 \wedge \inf_{ \tau\in\mathrm{ST}(\fp{X}^\mathrm{n}) ,\, \tau\geq 2 } \E[c^\mathrm{n}_\tau(X^\mathrm{n}) ].
\end{equation}

In order to relate the conditional stopping problems after time 2 appearing in \eqref{eq:opt.stop.dyn.prog.Y} and \eqref{eq:opt.stop.dyn.prog.X}, recall that we view $\fp X$ and $\fp Y$ as $(N-1)$-time step processes by appending a trivial initial time step; see the explanation right after \eqref{eq:def prob space}.
   The same convention is applied here to $c$ as well.

   In a first step we show the following decomposition of stopping times:
\begin{equation} \label{eq:stopping times decomposition}
   \mathrm{ST}(\fp Y^\mathrm n) 
   = \left\{ \alpha 1_{\{V=0\}} + \beta 1_{\{V=1\}} \colon \alpha \in \mathrm{ST}(\fp X), \beta \in \mathrm{ST}(\fp Y)\right\}.
\end{equation}
The right-hand side is clearly contained in the left-hand side.
For the reverse inclusion, pick $\tau \in \mathrm{ST}(\fp Y^\mathrm{n})$. 
By definition of $\F_t^{\fp Y^\mathrm n}$, there are sets $A_t\in \mathcal{G}_t^\fp X$ and $B_t\in\mathcal{G}_t^\fp Y$ such that
\[
	\{\tau = t\} = A_t \cap \{V = 0\} \cup B_t \cap \{V=1\}
\]
for every $1\leq t\leq N$.
   One can then check that 
\begin{align*}
	\alpha := \min \{ 1 \le t \le N \colon 1_{A_t} = 1 \},\quad 
	\beta := \min\{ 1 \le t \le N \colon 1_{B_t} = 1\}
\end{align*}
define stopping times in $\mathrm{ST}(\fp X)$ and $\mathrm{ST}(\fp Y)$, respectively, and that $\tau = \alpha 1_{\{V=0\}} + \beta 1_{\{V=1\}}$.
This shows \eqref{eq:stopping times decomposition}.

We are now ready to finish the proof.
Recalling that $\F_1^{\fp X^\mathrm n}$ is the trivial $\sigma$-algebra and that $\F_{t}^{\fp X^\mathrm n}= \F_t^{\fp Y^\mathrm n}$ for $2\leq t\leq N$, independence of $V$ and $\mathcal{G}_N^\fp X \vee \mathcal{G}_N^\fp Y$ and the decomposition of stopping times \eqref{eq:stopping times decomposition} shows that
\begin{align*}
   \inf_{\tau\in\mathrm{ST}(\fp{X}^\mathrm{n}),\,\tau\geq 2} \E\left[c^\mathrm{n}_\tau(X^\mathrm{n})\right]
   &=\inf_{(\alpha,\beta) \in \mathrm{ST}(\fp{X})\times\mathrm{ST}(\fp{Y})} \E\left[c_\alpha(X) 1_{\{V=0\}} + c_\beta(Y)1_{\{V=1\}}\right]\\
   &=\left(v_{c}(\fp{X})+v_c(\fp{Y})\right)/2 =c_1^\mathrm{n}.
\end{align*}
In a similar manner
\begin{align*}
\inf_{\tau\in\mathrm{ST}(\fp{Y}^\mathrm{n}),\,\tau\geq 2} \E\left[c^\mathrm{n}_\tau(X^\mathrm{n})\middle|\F^{\fp{Y}^\mathrm{n}}_1\right]
   &=\inf_{\alpha \in \mathrm{ST}(\fp{X}) } \E\left[c_{\alpha}(X)\right] 1_{\{V=0\}}+ \inf_{\beta\in \mathrm{ST}(\fp{Y})} \E\left[c_{\beta}(Y)\right] 1_{\{V=1\}}\\
   &=v_c(\fp{X})1_{\{V=0\}} + v_c(\fp{Y})1_{\{V=1\}}.
\end{align*}
As $v_c(\fp{Y})<c_1^\mathrm{n}<v_c(\fp{X})$, plugging the above equalities in \eqref{eq:opt.stop.dyn.prog.Y} and \eqref{eq:opt.stop.dyn.prog.X} readily shows $v_{c^\mathrm{n}}(\fp{Y}^\mathrm{n})<v_{c^\mathrm{n}}(\fp{X}^\mathrm{n})$ and thus completes the proof.

\vspace{1em}
\noindent
\textsc{Acknowledgements:}
Daniel Bartl is grateful for financial support through the Austrian Science Fund (FWF) projects ESP-31N and P34743N.
Mathias Beiglb\"ock  is grateful for financial support through the Austrian Science Fund (FWF) projects Y0782 and P35197. 

 
\bibliographystyle{abbrv} 


\appendix

\section{The adapted block approximation}\label{sec:block approximation} 

We have seen in Lemma \ref{lem:process to canonical} that bicausal couplings $\pi$ on arbitrary filtered probability spaces induce bicausal couplings in the canonical setting.
In order to establish that the association of $\fp X \in \FP_p$ with its canonical representative $\cfp X \in \CFP_p$ is an isometry w.r.t.\ $\AW_p$ (see Subsection \ref{ssec:isometry}), we have to find for $\overline{\pi} \in \cplba(\cfp X,\cfp Y)$ a similar coupling $\pi \in \cplba(\fp X, \fp Y)$.
For this reason, we introduce in this section what we call the \emph{adapted block approximation} in Proposition \ref{prop:adapted block} (for the canonical filtered setting).
The main result of this section is Theorem \ref{thm:adapted block}, which allows us then to pull-back block approximations of elements in $\cplba(\cfp X,\cfp Y)$ to $\cplba(\fp X,\fp Y)$.

We start with a characterization of bicausality in terms of kernels for canonical filtered processes.

\begin{lemma} \label{lem:bicausal cfp rep}
Let $\fp X, \fp Y \in \CFP_p$, let $\pi \in \Pc_p(\Z \times \Z)$, and set $\pi_1 := ( \proj_{\Z_1 \times \Z_1} )_\ast\pi$.
	Then $\pi \in \cplba(\fp X,\fp Y)$ if and only if 
		\[\pi_1\in \cpl(\Law(\ip_1(\fp X)),\Law(\ip_1(\fp Y)))\] 
		and, for $1 \le t \le N-1$, there are kernels
		\[k_t \colon \Z_{1:t} \times \Z_{1:t} \to \Pc_p(\Z_{t+1} \times \Z_{t+1}) \text{ with }k_t^{z_{1:t},\hat z_{1:t}} \in \cpl(z_t^+,\hat z_t^+)\]
		such that
		\begin{equation} \label{eq:bicausal cpl kernel rep}
			\pi =  \pi_1 \otimes k_1  \ldots  \otimes k_{N-1}.
		\end{equation}
\end{lemma}
\begin{proof}
	Clearly any coupling can be represented by a family of measurable kernels $(k_t)_{t=1}^{N-1}$ as in \eqref{eq:bicausal cpl kernel rep} with 
	\[
		k_t = \Law_\pi \big( z_{t+1}, \hat z_{t+1} | \F_{t,t}^{\Z,\Z} \big).
	\]
	The only thing we need to show is that $k_t^{z_{1:t},\hat z_{1:t}} \in \cpl(z_t^+,\hat z_t^+)$ for every $1\leq t\leq N-1$ if and only if $\pi$ is bicausal.
	But this follows from Lemma \ref{lem:causal CI} and Lemma \ref{lem:unfoldprops.measures}.
	Indeed, by these lemmas, we have that $\pi$-almost surely
	\begin{align*}
		\Law_\pi \big(  z_{t+1} | \F_{t,t}^{\Z,\Z} \big) 
		&= \Law_{\Law(\ip(\fp X))} \left( z_{t+1} | \F_{t}^{\Z} \right) 
		= z_t^+, \\
		\Law_\pi \big(  \hat z_{t+1} | \F_{t,t}^{\Z,\Z} \big) 
		&= \Law_{\Law(\ip(\fp Y))} \left( \hat z_{t+1} | \F_{t}^{\Z} \right) 
		= \hat z_t^+
	\end{align*}
	for all $1 \le t \le N-1$, if and only if $\pi \in \cplba(\fp X,\fp Y)$.
	This completes the proof.
\end{proof}

\begin{lemma}
	\label{lem:adapted block}
	Let $1 \le t \le N-1$, let $\pi \in \Pc_p(\Z_{1:t} \times \Z_{1:t})$, and  let
	\[k \colon \Z_{1:t} \times \Z_{1:t} \to \Pc_p(\Z_{t+1} \times \Z_{t+1}) 
	\text{ with } k^{z_{1:t},\hat z_{1:t}} \in \cpl(z_t^+,\hat z_t^+).\]
	Further let $(\pi^n)_{n \in \N}$ in $\Pc_p(\Z_{1:t} \times \Z_{1:t})$ such that $\W_p(\pi,\pi^n)\to 0$.
	Then there are kernels
	\begin{align}
		\label{eq:adapted block lemma1}
		k^n \colon \Z_{1:t} \times \Z_{1:t} \to \Pc_p(\Z_{t+1} \times \Z_{t+1})
		\text{ with } k^{n,z_{1:t},\hat z_{1:t}}\in \cpl(z_t^+,\hat z_t^+)
	\end{align}
	such that 
	\begin{align}
		\label{eq:adapted block lemma2}
		\W_p(\pi\otimes k , \pi^n \otimes k^n) \to 0.
	\end{align}
\end{lemma}

\begin{proof}
	In this proof we deal with the spaces $(\Z_{1:t}\times \Z_{1:t})\times (\Z_{1:t}\times \Z_{1:t})$ and $(\Z_{t+1} \times \Z_{t+1}) \times (\Z_{t+1} \times \Z_{t+1})$.
	For the sake of a clearer presentation we baptise the first product space by $\mathcal{A} \times \mathcal{B}$ and the second one by $\mathcal C \times \mathcal E$.
	We write $a = (\hat a,\check a) = (\hat a_{1:t}, \check a_{1:t})$ for elements in $\mathcal{A}$ (the space in the first bracket) and $\hat a_t = (\hat a_t^-,\check a_t^+)$ as well as $\check a_t = (\check a_t^-,\check a_t^+)$.
	Similar conventions apply to elements in the spaces $\mathcal{B}, \mathcal C , \mathcal E$.
	
	For every $n$, let $\Pi^n$ be an optimal coupling for $\W_p(\pi,\pi^n)$ and denote by $K^n$ its disintegration w.r.t.\ $\pi^n$, that is,
	\[ \Pi^n(da,db) = \pi^n(d b) \, K^{n,b}(da).\]
	
	In a first step, we define an auxiliary kernel $\tilde k^n \colon \mathcal B \to \Pc_p(\mathcal C)$, which is $\pi^n$-almost surely well-defined, for $b \in \mathcal B$ by
	\[
		\tilde k^{n,b}(\cdot)
		:= \int k^a(\cdot) \, K^{n,b}(da) \in \Pc_p(\mathcal C).
	\]
	In general, $\tilde k^{n,b}$ needs not to be an element of $\cpl(\hat b_t^+,\check b_t^+)$.
	To amend this, we pick measurable kernels $K^1 \colon \mathcal B \to \Pc_p(\mathcal C)$ and $K^2 \colon \mathcal B \to \Pc_p(\mathcal E)$ where
	\begin{align*}
		K^{1,b} &\text{ is an optimal coupling for }\W_p((\proj_1)_\ast \tilde k^{n,b},\hat b_{t}^+), \\
		K^{2,b} &\text{ is an optimal coupling for }\W_p(\check b_t^+, (\proj_2 )_\ast \tilde k^{n,b}).
	\end{align*}
	A disintegration of $K^{1,b}$ w.r.t.\ the first coordinate is denoted by $(K^{1,b,\hat c})_{\hat c \in \Z_{t+1}}$, and similarly a disintegration of $K^{2,b}$ w.r.t.\ the second coordinate is called $(K^{2,b,\check e})_{\check e \in \Z_{t+1}}$.
	Proceeding from this, we can define the kernel $k^n \colon \mathcal B \to \Pc_p(\mathcal C)$ with $k^{n,b} \in \cpl(\hat b_t^+,\check b_t^+)$
	\[
		k^{n,b}(\cdot) := \int K^{1,b,\hat c}(\cdot) \otimes K^{2,b,\check e}(\cdot) \, d\tilde k^{n,b}(d\hat c,d\check e).
	\]
	From here we get
	\begin{align} \nonumber
		\W_p^p(k^{n,b},\tilde k^{n,b}) &\le \iint d^p( \hat c, \check c ) + d^p(\hat e, \check e) \, dK^{1,b,\hat c} \otimes K^{2,b,\check e} \, \tilde k^{n,b}(d\hat c, d \check e)\\ 
		&= \int d^p(\hat c,\check c) \, dK^{1,b} + \int d^p(\hat e,\check e) \, dK^{2,b} 
		\label{eq:adapted block lemma kernel est1}
		= \W_p^p( \proj_1 \tilde k^{n,b}, \hat b_t^+ ) + \W_p^p ( \check b_t^+,  \proj_2 \tilde k^{n,b} ).
	\end{align}
	Due to Jensen's inequality we obtain
	\begin{align} \label{eq:adapted block lemma kernel est2}
	\begin{split}
		&\left( \int \W_p^p( \proj_1 \tilde k^{n,b}, \hat b_t^+) + \W_p^p ( \proj_2 \tilde k^{n,b}, \check b_t^+ ) \, \pi^n (db)\right)^\frac{1}{p} \\
		&\qquad \le \Big( \int \W_p^p(\hat a_t^+, \hat b_t^+  ) + \W_p^p ( \check a_t^+, \check b_t^+ ) \, K^n(da,db)\Big)^\frac{1}{p} \le \W_p(\pi^n,\pi).
		\end{split}
	\end{align}
	It remains to verify that the sequence $(k^n)_{n \in \N}$ satisfies \eqref{eq:adapted block lemma2}. We have by Minkowski's inequality
	\begin{align*}
		\W_p\left( \pi^n \otimes k^n, \pi \otimes k \right) \le \W_p\left( \pi^n, \pi \right) + \Big( \int \W_p^p(k^a,k^{n,b}) \, K^n(da,db) \Big)^\frac{1}{p}.
	\end{align*}
	Using \eqref{eq:adapted block lemma kernel est1}, \eqref{eq:adapted block lemma kernel est2}, and Minkowski's inequality we  bound the last term of the right-hand side by
	\begin{align*}
		&\Big( \int \W_p^p(k^a,\tilde k^{n,b}) \, K^n(da,db)\Big)^\frac{1}{p} + \Big( \int \W_p^p(\tilde k^{n,b},k^{n,b}) \, \pi^n(db) \Big)^\frac{1}{p} \\
		&\qquad\le \W_p(\pi^n,\pi) +\Big( \iint \W_p(k^a, k^{a'}) \, K^{n,b}(da') \, K^n(da,db) \Big)^\frac{1}{p}.
	\end{align*}
	The last term vanishes by \cite[Lemma 2.7]{Ed19}, since $\int K^{n,b}(\cdot) \otimes K^{n,b}(\cdot) \, \pi^n(db) \in \cpl(\pi,\pi)$ and
	\begin{align*}
	&\Big( \iint d^p(a,a') \, K^{n,b}(da') \, K^{n}(da,db) \Big)^\frac{1}{p} \\
	&\qquad	\le \Big( \iint \big( d(a,b) + d(b,a')\big)^p \, K^{n,b}(da') \, K^n(da,db)\Big)^\frac{1}{p} \le 2 \W_p(\pi^n,\pi).\qedhere
	\end{align*}
\end{proof}

\begin{proposition}[Adapted block approximation] \label{prop:adapted block}
	Let $\fp X, \fp Y \in \CFP_p$, let $\pi \in \cplba(\fp X, \fp Y)$, and let $\varepsilon>0$.
	Then, for every $1\leq t\leq N$,  there are countable partitions $P_t$ of $\Z_t$ and families of measurable functions $(w^{A,B}_t)_{(A,B) \in P_t \times P_t}$ mapping from $\Z_{1:t-1} \times \Z_{1:t-1}$ to $[0,1]$ with
	\begin{align}
	\label{eq:box approximation2}
	\begin{split}
	&\sum_{A \in P_t} w^{A,B}_t(z_{1:t-1},\hat z_{1:t-1}) = \hat z_{t-1}^+(B),\\
	&\sum_{B \in P_t} w^{A,B}_t(z_{1:t-1},\hat z_{1:t-1}) = z_{t-1}^+(A)
	\end{split}
	\end{align}
	such that the coupling $\pi^\varepsilon \in \cplba(\fp X,\fp Y)$
	\begin{align}
		\label{eq:block approximation}
	\begin{split}
		\pi^\varepsilon(dz,d\hat z) 
		&:= \prod_{s = 1}^{N} \sum_{(A_s,B_s) \in P_s \times P_s} w^{A_s,B_s}_s(z_{1:s-1},\hat z_{1:s-1}) 
		 z_{s-1}^+(dz_s | A_s) \hat z_{s - 1}^+(d\hat z_s | B_s),
		\end{split}
	\end{align}
	satisfies $\W_p(\pi,\pi^\varepsilon) < \varepsilon$, where $z_0^+ := \Law(\ip_1(\overline{\fp X}))$ and $z_0^+ := \Law (\ip_1(\overline{\fp Y}))$.
\end{proposition}

	Here we used the notation $\nu(\,\cdot \, |A):=\nu(\,\cdot \, \cap A ) / \nu(A)$ if $\nu(A)>0$ with an arbitrary convention otherwise.
	
\begin{proof}
	By Lemma \ref{lem:bicausal cfp rep} a coupling of the form \eqref{eq:box approximation2} is bicausal between $\fp X$ and $\fp Y$.
	For  $1\leq t\leq N$, define
	\[ \pi_{1:t} := (\proj_{\Z_{1:t} \times \Z_{1:t}} )_\ast \pi.\]
	The proof uses an induction over $t$ and we claim the following:
	For given $t$ and for all $\varepsilon>0$ and $1\leq s\leq t$, there are partitions $P_s$ of $\Z_s$ and and mappings $w_s$ as in the statement of the proposition such that $\pi_{1:t}^\varepsilon$ satisfies $\W_p(\pi_{1:t},\pi_{1:t}^\varepsilon) < \varepsilon$.
	Here $\pi_{1:t}^\varepsilon$ is the adapted block approximation up to time $t$, i.e.\ defined as in \eqref{eq:block approximation} but with the product taken over $1\leq s\leq t$ instead of all $1\leq s\leq N$.

	We start with $t=1$.
	To that end, fix $\varepsilon$ and let $P_1$ be a countable partition of $\Z_t$ into measurable sets of diameter at most $\varepsilon$.
	Denote by 
	\[ w^{A,B}_1:= \pi_{1:1}(A\times B)\]
	for $(A,B) \in P_1 \times P_1$.
	Since the diameter of $A$ resp.\ $B$ is smaller than $\varepsilon$, we clearly have that $\W_p(\pi_{1:1},\pi_{1:1}^{\varepsilon})<2\varepsilon$,
	where $\pi_{1:1}^{\varepsilon}$ is the adapted block approximation up to time $t=1$.
	Further a straightforward calculation shows that $w_1$ satisfies  \eqref{eq:box approximation2}.

	Assuming that our induction claim is true for $1\leq t\leq N$, fix some $\varepsilon> 0$.
	By Lemma \ref{lem:adapted block} there is a measurable kernel 
	\[k_t \colon \Z_{1:t} \times \Z_{1:t} \to \Pc_p(\Z_{t+1} \times \Z_{t+1}) \text{ with }k_t^{z_{1:t},\hat{z}_{1:t}}\in\cpl(z_t^+,\hat{z}_t^+)\]
	and $\delta>0$ small enough such that 
	\[
		\W_p\left( \pi_{1:t+1},  \pi_{1:t}^{\delta} \otimes k_t \right) < \varepsilon.
	\]
	Now let $P_{t+1}$ be a countable partition of $\Z_{t+1}$ into measurable sets with diameter at most $\varepsilon$.
	For $(A,B) \in P_{t+1} \times P_{t+1}$ define 
	\[w^{A,B}_{t+1}(z_{1:t},\hat z_{1:t}) := k_t^{z_{1:t},\hat z_{1:t}}(A\times B)\]
	and set
	\begin{align*}
	\tilde k_t^{z_{1:t},\hat z_{1:t}} 
	:= \sum_{(A,B) \in P_{t+1} \times P_{t+1}} w^{A,B}_{t+1}(z_{1:t},\hat z_{1:t}) z_t^+\otimes \hat z_t^+(\,\cdot\,|A\times B).
	\end{align*}
	As the sets in $P_{t+1}$ have diameter at most $\varepsilon$, it follows that 
	\[\W_p(\tilde k_t^{z_{1:t},\hat z_{1:t}}, k_t^{z_{1:t},\hat z_{1:t}}) < 2\varepsilon\]
	for every $z_{1:t},\hat z_{1:t}\in\Z_{1:t}$.
	Further, recalling that $k_t^{z_{1:t},\hat{z}_{1:t}}\in\cpl(z_t^+,\hat{z}_t^+)$, it follows that $ w^{A,B}_{t+1}$ satisfies \eqref{eq:box approximation2}.
	Finally set 
	\[\pi_{1:t+1}^{\varepsilon} := \pi_{1:t}^{\delta} \otimes \tilde k_t .\]
	A straightforward calculation shows that $\W_p(\pi_{1:t+1},\pi_{1:t+1}^{\varepsilon}) < 3 \varepsilon$.
	It remains to note that $\pi_{1:t+1}^{\varepsilon}$ has the form as claimed in our induction statement, which completes the proof.
\end{proof}

\begin{theorem} \label{thm:adapted block}
	Let $\fp X, \fp Y \in \FP_p$, let $\overline{\fp X}, \overline{\fp Y}$ be their associated canonical processes, and let $\pi \in \cplba(\overline{\fp X}, \overline{\fp Y})$.
	Then, for every $\varepsilon > 0$, there is $\Pi^\varepsilon \in \cplba(\fp X, \fp Y)$ such that
	\begin{align}
		\W_p\left( \pi, (\ip(\fp X),\ip(\fp Y))_\ast \Pi^\varepsilon \right) < \varepsilon.
	\end{align}
\end{theorem}

\begin{proof}
	Let $\pi^\varepsilon$ be the adapted block approximation of $\pi$ given in  Proposition \ref{prop:adapted block}.
	Then we can express $\pi^\varepsilon$ as 
	\begin{align*}
		\pi^\varepsilon(dz,d\hat z) 
		&=  \prod_{t = 1}^N  \sum_{(A_t,B_t) \in P_t \times P_t} 1_{A_t \times B_t}(z_t,\hat z_t) \frac{w_{t}^{A_t,B_t}(z_{1:t-1},\hat z_{1:t-1})}{z_{t-1}^+(A_t) \hat z_{t-1}(B_t)}  \big( \mu\otimes\nu \big)(dz,d\hat z).
	\end{align*}
	With this representation of $\pi^\varepsilon$ in mind, by slight abuse of notation write $1_A=1_A(\ip_t(\fp X))$ for $A\in P_t$ and similarly for $\fp Y$, and define 
		\[ D_t:= \sum_{(A,B)\in P_t\times P_t} 1_{A\times B} \frac{ w^{A,B}_t(\ip_{1:t-1}(\fp X), \ip_{1:t-1}(\fp Y))}{\ip_{t-1}^+(\fp X)(A) \ip_{t-1}^+(\fp Y)(B)} \]
	for every $1\leq t\leq N$.
	Now let $\Pi^\varepsilon$ be the measure absolutely continuous w.r.t.\ $\P^\fp X \otimes \P^\fp Y$ with density	
	\[
		\frac{d\Pi^\varepsilon }{d\P^\fp X \otimes \P^\fp Y}
		:= \prod_{t = 1}^N D_t.
	\]
	In particular $\pi^\varepsilon=(\ip(\fp X), \ip(\fp Y))_\ast \Pi^\varepsilon$.
	Before proving the theorem, we interject the following claim.

	\emph{Auxiliary claim:} For every $1 \le t \le N-1$ and every $U$ that is $\F^\fp X_N$-measurable and bounded, we have that
	\[
	 \E_{\P^\fp X \otimes \P^\fp Y} \left[ D_{t+1} U  \middle| \F_{t,t}^{\fp X,\fp Y}  \right]	
	 = \E_{\P^\fp X} \left[ U \middle| \F_{t}^\fp X \right].
	\]
	To see that this claim is true, note that by definition of $D_{t+1}$ the left-hand side equals
	\begin{align*}
	& \sum_{(A,B) \in P_{t+1} \times P_{t+1}} \frac{ w^{A,B}_{t+1}(\ip_{1:t}(\fp X), \ip_{1:t}(\fp Y))}{\ip_{t}^+(\fp X)(A) \ip_{t}^+(\fp Y)(B)} \E_{\P^\fp X} \left[ U 1_{A } | \F_{t}^\fp X\right] \P( \ip_{t+1}(\fp Y) \in B|\F_{t}^\fp Y)
	\\	
	&\qquad= \sum_{A \in P_{t+1}} \E_{\P^\fp X} \left[ U 1_{A} | \F_{t}^\fp X\right] \sum_{B \in P_{t+1}} \frac{w_{t+1}^{A,B}(\ip_{1:t}(\fp X), \ip_{1:t}(\fp Y))}{\ip_{t}^+(\fp X)(A)}.
	\end{align*}
	The property  \eqref{eq:box approximation2} of Proposition \ref{prop:adapted block} of $w$ implies that the sum over $B\in P_{t+1}$ equals 1.
	Further, as $P_{t+1}$ is a partition, our claim follows.
	
	We are now ready to prove the theorem.
	In the first step note that $\Pi^\varepsilon$ is indeed a coupling of $\fp X$ and $\fp Y$, which follows from property \eqref{eq:box approximation2} of $w$.
	Further, for symmetry reasons, it suffices to show that $\Pi^\varepsilon$ is causal from $\fp X$ to $\fp Y$.
	By Lemma \ref{lem:causal CI}, the latter holds true if
	\[
		\E_{\Pi^\varepsilon} \left[ U V \right] 
		= \E_{\Pi^\varepsilon} \left[ V \E_{\P^\fp X} \left[ U | \F_t^\fp X \right] \right]
	\]
	for every $1 \le t \le N-1$ and every $U$ and $V$ that are bounded and $\F_N^\fp X$ and $\F_{t,t}^{\fp X,\fp Y}$-measurable, respectively.
	Now fix such $t$, $U$, and $V$.	
	
The definition of $\Pi^\varepsilon$, and iteratively applying the tower property and our auxiliary claim imply that 
	\begin{align*}
		\E_{\Pi^\varepsilon} [ U V ] 
		&= \E_{\P^\fp X \otimes \P^\fp Y} \left[  \prod_{s = 1}^{N-1} D_s  D_N U V  \right] \\
		&= \E_{\P^\fp X \otimes \P^\fp Y} \left[ \prod_{s = 1}^{N-1} D_s  V \E_{\P^\fp X \otimes \P^\fp Y} \left[ D_N U \middle| \F_{N-1,N-1}^{\fp X,\fp Y}  \right]\right] \\
		&= \E_{\P^\fp X \otimes \P^\fp Y} \left[ \prod_{s = 1}^{N-1} D_s  V \E_{\P^\fp X } \left[ U \middle| \F_{N-1}^{\fp X} \right] \right]
		=\ldots = \E_{\P^\fp X \otimes \P^\fp Y} \left[ \prod_{s=1}^t D_s V  \E_{\P^\fp X}\left[ U \middle| \F_{t}^\fp X \right] \right].
	\end{align*}
	Finally, note that the auxiliary claim (applied with $U=1$) also shows that, for every $1\leq s\leq N-1$, we have $\E_{\P^\fp X \otimes \P^\fp Y} \left[  D_{s+1} | \F_{s,s}^{\fp X,\fp Y} \right] = 1$.
	Hence, another application of the tower property gives 
	\[
		\E_{\Pi^\varepsilon} [ U V ] 
		=\E_{\P^\fp X \otimes \P^\fp Y} \left[ \prod_{s=1}^t D_s V  \E_{\P^\fp X}[ U | \F_{t}^\fp X ] \right]
		=\E_{\Pi^\varepsilon} [  V  \E_{\P^\fp X}[ U | \F_{t}^\fp X ] ].
	\]
	This concludes the proof.
\end{proof}

\section{Infinite discrete time}
\label{sec:infinite.time}

The main focus of this article is on stochastic processes in a finite discrete time framework. In this section, we consider the class of stochastic processes with discrete but infinite time horizon. In fact, many of our results  carry over to this  instance, by simple limit arguments. 

We consider the class $\FP^{(\infty)}$ of  processes of the form
$$\fp X = (\Omega, \F, \P, (\F_t)_{t=1}^\infty,  (X_t)_{t=1}^\infty),$$
where $X=(X_t)_{t=1}^\infty $ takes values\footnote{As above, the results and arguments of this section are  valid in the case of an arbitrary Polish state space.} in $\R^{\infty}$. To turn $\R^{\infty}$ into a Polish space, we equip it with the product topology and, more specifically, with the distance 
\[ d(x,y)=d_p(x,y):= \Big(\sum_{t= 1}^\infty \frac{1}{2^t} (|x_t-y_t|^p \wedge 1) \Big)^\frac{1}{p}.\]
In complete analogy with Definition \ref{def:causal} above we can then consider bicausality as well as   $\AWA_p^{(\infty)}$ on $\mathcal P_p (\R^{\infty})$ and 
$\FP^{(\infty)}_p$. 
As above we write $\FFP^{(\infty)}_p$ for the class obtained after identifying $\fp X, \fp Y$ with $\AWA_p(\fp X, \fp Y)=0$. 
Note that  the metric $d$ is bounded, consequently $\FP^{(\infty)}_p$ / $\FFP^{(\infty)}_p$ do not depend on the choice of $p$ and  $\mathcal P_p(\R^{\infty})$ carries the usual weak topology.
In analogy to the results stated in the introduction we then obtain: 
\begin{theorem}\label{MainTheoremInfiniteCase}
The following hold.
\begin{enumerate}[label=(\roman*)]
\item \label{it:mtinf.1} $\AWA_p^{(\infty)}$ is a metric on $\FFP^{(\infty)}_p$.
\item \label{it:mtinf.2} $(\FFP^{(\infty)}_p, \AWA^{(\infty)}_p)$ is the completion of $(\mathcal P_p(\R^\infty), \AWA_p^{(\infty)})$, where we identify laws in $\mathcal P_p(\R^\infty)$ with corresponding plain processes.
\item \label{it:mtinf.3} The set
	\[
		\left\{ \fp X \in \FFP^{(\infty)}_p \colon \fp X \text{ is Markov and has a representative on a finite probability space} \right\}
	\]
	is dense in $\FFP^{(\infty)}_p$.
\item \label{it:mtinf.4} If a sequence $(\fp X^n)_n \subset \FFP^{(\infty)}_p$ of martingales converges to $\fp X\in \FFP_p^{(\infty)}$ w.r.t.\ $\AW_p^{(\infty)}$ and $\{ X^n: n\in\N\}$ is uniformly integrable, then $\fp X$ is a martingale.
\item \label{it:mtinf.5} If a sequence $(\fp X^n)_n \subset \FFP^{(\infty)}_p$  converges to $\fp X\in \FFP_p^{(\infty)}$ w.r.t.\ $\AW_p^{(\infty)}$ and $\{ X^n: n\in\N\}$ is uniformly integrable, then the Doob-decomposition $\fp D^{\fp X^n}$ of $\fp X^n$ converge to the Doob-decomposition $\fp D^{\fp X}$ of $\fp X$ w.r.t.\ $\AW_p$.
\item \label{it:mtinf.6} If $G_t\colon\R^t\to \R$ is bounded,  continuous for each $t\in \N$, and  $\lim_{t\to \infty} \| G_t\|_\infty =0 $ then
\begin{align}
\label{OptStopN}
\sup\{ \E[ G_\tau (X_1, \ldots, X_\tau)] :  \tau \text{ is finite stopping time}\}
\end{align} 
 is continuous  in $\fp{X}$. 
\item \label{it:mtinf.7} (`Prohorov') A set $K\subseteq \FFP^{(\infty)}_p$ is precompact if and only if the respective set of laws in  $\mathcal P_p(\mathbb{R}^\infty)$ is precompact.
\end{enumerate}
\end{theorem}


To establish Theorem \ref{MainTheoremInfiniteCase} we need some  notations. 
We write $\FP^{(N)}_p $ for the set of $N$-step filtered processes,  $\AW_p^{(N)}$ for the adapted Wasserstein distance w.r.t.\ $d^{(N)}$ given by
$$  d^{(N)}(x,y)=d^{(N)}_p(x,y) = \Big( \sum_{t= 1}^N \frac{1}{2^t} (|x_t-y_t|^p \wedge 1) \Big)^{\frac{1}{p}},$$
and we write $\FFP^{(N)}_p$ for the space obtained after identifying equivalent processes. 
For $N\in \N$ we consider the function
$ r_N\colon \bigcup_{M\in\{N,N+1,\ldots,\infty\}} \FP^{(M)}_p\to  \FP^{(N)}_p$ given by
\[r_N \big( (\Omega, \F, \P, (\F_t)_{t=1}^M,  (X_t)_{t=1}^M)\big) 
= (\Omega, \F, \P, (\F_t)_{t=1}^N,  (X_t)_{t=1}^N).\]

The next two lemmas will allow us to derive Theorem  \ref{MainTheoremInfiniteCase} from the respective results in the finite time horizon case.

\begin{lemma}\label{SimilarMetric}
	For every $\fp X, \fp Y\in\FP_p^{(\infty)}$ and every $N\in\N$, we have
	\[ \AW_p^{(N)} (r_N(\fp X), r_N(\fp Y) )  
	\leq  \AW_p^{(\infty)} (\fp X, \fp Y ) 
	\leq   \AW_p^{(N)} (r_N(\fp X), r_N(\fp Y) ) + \frac 1 {2^N}.
	\]
	In particular $\AW_p^{(\infty)} (\fp X, \fp Y ) = \lim_{N\to \infty}  \AW_p^{(N)} (r_N(\fp X), r_N(\fp Y) ) $ for $\fp X, \fp Y \in \FP^{(\infty)}_p$.
\end{lemma}
\begin{proof} 
	The first inequality is trivial. 
	Similar to Theorem \ref{thm:isometry} the other inequality has a short proof under the additional assumption that  $\fp X, \fp Y$ are supported by Polish probability spaces: 
	Indeed, let $\pi \in \cplba(r_N(\fp X), r_N(\fp Y))$ and note that 
	\[\bar\pi :=(\ip(r_N(\fp X), \ip(r_N(\fp Y)))_\ast \pi \in \cplba(\overline{r_N(\fp X)}, \overline{r_N(\fp Y)}).\]
We then consider 
\[
		\gamma := \left( \id, \ip(r_N (\fp X)) \right)_\ast \P^\fp X 
		\quad\text{and}\quad
		 \hat{\gamma} := \left( \id, \ip(r_N(\fp Y)) \right)_\ast \P^\fp Y.\]
	As in the proof of Theorem \ref{thm:isometry},  these couplings  admit disintegrations $(\gamma_{z})_{z}$ and $(\hat{\gamma}_{\hat{z}})_{\hat{z}}$ w.r.t.\ the second variable  (since the considered probability spaces are Polish by assumption) and we may consider $		\Pi(d\omega,d\hat{\omega}) := \int \gamma_z(d\omega) \hat{\gamma}_{\hat{z}}(d\hat{\omega}) \, \overline{\pi}(dz,d\hat{z}).$ 
Arguing similar as in the proof of Theorem \ref{thm:isometry} we then obtain that $\Pi \in \cplbc(\fp X, \fp Y)$. 
Since $\pi$ was arbitrary, it is straightforward to verify that $\AW_p^{(\infty)} (\fp X, \fp Y ) 
	\leq   \AW_p^{(N)} (r_N(\fp X), r_N(\fp Y) ) + \frac 1 {2^N}$.

We now drop the assumption that the underlying probability spaces are Polish. 
Let\linebreak 	$\pi \in \cplbc(\overline{r_N(\fp X)}, \overline{r_N(\fp Y)})$. 
We first claim that for every $\varepsilon > 0$, there exists $\Pi^\varepsilon \in \cplbc(\fp X, \fp Y)$ such that
	\begin{equation}
		\label{eq:InfiniteBlockApprox}
		\W_p(\pi, (\ip(r_N(\fp X)), \ip(r_N(\fp Y)))_\ast \Pi^\epsilon) < \varepsilon.	
	\end{equation}
	Indeed, let $\Pi^\varepsilon \in \cplbc(r_N(\fp X), r_N(\fp Y))$ be the coupling constructed in (the proof of) Theorem \ref{thm:adapted block}; hence
	$
		\frac{d\Pi^\varepsilon}{d \P^\fp X \otimes \P^\fp Y} = \prod_{t = 1}^N D_t
	$
	for some $(\F^{\fp X,\fp Y}_{t,t})_{t = 1}^N$-adapted process $D = (D_t)_{t = 1}^N$.
	To see that $\Pi^\varepsilon \in \cplbc(\fp X,\fp Y)$, let $V$ be bounded and $\F^\fp X$-measurable.
	Recall the `chain rule' for conditional expectations:  if $\alpha$ and $\beta$ are two probability measures such that $\frac{d\alpha}{d\beta}=Z$, then $\E_\alpha[\, \cdot\,|\mathcal{H}]=\E_\beta[Z \cdot | \mathcal{H}]/ \E_\beta[Z|\mathcal{H}]$.
	Hence, 
	\begin{align*}
		\E_{\Pi^\varepsilon}[V | \F_{t,t}^{\fp X,\fp Y}] 
		&= \E_{\P^\fp X \otimes \P^\fp Y} \left[ \prod_{s = t+1}^N D_s  V | \F^{\fp X, \fp Y}_{t,t} \right]
		= 
		\E_{\P^\fp X \otimes \P^\fp Y} \left[ \prod_{s = t+1}^N D_s \E_{\P^\fp X \otimes \P^\fp Y}  \left[ V | \F_{N,N}^{\fp X, \fp Y} \right] | \F^{\fp X, \fp Y}_{t,t} \right]
	\end{align*}
	where the second inequality follows from the tower property.
	Moreover, $\E_{\P^\fp X \otimes \P^\fp Y}  [ V | \F_{N,N}^{\fp X, \fp Y} ]= \E_{\P^\fp X} [ V | \F_{N}^{\fp X} ]$ because $\P^\fp X \otimes \P^\fp Y \in \cplbc(\fp X, \fp Y)$; thus by the chain rule   and because $\Pi^\varepsilon \in \cplbc(r_N(\fp X), r_N(\fp Y))$,
	\[\E_{\P^\fp X \otimes \P^\fp Y} \left[ \prod_{s = 1}^N D_s  \E_{\P^\fp X}[ V | \F_{N}^\fp X ]| \F^{\fp X, \fp Y}_{t,t} \right] \\
		= \E_{\Pi^\varepsilon}[ \E_{\P^\fp X}[ V | \F_{N}^\fp X ]| \F_{t,t}^{\fp X,\fp Y}] 
		= \E_{\P^\fp X}[V | \F^\fp X_t]. \]
	Hence, $\Pi^\varepsilon$ is causal and thus, by symmetry, bicausal.

It follows that 
	\[
		\E_\pi[d_p^{(N)}(X,Y)^p]^\frac1p + \frac{1}{2^N} + \varepsilon \ge \E_{\Pi^\varepsilon}[d_p^{(N)}(X,Y)^p]^\frac1p + \frac{1}{2^N} \ge \E_{\Pi^\varepsilon}[d_p(X,Y)^p]^\frac1p,
	\]
	where we used \eqref{eq:InfiniteBlockApprox} for the first inequality.
	As $\varepsilon > 0$ and $\pi \in \cplbc(r_N(\fp X), r_N(\fp Y))$ were fixed but arbitrary, this concludes the proof.
\end{proof}

The following lemma is a version of the Kolmogorov extension theorem. \
\begin{lemma}\label{KolmogorovConsistency}
Let $\fp X ^{(N)}\in \FP^{(N)}_p, N\in \N$ be a sequence of filtered processes such that $r_N(\fp X ^{(N+1)})= \fp X ^{(N)}, N\in \N$.  Then there exists a stochastic process $\fp X\in \FP^{(\infty)}_p$ defined on a Polish probability space such that 
\[ \AW_p^{(N)}(r_N(\fp X), \fp X ^{(N)}) =0 \text{ for all } N\in \N.
\]
\end{lemma}
\begin{proof}
	For every $N \in \N$ and $1 \le k \le N$, we denote by $I^{k,N}$ the information processes of $r_k(\fp X^{(N)})$, that is, $(I^{k,N}_t)_{t=1}^k := (\ip_t( r_k(\fp X^{(N)}))_{t=1}^k$.
	Hence, by assumption, for every $N \in \N$ and $M\geq N$,
	\begin{equation}
		\label{eq:KolmogorovConsistency1}
	\Law\big( ( I_t^{k,N})_{1 \le t \le k \le N} \big) 	
	=\Law\big( ( I_t^{k,M} )_{ 1 \le t \le k \le N } \big).
	\end{equation}
	By the Kolmogorov extension theorem there exists a Polish probability space $(\Omega,\F,\P)$ supporting random variables $\{ J^k_t : k \in \N, 1 \le t \le k \}$ such that for all $N\geq k$,
	\begin{equation}
		\label{eq:KolmogorovConsistency}
		\Law\big( (I^{k,N}_t )_{1 \le t \le k \le N } \big)
		=\Law \big( (J^k_t )_{1 \le t \le k \le N} \big).
	\end{equation}
	Note that  (for $t<k$) $J^k_t$ has two components, $J^k_t=( J^{k,-}_t, J^{k,+}_t)$, and define the filtered process $\fp X$ as the tuplet
	\[
		\fp X := \left( \Omega,\F,\P, (\F_t)_{t=1}^\infty, X := (J_t^{t, -})_{t=1}^\infty \right)
	\] 
	where $\F_t := \sigma ( J^k_s : k \in \N, 1 \le s \le t\wedge k  )$.
	Note that $J^{k,-}_t = J^{t,-}_t$ $\P$-a.s.\ (for $t\leq k$).

	We claim that for every $1\leq t\leq N$ we have $\P$-a.s. 
	\[
		\ip_t(r_N(\fp X)) = J_t^N,
	\]
	which would yield the assertion of the lemma.

	We proceed to show this claim by backward induction:
	Indeed, when $t = N$ we have that $J_N^N = X_N = \ip_N(r_N(\fp X))$.
	Next, assume the claim to be true for $t+1$.
	Then we have $\P$-a.s.
	\begin{align*}
		\ip_t^-(r_N(\fp X)) = J_t^{N,-} \text{ and } \ip_t^+(r_N(\fp X)) = \Law_\P\left( J^N_{t + 1} | \F_t \right)
	\end{align*}
	So, it remains to show that $\Law_\P\left( J^N_{t + 1} | \F_t \right)$ coincides with $J_t^{N,+}$.
	To this end, define the $\sigma$-algebras
	\[
		\tilde \F^{N}_t := \sigma \left( I_s^{k,N} : 1 \le k \le N, 1 \le s \le k \wedge t \right).
	\]
	By Lemma \ref{lem:ip.self.aware}, for $t+1\leq k \leq N$, we have $\P^{\fp X^{(N)}}$-a.s.
	\begin{equation}
		\label{eq:KolmogorovConsistency3}
		I^{k,N,+}_t = \Law_{\P^{\fp X^{(N)}}} \left( I^{k,N}_{t + 1} | \F_t^{\fp X^{(N)}}\right) = \Law_{\P^{\fp X^{(N)}}} \left( I^{k,N}_{t + 1}  | \tilde \F_t^N \right).
	\end{equation}	
	Setting $\F^N_t=\sigma(J^k_s : 1 \le k \le N, 1\leq s\leq k \wedge t)$, we obtain from \eqref{eq:KolmogorovConsistency}  that
	\begin{equation}
		\label{eq:KolmogorovConsistency4}
		\Law\left(J_t^{k,+}, \Law_\P\left( J_{t + 1}^k | \F_t^N \right)\right)
		=\Law	\left(I^{k,N,+}_t, \Law_{\P^{\fp X^{(N)}}} \left( I^{k,N}_{t + 1}  | \tilde \F_t^N \right) \right).
	\end{equation}
	Let $M \ge N$ and set $k = N$.
	Then \eqref{eq:KolmogorovConsistency3} and \eqref{eq:KolmogorovConsistency4} yield $\P$-a.s.
	\begin{equation}
		\label{eq:KolmogorovConsistency5}
		\Law_\P\left( J^N_{t + 1} | \F_t^N \right) = J_t^{N,+} = \Law\left( J^N_{t + 1} | \F_t^M \right).
	\end{equation}
	As $M \ge N$ was arbitrary and $(\F_t^M)_{M=1}^\infty$ is increasing and $\lim_{M \to \infty} \F_t^M$ generates $\F_t$, we find that \eqref{eq:KolmogorovConsistency5} also holds true when replacing $\F^M_t$ by $\F_t$. This concludes the proof.
\end{proof}

\begin{proof}[Proof of Theorem \ref{MainTheoremInfiniteCase}]
\ref{it:mtinf.1}: This follows from Lemma \ref{SimilarMetric} and the fact that  $\AW_p^{(N)} $ is  a metric for each $N\in \N$.

\ref{it:mtinf.2}: To verify completeness, consider a Cauchy sequence $\fp X^m\in \FFP^{(\infty)}_p$, $m\in \N$. 
Passing to subsequences and using a diagonalization argument,  there exist an increasing sequence $(m_k)_{k\geq 1}$ and filtered processes $\fp Y^{(N)}\in \FFP^{(N)}_p$ such that for each $N\in \N$
$$r_N(\fp X^{m_k}) \to \fp Y^{(N)}.$$ 
By the consistency result (i.e.\ Lemma \ref{KolmogorovConsistency}) there is a process $\fp Y\in \FFP^{(\infty)}_p $ such that $\fp Y^{(N)}=r_N(\fp Y), N\in \N$ and then $\fp  X^{m_k} \to \fp Y$.

To see denseness of $\mathcal P_p(\R^\infty)$, we can (for instance) note that 
\begin{align}\label{InitialParts}
\bigcup_{N\in \N} \left\{
(\Omega, \F, \P, (\F_t)_{t=1}^\infty,  (X_t)_{t=1}^\infty) :
\begin{array}{l}
 (\Omega, \F, \P, (\F_t)_{t=1}^N,  (X_t)_{t=1}^N)\in \FFP^{(N)}_p,\\
  X_{s}=X_N, \F_s=\F_N, s\geq N
\end{array} 
\right\}
\end{align}
is dense; hence the claim follows from Theorem \ref{thm:dense}. 

\ref{it:mtinf.3}: This follows again from the denseness of the set in \eqref{InitialParts} together with Theorem \ref{thm:dense}.

\ref{it:mtinf.4} and \ref{it:mtinf.5}: First note that a function $\Psi\colon \FFP_p^{(\infty)}\to \FFP_p^{(\infty)} $ is continuous if and only if for every $N$, $\Psi\circ r_N$ is continuous.
Hence, \ref{it:mtinf.4} follows from Proposition \ref{lem:stopping.lipschitz} and \ref{it:mtinf.5} follows exactly as in the proof of Proposition \ref{prop:doob}.

\ref{it:mtinf.6}: This is a straightforward consequence of Lemma \ref{SimilarMetric}  and Theorem \ref{thm:OptStopIntro}

\ref{it:mtinf.7}: A set $K\subseteq \FFP^{(\infty)}_p$ is precompact if and only if all the sets $r_N[K], N\geq 1 $ are precompact; and a set $\tilde K \subseteq \mathcal{P}_p(\R^\infty)$ is precompact if and only if all the sets $r_N[\tilde K], N\geq 1 $ are precompact (where we interpret $r_N$ as a function on $\mathcal{P}_p(\R^\infty)$ by identifying again laws and processes). 
The result thus follows from Theorem \ref{CompactPre}.
\end{proof}

\section{Comments on the continuous time case}
\label{sec:cont.time}

Depending on the context and intended application, different authors have considered different `adapted' notions of equivalence and similarity  for stochastic processes. 
As we  discuss  below, these notions can be rephrased using ideas from adapted transport as considered above.
Therefore, we believe that the framework and results developed in the present paper provide a blueprint for the respective  theories in the continuous time case. 
 We  describe some natural `adapted' topologies / distances from the perspective of the present paper:

On the one end of the spectrum, the direct continuous time extension of the distance considered in the current paper is 
\[ \mathcal{AW}_p(\fp X,\fp Y) 
=\inf_{\pi\in\cplbc(\fp X,\fp Y)} \E_\pi[ d^p(X,Y)]^{\frac{1}{p}} \]
where $d$ is a distance on the paths. This or very closely related notions are considered (on path spaces) for instance in \cite{AcBaZa20,BaBaBeEd19a, BaBeHuKa20} and 
 `almost all' probabilistic operations (such as stochastic integration) are continuous w.r.t.\ this topology, see, e.g,\ \cite{BaBaBeEd19a}. We also note that results of the present paper concerning e.g.\ the representation of the completion based on filtrations as well as geodesic properties of the distance will carry over to this setting.
A disadvantage,  is that this topology is not separable. 
Moreover, there are not too many relatively compact sets and  scaled random walks do not converge to Brownian motion (no matter which distance $d$ one chooses on the paths), cf.\ the argument in \cite[p240]{Ho91}. 
This means that $\AW_p$ is well suited e.g.\ for analyzing the  sensitivity  of stochastic optimization problems w.r.t.\ the input process, but less suited to analyze the transition from discrete to continuous time, or to establish the existence of certain `extremal' processes. 

Bion-Nadal and Talay  \cite{BiTa19} single out a specific instance of $\AW_2$, where $d$ is induced by the $L^2$-distance and provide a Hamilton-Jacobi-Bellman equation for the calculation of  $\AW_2$. In particular, they obtain a numerically tractable version of an adapted Wasserstein distance. This line of research is continued by Backhoff-K\"allblad-Robinson \cite{BaKaRo22}.

A further variation of the distance $\AW_p$ for semi-martingales absolutely continuous w.r.t.\ to Wiener measure and with Cameron-Martin cost is considered by Lasalle \cite{La18} who uses it to provide a new interpretations  of transport-information inequalities  on the Wiener space.  F\"ollmer \cite{Fo22b, Fo22a} investigates yet another adapted Wasserstein distance for semi-martingales with drift dominated by quadratic variation and extends  Talagrand's inequality to measures which are not absolutely continuous w.r.t.\ Wiener measure.

On another end of the spectrum
 of weak adapted topologies lie the contributions initiated by Aldous  \cite{Al81} and Hoover-Keisler \cite{Ho91, HoKe84}.
Historically the first extension of the weak topology to take information into account was the extended weak topology as defined by 
Aldous \cite{Al81}. 
The idea is to identify the law of a process with the law of its prediction process and to then measure the distance of stochastic processes through the distance of the respective prediction processes. Hoover and Keisler  \cite{Ho91, HoKe84} build on this idea and construct an infinitely iterated prediction process which captures further properties of the underlying filtration of the stochastic process. In view of Theorem \ref{thm:stopping.rank}, this infinite iteration  is already  necessary if the goal is to capture the information required for optimal stopping problems. It would be possible to metrize the Hoover-Keisler topology through a Wasserstein-type  distance  if one modifies the bicausality condition (in spirit similar to the $J_1$ metric on c\`adl\`ag paths):
\[\inf_{\pi\in\cpl(\fp X,\fp Y)} \left( \E_\pi[ d^p(X,Y)]^{\frac{1}{p}} +\begin{array}{l}
\text{penalization how far $\pi$}\\
\text{is from beeing bicausal}
\end{array} \right). \]
Although the continuity of probabilistic operations w.r.t.\ this distance becomes more subtle, this distance has important advantages:
First, as already established by Hoover \cite{Ho91}, many discrete-time objects such as scaled random walks converge in this topology to their continuous time limit. 
Further, just as in our paper,  tightness of the laws of the processes guarantees relative compactness and thus there are many relatively compact sets.

We believe that the theory on adapted optimal transport developed in this paper forms the foundation for analyzing the continuous time theories from the (geometric and metric) perspective of optimal transport.
Although we expect it can provide significant new insights, exploring the details is beyond the scope of this paper, and we defer this aspect to future research.

\end{document}